\documentclass[12pt]{article}

\usepackage{amsthm,amsmath,stmaryrd,bbm,hyperref,geometry,color}
\usepackage{amssymb}
\usepackage[english]{babel}
 \usepackage[utf8]{inputenc}
\usepackage{graphicx}
\usepackage{verbatim}
\usepackage{enumitem}

\newcommand{\po}{\left(}
\newcommand{\pf}{\right)}
\newcommand{\co}{\left[}
\newcommand{\cf}{\right]}
\newcommand{\cco}{\llbracket}
\newcommand{\ccf}{\rrbracket}
\newcommand{\R}{\mathbb R}

\newcommand{\N}{\mathbb N} 
\newcommand{\Ent}{\mathrm{Ent}} 
\newcommand{\bG}{\mathbf{G}} 
\newcommand{\bW}{\mathbf{W}} 
\newcommand{\D}{\mathcal  D} 
\newcommand{\Q}{\mathcal  Q} 
\renewcommand{\P}{\mathcal  P} 
\renewcommand{\H}{\mathcal  H}

\newcommand{\dd}{\text{d}}
\newcommand{\na}{\nabla}
\newcommand{\1}{\mathbbm{1}}

\newcommand{\new}[1]{#1}

\newtheorem{theorem}{Theorem}
\newtheorem{proposition}{Proposition}
\newtheorem{remark}{Remark}
\newtheorem{lemma}{Lemma}
\newtheorem{corollary}{Corollary}
\newtheorem{assumption}{Assumption}

\title{An entropic approach for Hamiltonian  Monte Carlo: the idealized case}
\author{Pierre Monmarché}

\begin{document}
\maketitle

\begin{abstract}
 Quantitative long-time entropic convergence and short-time regularization are established for an idealized Hamiltonian Monte Carlo chain which alternatively follows an Hamiltonian dynamics for a fixed time and then partially or totally refresh\new{es} its velocity with an auto-regressive Gaussian step. These results, in discrete time, are the analogous of similar results for the continuous-time kinetic Langevin diffusion, and the latter can be obtained from our bounds in a suitable limit regime. The dependency in the log-Sobolev constant of the target measure is sharp and is illustrated  on a mean-field case and on a low-temperature regime, with an application to the simulated annealing algorithm. The practical unadjusted algorithm is briefly discussed.
\end{abstract}

\section{Introduction}

\subsection{Overview}

Let $\pi$ be a probability measure on $\R^d$ with density proportional to $e^{-U}$ for some $U\in\mathcal C^2(\R^d)$. The main subject of this work is the Markov chain $(X_k,V_k)_{k\in\N}$ on $\R^d \times \R^d$ whose transitions are given by the alternance of two steps: first, the chain follows for a given fixed time $t>0$ the Hamiltonian dynamics associated to the potential $U$, namely
\begin{equation}\label{eq:HD2}
(X_k',V_k') = \Phi_t \po X_k,V_k\pf 
\end{equation}
where $\R_+ \ni s\mapsto \Phi_s(x,v)=:(x_s,v_s)\in\R^{2d}$ is the solution of
\begin{equation}\label{eq:HD}
\dot x_s = v_s\qquad \dot v_s = -\na U(x_s)\qquad (x_0,v_0)=(x,v)\,.
\end{equation}
Second, an auto-regressive Gaussian \new{randomization} of the velocity is performed, namely
\begin{equation}\label{eq:damping}
X_{k+1} = X_k' \qquad V_{k+1} = \eta V_k' + \sqrt{1-\eta^2} G_k
\end{equation}
where $\eta\in[0,1)$ is a given damping parameter and $(G_k)_{k\in\N}$ is an i.i.d. sequence of standard (mean $0$, variance $I_d$) $d$-dimensional Gaussian variables. Let $\mu = \pi\otimes \mathcal N(0,I_d)$  where $\mathcal N(0,I_d)$ stands for the standard Gaussian distribution on $\R^d$. It is readily checked that $\mu$ is invariant for both the \new{randomization} and Hamiltonian steps.

  We call $(X_k,V_k)_{k\in\N}$ the idealized Hamiltonian Monte Carlo (HMC) chain.  HMC is a widely used algorithm for sampling the target distribution $\pi$. Here, \emph{idealized} refers to the fact the true Hamiltonian dynamics is performed in contrast to the practical use of HMC where it is replaced by a numerical propagator (possibly corrected by a Metropolis step). Alternatively, the Hamiltonian dynamics also appears as the limit in some regimes (high dimension \cite{Doucet} or high frequency \cite{MRZ}) of some piecewise deterministic continuous-time samplers. In this work we focus on the idealized chain (also called the \emph{exact} HMC in \cite{BouRabeeEberle}) and postpone the study of other cases to future works\footnote{\new{See \cite{Toappear}   for the unadjusted case.}}.

In the most standard case, $\eta=0$, so that $(X_k)_{k\in\N}$ alone is a Markov chain. Alternatively, if $\eta= e^{-\gamma t}$ for some fixed $\gamma$ then, as $t\rightarrow 0$,  $(X_{\lfloor s/t\rfloor},V_{\lfloor s/t\rfloor})_{s\geqslant 0}$ converges \new{pathwise} to the Langevin diffusion, which is the solution of
\begin{equation}\label{eq:kinLangevin}
\left\{\begin{array}{rcl}
\dd X_t & = & V_t \dd t \\
\dd V_t & = & -\na U(X_t)\dd t - \gamma V_t \dd t + \sqrt{2\gamma }\dd B_t\,,
\end{array}\right.
\end{equation}
where $(B_t)_{t\geqslant 0}$ is a standard $d$-dimensional Brownian motion. In other word\new{s}, for $t\ll 1$, the HMC chain can be seen as an idealized splitting scheme of the Langevin diffusion \cite{MonmarcheSplitting}. If, in contrast, $\eta=0$ then, as $t\rightarrow 0$,  \new{$X_{k+1}\simeq X_k - (t^2/2)\na U(X_k) + t G_{k-1}$ } and then $(X_{\lfloor \new{2} s/t^2\rfloor})_{s\geqslant 0}$ converges \new{pathwise} to the overdamped Langevin diffusion\footnote{In this work we use the terminology of statistical physics and molecular dynamics, namely without any specification, \emph{Langevin diffusion} refers to the kinetic/underdamped process \eqref{eq:kinLangevin}. In contrast, in Bayesian statistics, this term often refers to the overdamped process \eqref{eq:overdLangevin}. }
\begin{equation}\label{eq:overdLangevin}
\dd X_t = -\na U(X_t) \dd t + \sqrt{2}\dd B_t\,.
\end{equation}
Finally, if, at each step, the \new{randomization} step \eqref{eq:damping} is only performed with a probability $p\in(0,1)$ then, taking $p=1-e^{-\lambda t}$ for some fixed $\lambda>0$ and letting $t$ vanish, $(X_{\lfloor s/t\rfloor},V_{\lfloor s/t\rfloor})_{s\geqslant 0}$  converges to the continuous-time Randomized HMC process, which follows \eqref{eq:HD} and, at the jump times of a Poisson process of intensity $\lambda$, undergoes a \new{randomization} step \eqref{eq:damping} (usually with $\eta=0$, i.e. the velocity is fully refreshed with a new Gaussian variable). The generator of these three continuous-time processes are respectively given by 
\begin{eqnarray*}
Lf(x,v) & = & v\cdot \na_x f(x,v) - \po \na U(x) + \gamma v\pf \cdot \na_v f(x,v) + \gamma \Delta_v f(x,v)\\
Lf(x) & = & -\na U(x) \cdot \na f(x) + \Delta f(x)\\
Lf(x,v) & = & v\cdot \na_x f(x,v) - \na U(x) \cdot \na_v f(x,v) \\
 & & \ +\  \lambda \int_{\R^d} \po f\po x, \eta v+  \sqrt{1-\eta^2} w\pf - f(x,v)\pf (2\pi)^{-d/2}e^{-|w|^2/2}\dd w\,.
\end{eqnarray*}
In statistical physics, the equations satisfied by the law of these processes are respectively called the kinetic Fokker-Planck, Fokker-Planck and BGK (or linear Boltzmann) equations.

The long-time behaviour of these three continuous-time processes has been studied under various assumptions and using various techniques. The simplest of them is the overdamped Langevin process, which is an elliptic reversible diffusion process. There is a plethoric literature on this process, and in particular its long-time convergence can be established with either Lyapunov/Doeblin approaches \cite{BAKRY2008727,HairerMattingly2008}, coupling arguments \cite{EBERLE20111101}, or functional inequalities  (e.g. spectral theory, entropy methods, hypercontractivity\dots cf. \cite{BAKRY2008727,BakryGentilLedoux,Markowich99onthe}). The analysis of the Langevin process is more complicated, since it is a non-reversible non-elliptic hypoelliptic diffusion process. However, much progress has been made in this matter in the last two decades, and there is now a lot of results available for this process too. In particular, again, the question of its long-time convergence to equilibrium has been addressed via Lyapunov methods \cite{Talay,Wu2001,MATTINGLY2002185}, direct couplings \cite{EberleGuillinZimmer,MonmarcheContraction,Bolley2010} or functional inequalities (e.g. spectral theory \cite{Helffer2005}, hypoellipticity \cite{Herau2007}, modified entropies \cite{Villani2009,DMS2011}, Bakry-Émery calculus \cite{Baudoin,MonmarcheGamma}, variational methods \cite{ArmstrongMourrat}). Finally, the Randomized HMC process is neither reversible nor a diffusion process, its generator is non-local and it doesn't have the regularization property of hypoelliptic diffusion (in particular, if the initial distribution of the process has atoms then so does its law at time $s$ for all $s\geqslant 0$). However, its long-time convergence can be established via Lyapunov methods \cite{RHMC}, direct couplings  \cite{Doucet},   entropy methods \cite{CaceresCarrilloGoudon,Hrau2006HypocoercivityAE,Evans2017,M21,Doucet,ADNR} or spectral theory \cite{Robbe2014}. The sample of references and techniques given here is not meant  to be exhaustive, since the literature is huge and these continuous-time dynamics are not the topic of this work. Besides, there are many relations between the different techniques, so the way we distinguish them here is partly arbitrary.

The long-time convergence of idealized or unadjusted HMC discrete-time chains have, up to now, mainly been addressed by Lyapunov or direct coupling methods \cite{Seiler2014PositiveCA,RHMC,MangoubiSmith,BouRabeeEberle,ChenVempala,BouRabeeEberleZimmer,BouRabeeSchuh} (the analysis for Metropolis-adjusted HMC may differ, we refer to \cite{Dwivedi1,Dwivedi2} and references within). Indeed, functional inequalities methods are more natural for continuous-time diffusion processes: typically, one starts by differentiating some entropy, and then it remains to relate the entropy dissipation to the entropy itself (which usually crucially involves integration by parts and the chain rule, i.e. the local property of diffusion generators). Such methods have been applied to discrete-time Markov chains in e.g. \cite{Chatterji,Ma2,Vempala} but those are all time-discretizations of continuous-time processes (overdamped or kinetic Langevin). They follow the continuous-time computations and keep track of a continuous-time instantaneous numerical error. \new{Developing} functional inequality methods for fully discrete problems is an active research area, where the analogous of many results from the continuous-time framework are either much harder or even false, see e.g. the reviews \cite{LogSobDiscret1,LogSobDiscret2} or the recent \cite{Salez,Caputo} and references within.

 The main point of the present work is to show that it is possible to apply directly entropic methods to the discrete-time idealized HMC, without referring to any continuous-time limit process. More precisely, we will establish the analogous for the idealized HMC of the two following results for the Langevin diffusion:
 \begin{itemize}
 \item The hypocoercive modified entropy decay of \cite{Villani2009} \new{(this is Theorem~\ref{thm:dissipation_modified}) }
 \item The entropy/Wasserstein regularization of \cite{GuillinWang} \new{(this is Theorem~\ref{thm:Wasserstein/entropie}).}
 \end{itemize}
 The interests of these results with respect to previous works are the following:
\begin{itemize}
\item Long-time convergence of idealized processes can be used in combination with the analysis of the time-discretization numerical error to get non-asymptotic efficiency bounds for MCMC algorithms. When we follow this program using the numerous results available for \eqref{eq:kinLangevin} or \eqref{eq:overdLangevin}, we have to compare a diffusion process with the corresponding numerical scheme, in which case the strong error is not of the same order (in the step size $\delta$) as the weak error. For instance, \cite{DurmusULA} (resp. \cite{ChengNonconvex}) give a Wasserstein distance between the limit process and the numerical scheme of order $\sqrt{\delta}$ (resp. $\delta$) for the \new{overdamped} (resp. kinetic) Langevin diffusion. This is not the case when using the idealized HMC as a reference, since the numerical error is then only due to the deterministic Hamiltonian step, see Section~\ref{sec:unadjusted} where complexity bounds for the unadjusted HMC algorithm \new{(with splitting schemes of the Langevin diffusion as a particular case)} are established from our results on the idealized process \new{(by some aspects the bounds given in  Section~\ref{sec:unadjusted} are relatively rough; a more detailed analysis of the unadjusted algorithm is postponed to an upcoming work \cite{Toappear}). Similarly, the works \cite{Chatterji,Ma2,Vempala}, which are based on entropy methods (and deal with unadjusted chains) both consider stochastic Euler numerical schemes of \eqref{eq:kinLangevin} or \eqref{eq:overdLangevin} which are of order $1$ in the step-size. One of the advantage of Hamiltonian-based schemes being the use of second-order  splitting schemes which require only one computation of $\na U$ per time step, one of the main motivation of the present paper is to set the first step of the analysis with entropy methods of unadjusted HMC \new{ and, importantly, unadjusted splitting schemes of the Langevin process \eqref{eq:kinLangevin} (which is a reason why it is important in this work to cover the case $\eta>0$, and more specifically $\eta\rightarrow 1$ as $t\rightarrow 0$). It is not clear whether, for such splitting schemes, directly adapting the computations of \cite{Chatterji,Ma2} based on \eqref{eq:kinLangevin} as the unbiased continuous-time reference (which would be quite technical) would give the correct second-order scaling of the error, in view of the strong/weak error question mentioned above.}  }
\item \new{Another motivation to cover the case $\eta>0$ is the following. It has been recently established in \cite{MonmarcheHMCconvexe} that, at least for Gaussian target distributions, using unadjusted HMC with inertia (i.e. $\eta>0$) outperforms the classical case $\eta=0$. More precisely, denoting by $\kappa$ the condition number of the variance matrix of the target, using a damping parameter with $1-\eta$ of order $1/\sqrt{\kappa}$ gives a convergence rate of order $\sqrt{\kappa}$ instead of $\kappa$, the latter being the convergence rate obtained for the optimal choice of the integration time $t$ when $\eta=0$. It is thus interesting to obtain non-asymptotic convergence estimates beyond the case $\eta=0$. To our knowledge, such results were not available in the non-convex case. }

\item \new{Our results concern the relative entropy, and not only the $L^2$ norm as in \cite{ADNR,MRZ,CaoLuWang,LuWang} based on the approaches of either \cite{DMS2011} or \cite{ArmstrongMourrat} (in fact our results also cover the $L^2$ case, see Section~\ref{subsc:G-entropy}). This is important in view of our practical motivation since, by contrast to the $L^2$ norm, the relative entropy is   amenable to the study of numerical schemes, as in \cite{Vempala,Ma2,Chatterji} (see also \cite{Toappear}). Moreover, contrary to Wasserstein distances and relative entropy, the $L^2$ distance does not scale well with the dimension of the space (the chi-square divergence between $\nu^{\otimes d}$ and $\mu^{\otimes d}$ is exponential in $d$). This is critical for modern high-dimensional problems, as we illustrate in the mean-field case, see Section~\ref{Sec:mean-field}. 
}

\item \new{Recently, reflection coupling arguments have been applied to the  kinetic Langevin diffusion \cite{EberleGuillinZimmer,ChengNonconvex} or  HMC (either idealized or unadjusted) with $\eta=0$ \cite{BouRabeeEberle,BouRabeeEberleZimmer,BouRabeeSchuh} assuming that $U$ is strongly convex outside a compact set. This yields a convergence in the $\mathcal W_1$ distance sense, which can then be transfered to a total variation convergence using a regularization result \cite{BouRabeeEberle}. Direct coupling methods are quite robust and, in particular, contrary to functional inequality approaches, they do not rely on an explicit expression of the invariant measure, and thus unadjusted numerical schemes are treated without additional difficulty with respect to continuous-time dynamics. On the other hand, apart from the fact we deal with stronger distances (relative entropy and $\mathcal W_2$ instead of total variation and $\mathcal W_1$), an advantage of functional inequality methods is that, in some cases, they provide bounds on the convergence rate which are sharper than those obtained by explicit coupling methods. A first example is the case of log-concave (but not log-strongly-concave) target measures, in relation with the recent results on the KLS conjecture \cite{ChenKLS,Lehec,Klartag} which yield a mild dependency in the dimension (see Section~\ref{sec:logconcave}). That being said, in the present work, we have particularly in mind the multi-modal non-convex case, in which case a way to get some understanding of the convergence rate is to consider the small temperature regime, namely  $\beta \rightarrow \infty$ with a target measure $\pi \propto e^{-\beta U}$ where $U$ has several local minima. In this framework, it is known that, as $\beta\rightarrow \infty$, up to a polynomial pre-factor, the convergence rate of any of the three continuous-time processes above behaves like $e^{-\beta c_*}$ where $c_* >0$ is the so-called critical height of the \new{potential} (see e.g. \cite{Holley,M17,Robbe2014,Nier2004} and Section~\ref{sec:annealing}), which is captured by our results, while coupling methods yield a convergence rate of order $e^{-\beta C}$ for some $C$ which is usually strictly larger than $c_*$ (see Remark~\ref{rem:comparaison_avec_couplage}). More generally, contrary to functional inequalities techniques, direct coupling methods  use some sort of ``worst case" local information about the potential $U$ (typically a bound on $(x-y)\cdot (\na U(x)-\na U(y))$ uniformly over $x$ and $y$ in some ball) which, except in particularly simple cases, is often not sufficient to convey a good non-local information about the geometry of the potential (such as the critical height).  We illustrate the interest of our sharp rates with  the theoretical analysis of the simulated annealing algorithm in Section~\ref{sec:annealing}. }
\end{itemize} 
Moreover, since the idealized HMC chain can be made to converge to different continuous-time processes depending on the choice of the parameters, our results shed some new lights on past works. In particular, we will keep in mind that our estimates should not degenerate when passing to the limit $t\rightarrow 0 $ in the regimes leading either to \eqref{eq:kinLangevin} or \eqref{eq:overdLangevin}.

The rest of this work is organized as follows. In the rest of this introduction, we introduce logarithmic Sobolev inequalities in Section~\ref{sec:logSob} and our assumptions in Section~\ref{sec:rescaling}. The main results are presented in Section~\ref{sec:MainResults}, in particular Theorem~\ref{thm:dissipation_modified} for the long-time convergence and Theorem~\ref{thm:Wasserstein/entropie} for the short-time regularization. Section~\ref{sec:prelimHamilton} gathers some preliminary results on the Hamiltonian dynamics, which are then used in Section~\ref{sec:proof} where the main results are proven. Finally, some examples are discussed in Section~\ref{sec:examples}, more precisely log-concave target measures in Section~\ref{sec:logconcave}, the mean-field scaling and its limit non-linear evolution in Section~\ref{Sec:mean-field}, the low-temperature regime and its application to the analysis of the simulated annealing algorithm in Section~\ref{sec:annealing}, and finally a complexity bound of the unadjusted HMC algorithm in Section~\ref{sec:unadjusted}.

\subsubsection*{Notations}

For $z\in \R^d$ and $A$ a matrix we write $|z|$ the Euclidean norm and $|A|$ the associated operator norm, and $\|A\|_F = (\sum_{ij}A_{ij}^2)^{1/2}$ the Frobenius norm. 
 For $\Phi\in\mathcal C^1(\R^n,\R^m)$,  we use the notation and convention $\na \Phi = (\partial_{z_i} \Phi_j)_{i\in\cco 1,n\ccf,j\in\cco 1,m\ccf}$ ($i$ stands for the \new{row} and $j$ the column) for the Jacobian matrix  of $\Phi$. This is the convention which ensures that $\na (\Psi\circ\Phi) =\na \Phi \na \Psi\circ \Phi$ for $\Psi\in \mathcal C^1(\R^{m},\R^d)$ and which, in the case $m=1$, is such that the Jacobian matrix of $\Phi$ is also its gradient.  If $\Phi(z)=Az$ with $A$ a constant matrix then $\na\Phi(z) = A^T $, where $A^T$ stands for the transpose of the matrix $A$.

\subsection{Relative entropy and log Sobolev inequality}\label{sec:logSob}

The relative entropy    of a law $\nu$ on $\R^{2d}$ with respect to a law $\mu$ is given by
\[\Ent(\nu|\mu) \ = \ \left\{\begin{array}{ll}
\int_{\R^{2d}} \ln\po \frac{\dd \nu}{\dd \mu}\pf \dd \nu & \text{if } \nu \ll \mu\\
+ \infty & \text{otherwise.}
\end{array}\right.\]
A related quantity is the Fisher Information
\[\mathcal  I(\nu|\mu) \ = \ \left\{\begin{array}{ll}
4 \int_{\R^{2d}} \left|\na \sqrt{\frac{\dd \nu}{\dd \mu}}\right|^2 \dd \nu & \text{if } \nu \ll \mu\\
+ \infty & \text{otherwise,}
\end{array}\right.\]
where for a measurable function $f$ on $\R^{2d}$, $|\na f|$ is defined as
\[|\na f|(z) \ = \ \lim_{r\downarrow 0}\ \sup\left\{ \frac{|f(z)-f(y)|}{|z-y|},\ y\in\R^{2d},\ 0<|y-z|\leqslant r\right\}\,.\]
Of course this definition is consistent with the norm of the gradient when $f$ is smooth.

The Pinskers' inequality states that, for all probability distributions $\nu,\mu$,
\[\|\nu-\mu\|_{TV}^2 \leqslant 2\Ent(\nu|\mu)\,,\]
 where $\|\cdot\|_{TV}$ stands for the total variation norm.

 The measure $\mu$ is said to satisfy a log-Sobolev inequality with constant $C_{LS}>0$ (which will often be shorten as \emph{$\mu$ satisfies a LSI($C_{LS}$)})  if
\begin{equation}\label{eq:log-Sob}
\forall \nu \ll \mu,\qquad \Ent(\nu|\mu) \ \leqslant \ C_{LS} \mathcal I(\nu|\mu)\,.
\end{equation}
If $\mu$ satisfies such an inequality then, as proven in \cite{OttoVillani}, it satisfies a $T_2$ Talagrand inequality with constant $C_{LS}$, which reads
\[\forall \nu \ll \mu\,,\qquad \mathcal W_2^2\po \nu,\mu \pf \leqslant C_{LS} \Ent(\nu|\mu)\,,\]
where $\mathcal W_2$ is the $L^2$ Wasserstein distance, defined by
\[\mathcal W_2^2\po \nu,\mu\pf = \inf_{r\in\mathcal C(\nu,\mu)} \int_{\R^{2d}\times \R^{2d}} |z-z'|^2 r(\dd z,\dd z')\] 
with $\mathcal C(\nu,\mu)$ the set of probability measures on $\R^{2d}\times\R^{2d}$ with marginals $\nu$ and $\mu$.

%

Since the standard Gaussian law  satisfies a log-Sobolev inequality with constant $1$ and such inequalities tensorises (see \cite{BakryGentilLedoux}), if we assume that $\pi$ satisfies a log-Sobolev \new{inequality} with constant $C_{LS}$, then $\mu = \pi\otimes\mathcal N(0,I_d)$ satisfies a log-Sobolev inequality with constant $\max(C_{LS},1)$, and more precisely, for all smooth positive $h$ with $\int_{\R^{2d}} h \dd \mu=1$,
\begin{equation}\label{eq:logSob_h_bar}
\int_{\R^{2d}} h \ln  h \dd \mu \leqslant \int_{\R^{2d}} \frac{C_{LS}|\na_x  h|^2 +|\na_v  h|^2}{ h} \dd  \mu\,.
\end{equation}

\subsection{Main assumption and rescaling}\label{sec:rescaling}

In this whole work,  we assume the following basic condition on the target measure and the integration time $t$ \new{(which is a fixed parameter throughout this work)}:
\begin{assumption}\label{hyp:main}
The target distribution $\pi$ has a density proportional to $\exp(-U)$ where $U\in\mathcal C^2(\R^d)$.  Moreover, there exists $L>0$ such that $|\na^2 U(x)|\leqslant L$ for all $x\in\R^d$, and $t\sqrt{L} \leqslant 1/4$.
\end{assumption}

\begin{remark}\label{rem:assu1}
The condition that $t$ is small enough is consistent with the usual restriction for HMC and is necessary for all the results stated in this work, since periodic Hamiltonian trajectories have to be avoided, see e.g. \cite{MonmarcheHMCconvexe}. The sharp condition  in the Gaussian case is $t \sqrt{L} < \pi$ \new{(here, contrary to most of the rest of the work, $\pi=3.14\dots$ does not stand for the target distribution)}. 
  Besides, we are particularly interested in the Langevin case, for which $t \sqrt{L} \ll 1$.
\end{remark}

Under Assumption~\ref{hyp:main}, let $(X_k,V_k)_{k\in\N}$ be an idealized HMC chain associated to a potential $U$, a time $t$ and a damping parameter $\eta$. Then $(\sqrt{L} X_k,V_k)_{k\in\N}$ is an idealized HMC chain associated to the potential $\tilde U(y)=U(y/\sqrt{L})$, the time $t\sqrt{L}$ and the damping parameter $\eta$. Indeed, if $(y_s,v_s)_{s\geqslant 0}$ follows the Hamiltonian dynamics $(\dot y,\dot v)=(v,-\na U(y))$, then $(u_s,w_s):=(\sqrt{L}y_{s/\sqrt{L}},w_{s/\sqrt{L}})$ solves $(\dot u,\dot w) =(w,-\na\tilde U(u))$. If the target measure $\pi\propto e^{-U}$ satisfies a LSI($C_{LS}$) then $\tilde \pi \propto e^{-\tilde U}$ (its image by the multiplication by $\sqrt{L}$) satisfies a LSI($LC_{LS}$). Moreover, $\na\tilde U$ is $1$-Lipschitz.

As a consequence, it is natural to work with the rescaled Wasserstein distance 
\[\mathcal W_{2,L}^2\po \nu,\mu\pf = \inf_{r\in\mathcal C(\nu,\mu)} \int_{\R^{2d}\times \R^{2d}} \po |x-x'|^2+\frac1{L}|v-v'|^2 \pf r(\dd x\dd v\new{,}\dd x'\dd v')\,,\]
so that
\[\mathcal W_{2,L}^2 (\mathrm{Law}(X_k,V_k),\mu) = \frac1L \mathcal W_{2}^2 (\mathrm{Law}(\sqrt {L} X_k,V_k),\tilde \mu)\]
with $\tilde \mu = \tilde \pi \otimes \mathcal N(0,I_d)$ and, similarly, the rescaled  Fisher Information
\[\mathcal I_{L}(\nu|\mu) = \int_{\R^{2d}} \frac{|\na_x h|^2 + \frac1L |\na_v h|^2}{h} \dd \mu \,. \]
The entropy is invariant by scaling, as can be seen by a change of variable:
 \[\Ent\po \mathrm{Law}(\sqrt{L} X_k,V_k)|\tilde\mu \pf  = \Ent\po \mathrm{Law}( X_k,V_k)|\mu \pf\,. \] 
Besides, recall that we are only interested in the law of the position $x_n$, and then
 \[\mathcal W_2\po \mathrm{Law}(X_k),\pi\pf \leqslant \mathcal W_{2,L}\po \mathrm{Law}(X_k,V_k),\mu\pf \]
 and, for the initial condition, we can start with a velocity at equilibrium and independent from the initial condition, in which case
 \[\mathcal W_{2,L}\po \mathrm{Law}(X_0,V_0),\mu\pf = \mathcal W_2\po \mathrm{Law}(X_0),\pi\pf  \,.\]
 The same goes for the Fisher information.  

The interest of working with these scaled quantities is that, in all the statements and proofs, without loss of generality, we can assume that $L=1$. Then, at the end, to get the results on the initial chain from the rescaled chain, one simply has to replace $t$ by $t\sqrt{L}$, $C_{LS}$ by $L C_{LS}$, $\mathcal W_2$ by $L\mathcal W_{2,L}$ and $\mathcal I$ by $L \mathcal I_L$.

\section{Main results}\label{sec:MainResults}

For fixed $t>0,\eta\in[0,1)$, consider the Markov transition operators $\D_\eta$ and $\H_t$ given by
\[\H_t f(z) = f\po \Phi_t(z)\pf \qquad\text{and}\qquad   \D_\eta f(x,v) = \mathbb E \po f  \po x, \eta v + \sqrt{1-\eta^2}G\pf\pf\,, \]
for all bounded measurable functions $f$, where $G$ is a standard $d$-dimensional Gaussian variable. We skip the subscript and simply write $\H$ and $\D$ when there is no ambiguity on the parameters. With these notations, we define the idealized HMC as the Markov chain on $\R^{2d}$ with  transition operator
  \[\P = \D  \H \,.\]
 Moreover, we set
  \[\gamma = \frac{1-\eta}{t }\,,\]
  which measures the strength of the damping. When we use the rescaling of Section~\ref{sec:rescaling}, $\gamma$ has to be replaced by $\gamma /\sqrt{L}$ in the result. 
  
  \subsection{Modified entropy dissipation}\label{subsec:convergence}

Our first result is an analogous of the hypocoercive entropy decay  of Villani \cite{Villani2009} for the Langevin diffusion. Consider a modified entropy of the form
\begin{equation}\label{eq:modifiedEntropie}
\mathcal L\po \nu \pf \ = \ \Ent(\nu|\mu) + a\int_{\R^{2d}} \frac{|\sqrt{L} \na_x h + \na_v h|^2 }{h}\dd \mu 
\end{equation}
with $h=\dd \nu/\dd \mu$, for some parameter $a>0$. 

\begin{theorem}\label{thm:dissipation_modified}
Under Assumption~\ref{hyp:main} with $L=1$, let 
\[
m_1  = 2t - \frac{21}{10} t^2  \,,\qquad
m_2 =  1 + \eta\po - 1 +  2 t + \frac{13}{5} t^2\pf\,,\]
\[ m_3 =  1- \eta^2\po    1 +2t + \frac{31}{10} t^2\pf + \frac{1-\eta^2}{2a}\,,
\]
and
\begin{equation}\label{eq:rho}
\rho  = (1-3t) \co \frac{m_1+m_3 \eta^{-2}}{2} - \sqrt{m_2^2 \eta^{-2} + \po \frac{m_1-m_3 \eta^{-2}}{2}\pf^2} \cf
\end{equation}
if $\eta>0$ or its limit $\rho = (1-3t)\co m_1 - 2a/(2a+1)\cf$ if $\eta =0$. Assume that $a$ is small enough so that $m_1m_3 > m_2^2$, which is equivalent to $\rho>0$, and that $\pi$ satisfies a LSI($C_{LS}$). Then, for all $\nu\in\mathcal P(\R^{2d})$,
\begin{equation}\label{eq:LnuP}
\mathcal L(\nu \P) \leqslant  \po 1+ \frac{\rho }{\max(C_{LS},1)/a +2 }\pf^{-1}  \mathcal L(\nu)   \,.
\end{equation}
In particular, if we take $a \leqslant \gamma / [14 + 8(\gamma+3)^2]$, then \eqref{eq:LnuP} holds with
$\rho \geqslant 3t/8$.
\end{theorem}

The proof is done in Section~\ref{sec:proof-hypoco}. Let us comment this result.

\begin{itemize}
\item The assumption that $\|\na^2 U\|_\infty <\infty$ implies that $U$ grows at most quadratically at infinity. On the other hand, a LSI implies that $U$ grows at least quadratically, see \cite[Theorem 3.1.21]{Royer}. Hence, our result concerns target measures with Gaussian tails. \new{Moreover, as mentioned in Remark~\ref{rem:assu1}, and by contrast with the kinetic Langevin process (see e.g. \cite[Theorem 35]{Villani2009} or \cite{CaoLuWang}) or the continuous-time Randomized HMC (see \cite{LuWang}) the fact that $\na^2 U$ is bounded and then the fact that $t$ is sufficiently small with respect to $\|\na^2 U\|_\infty$ is necessary to get the decay of the entropy. More precisely, if $t \sqrt{L} \geqslant \pi$, there exists $U\in\mathcal C^2(\R^d)$ with $\|\na^2 U\|_\infty \leqslant L$ and such that, for any $\eta\in[0,1)$, any idealized HMC chain $(X_n,V_n)_{n\in\N}$ with the corresponding parameters is such that $X_n=X_0$ for all $n\in\N$ (which clearly prevents  Theorem~\ref{thm:dissipation_modified} to hold). Such a counter-example is for instance given by $U(x) = \pi^2|x|^2/(2t^2)$, see \cite{MonmarcheHMCconvexe} for more details.  }
 
\item For clarity, let us  focus on the bound on the  contraction rate stated in the last part of Theorem~\ref{thm:dissipation_modified}, namely $3t/[8 \max(C_{LS},1)/a+16]$ with $a = \gamma / [14 + 8(\gamma+3)^2]$.
It scales as $\gamma$ as $\gamma$ vanishes and as $1/\gamma$ as $\gamma\rightarrow \infty$, which is sharp, as in the case of the continuous time Langevin diffusion. For fixed $t,L,\gamma$, it scales as $1/C_{LS}$ when $C_{LS}\rightarrow \infty$. It is of order $t$ (for a fixed $\gamma$), which is expected. In particular, letting $t$ vanish in the bound on $\mathcal L(\nu P^n)$ given by Theorem~\ref{thm:dissipation_modified} with $n=\lfloor s/t\rfloor $ for some fixed $s>0$ with a fixed $\gamma$ we recover 
\begin{equation}\label{eq:Ts}
\mathcal L(\nu T_s) \leqslant \exp\po - \frac{ 3s   }{8\max(C_{LS},1)/a +16 }\pf \mathcal L(\nu) 
\end{equation}
where $(T_s)_{s\geqslant 0}$ is the semi-group of the Langevin diffusion. Alternatively, in the case $\eta=0$, as $t\rightarrow 0$, taking $a=t/2$, $\rho = (1-3t)\co m_1 - 2a/(2a+1)\cf \simeq t$. Since, in that case, $a$ vanishes with $t$, applying Theorem~\ref{thm:dissipation_modified} (without assuming that $L=1$ in Assumption~\ref{hyp:main}) with $n=\lfloor \new{2} s/t^2\rfloor $ iterations for some fixed $s>0$, we recover
\begin{equation}\label{eq:Rs}
\Ent(\nu R_s|\new{\pi}) \leqslant \exp\po - \frac{ L s}{ \max(LC_{LS},1) } \pf \Ent(\nu |\new{\pi})\,,
\end{equation}
where $(R_s)_{s\geqslant 0}$ is the semi-group of the overdamped Langevin diffusion.  Since $L$ is only assumed to be an upper bound of $\|\na^2 U\|_\infty$ it can be taken arbitrarily large and we end up with
\[\forall s \geqslant 0,\ \forall \nu\in\mathcal P(\R^d),\quad \Ent(\nu R_s|\pi) \leqslant \exp\po - \frac{ s}{C_{LS} } \pf \Ent(\nu |\pi)\,,\]
which, from \cite[Theorem 5.2.1]{BakryGentilLedoux} is \emph{equivalent} to the fact that $\pi$ satisfies a LSI($C_{LS}$), which means \new{that, provided the other assumptions, the conclusion of Theorem~\ref{thm:dissipation_modified} is equivalent to the LSI (in other words, we obtain the exact rate in this regime).} Also the fact that $L$ does not intervene is consistent with the fact that no bound on $\na^2 U$ is required in the continuous-time overdamped case.

\item
For a fixed $t$, the $\gamma$ which maximizes our bound only depends on $L$, and in particular in the strongly convex case where  $\na^2 U \geqslant m$ for some $m>0$ (in which case $\pi$ satisfies LSI($1/m$)) we do not recover results similar to  \cite{CaoLuWang} for the continuous-time Langevin  diffusion \new{(for the $L^2$ norm) or \cite{MonmarcheHMCconvexe} for the unadjusted HMC with Gaussian targets (for the Wasserstein distances), } where the optimal $\gamma$ is of order $\sqrt{m}$ and yields a contraction of order $\sqrt{m}$ instead of $m$ for small $m$ (in the spirit of Nesterov acceleration for convex optimization). \new{In the Gaussian case, in fact, the proof of Theorem~\ref{thm:dissipation_modified} is easily adapted to follow the analysis of \cite{MonmarcheHMCconvexe} and yield a similar acceleration in this case (but now in terms of relative entropy), see Section~\ref{subsec:convex_result}.  Concerning the work of Cao, Lu and Wang, the main point of the result of \cite{CaoLuWang} (which leads to the acceleration) is that the convergence rate does not involve an upper bound on $\na^2 U$ (in fact, the result of \cite{CaoLuWang} applies to potentials with unbounded Hessian matrices) and, due to the periodicity issue mentioned above, it is clear that such a result cannot hold for the chain studied in the present work. It is unclear how the results of \cite{CaoLuWang} (or \cite{LuWang} by the same authors but for the continuous-time Randomized HMC process) could be used for some numerical schemes, and a related question is whether it could be adapted to the relative entropy.  Notice that  \cite{CaoLuWang,LuWang} relies on the $L^2$ hypocoercivity method of \cite{ArmstrongMourrat}, however, as already mentioned, with his initial modified norm method (that inspires the proof of Theorem~\ref{thm:dissipation_modified}), Villani was already able to get the $L^2$ convergence for unbounded Hessian matrices in \cite[Theorem 35]{Villani2009} (so it may be possible to get the $\sqrt{m}$ scaling of \cite{CaoLuWang} in the convex case with this approach, although this is not entirely clear), but not for the relative entropy. Both approaches of \cite{ArmstrongMourrat} and \cite{Villani2009} in $L^2$ rely on Hilbert analysis. To our knowledge, the only result for hypocoercivity in relative entropy with unbounded Hessian matrices has been obtained in \cite{Cattiaux} and, as the analysis is more involved, it is unclear whether it would be possible to get sharp rates with this approach in the convex case. Finally, let us notice that it is not always clear that results in continuous-time can be transfered to numerical schemes. For instance, at the continuous-time level it is possible to add a divergence-free drift to the overdamped Langevin diffusion to get a non-reversible diffusion with an arbitrarily large convergence rate to equilibrium, but then the numerical schemes  are sensitive to the Lipschitz constant of the drift, see \cite{Constantin,2023arXiv230318168C} and references within in this topic.}
\item In practice, the computational cost of one iteration of the chain is overwhelmed by the simulation of the Hamiltonian dynamics, and is thus proportional to $t$. It means that the contraction  rate per computational time (i.e. the contraction rate divided by $t$) provided by Theorem~\ref{thm:dissipation_modified} is bounded uniformly in $t\in(0,\sqrt{L}/4]$ \new{(when $\gamma$ is fixed)}. The good side is that it means any choice of $t$ in this range gives a reasonable algorithm. The bad side is that our result is not accurate enough to compare different values of $t$, and in particular we cannot say if, for instance for the unadjusted HMC, it is better to use the classical HMC scaling (i.e. the integration time $t$ is fixed, of order $\sqrt{L}$, independent from the step size of the Verlet integrator) or the Langevin scaling (i.e. $t$ is exactly the step size of the Verlet integrator). On this topic, see the discussion of \cite{MonmarcheHMCconvexe} in the Gaussian case. \new{However, when $t$ is small, what does appear in Theorem~\ref{thm:dissipation_modified} is the difference between the ballistic scaling (namely $\gamma$ fixed; the position covers a distance of order $1$ in a number of iterations of order $1/t$) with respect to the diffusive scaling (namely $\eta=0$; the position covers a distance of order $1$ in a number of iterations of order $1/t^2$). Indeed, when $\gamma$ is fixed, so is $a= \gamma / [14 + 8(\gamma+3)^2] $, so that the convergence rate $\rho$ is of order $t$ (which is why we got \eqref{eq:Ts}) while, when $\eta=0$, $a$ is of order $t$ and then $\rho$ is of order $t^2$ (which is why we got \eqref{eq:Rs}).    }
\item \new{In relation to the two previous points, let us mention that a square-root acceleration of the convergence rate (i.e. going from an optimal convergence rate $\rho$ for HMC with $\eta=0$ to an optimal convergence rate $\sqrt{\rho}$ for HMC with $\eta>0$ or the Langevin diffusion) cannot be expected in a general non-convex case simply by using kinetic processes such as the underdamped Langevin diffusion. Indeed, as already mentioned, in the low temperature regime $\beta\rightarrow +\infty$ with a target $e^{-\beta U}$ where $U$ has several local minima, up to some sub-exponential prefactor in $\beta$, the sharp convergence rate of the kinetic and overdamped Langevin diffusion and of   the continuous-time Randomized HMC  are all of the order $e^{-\beta c_*}$ with the same $c_*>0$ \cite{Holley,M17,Robbe2014,Nier2004} (the sub-exponential prefactor are slightly different, in particular the inertia is seen in kinetic cases by the fact that the Hessian of the bottle-neck saddle point does not intervene in the leading term of this prefactor, contrary to the diffusive case, but this is far from a square-root improvement of the rate). Among processes with a local motion (e.g. continuous trajectory of the position), inertia and non-reversibility is expected to help  in convex and flat regions (see e.g. \cite{Diaconis2000,M22}) but it doesn't reduce energy barriers (notice that, in our case, the chain is reversible if and only if $\eta=0$, see Remark~\ref{eq:non-reversible}). As already mentioned in the Introduction, the advantage of Hamiltonian-based kinetic processes, with respect to e.g. the overdamped Langevin diffusion \eqref{eq:overdLangevin}, that we are targeting with our result is not that it gives a significative improvement  of the long-time convergence rate but rather that it enables the use of second-order discretization schemes \cite{Toappear} (which is somehow related to the diffusive/ballistic dichotomy, since an unadjusted HMC with $\eta=0$ and step-size $t$ corresponds to an Euler scheme of \eqref{eq:overdLangevin} with stepsize $t^2/2$).

}

\item In practice, one can try to have a good initial distribution by finding first $x_*$ a local minimizer of $U$ using a deterministic optimization algorithm, and then taking an initial condition with law $\nu_0 = \nu_0^1\otimes \mathcal N(0,I_d)$ with $\nu_0^1 = \mathcal N(x_*, I_d/L)$. Indeed, in that case, using that $U(x) \leqslant L|x-x_*|^2/2+U(x_*)$  we can bound
\begin{eqnarray*}
\Ent(\nu_0|\mu ) 
& \leqslant &  \int_{\R^d} \nu_0^1 \ln \nu_0^1 + \frac{L}2 \int_{\R^d}|x-x_*|^2 \nu_0^1(\dd x) + U(x_*) + \ln \int_{\R^d} e^{-U} \\
& = & d  \ln \po L/\sqrt{2\pi}\pf     + U(x_*) + \ln \int_{\R^d} e^{-U}\,,
\end{eqnarray*}
and similarly, using that $|\na U(x)| \leqslant L|x-x_*|$,
\[
\mathcal I(\nu_0|\mu)   \leqslant  2 \int_{\R^d} |\na \ln \nu_0^1|^2 \nu_0^1 + 2 \int_{\R^d} |\na U|^2 \nu_0^1  \leqslant 4dL
\]
(which also means we can bound $\Ent(\nu_0|\mu) \leqslant 4dL C_{LS}$). Hence, in terms of the dimension, we can consider that $\mathcal L(\nu_0)$ is $\mathcal O(d)$.
\end{itemize}

\subsection{Regularization}\label{subsec:regularization}

 Our second main result is the an entropy/Wasserstein regularization result, similar to the result of Guillin and Wang in \cite{GuillinWang} for the Langevin diffusion. We require the following additional condition:

\begin{assumption}\label{hyp:Frobenius}
The Hessian of $U$ is  Lipschitz. In particular, we consider  $L_H>0$ such that, for all $x,x'\in\R^{d}$,
\begin{equation}\label{ass:na2UFrob}
\|\na^2 U(x) - \na^2 U(x')\|_F \leqslant L_H |x-x'|\,.
\end{equation}
\end{assumption}

In the next result, we use the rescaling of Section~\ref{sec:rescaling} to work with $L=1$. To apply it with $L\neq 1$, $L_H$ has to be replaced by $L_H/L^{3/2}$. We consider the Markov transition operator $\Q = \D_\eta \H_t \D_\eta$.

\begin{theorem}\label{thm:Wasserstein/entropie}
\new{Under Assumptions~\ref{hyp:main} and \ref{hyp:Frobenius}  with $L=1$}, for all $\nu\in\mathcal P(\R^{2d})$,
\[\Ent \po \nu \Q |\mu \pf \ \leqslant \ c_1(\eta,t) \mathcal W_{2}^2 (\nu,\mu)\,, \] 
where
\[c_1(\eta,t) = \max(1,\eta^2 t^2 ) \po \frac{13}{2t^2(1-\eta^2)} +\new{5} L_H^2  t^4 \pf\,. \]
If, moreover, $\eta>0$ \new{and $t\leqslant 1/8$} then for all $\nu\in\mathcal P(\R^{2d})$ and all $n\geqslant 2$, 
\[\Ent \po \nu \Q^n |\mu \pf \ \leqslant \ c_{\new{n}}(\eta,t)\mathcal W_{2}^2 (\nu,\mu) \] 
 with $c_{n}(\eta,t) =  c_*(  t \min(n,\lfloor 1/(4t)\rfloor)  )$   where, recalling that $\gamma = (1-\eta)/t$, for all $s>0$,
\[c_*(s)=\max(1,\eta^2 s^2) \po\frac{ s  }{\gamma  } \po       \frac{12}{ \eta s^2}  + \frac{6\gamma}{\eta s   } +              4\pf^2  + \new{132}  e^{4s/3} (L_H  s)^2\pf\,.\]

\end{theorem}

This is proven in Section~\ref{subsec:proof-W2Ent}. Some remarks:
\begin{itemize}
\item Since $\D_{\sqrt{\eta}}^2 = \D_\eta$, we can use this result to get an information about  $\P$ by writing either $\P^n \D_{\sqrt\eta} = \D_{\sqrt{\eta}} \po \D_{\sqrt{\eta}}\H_t \D_{\sqrt\eta}\pf^{n}$ (in which case we simply replace $\eta$ by $\sqrt{\eta}$ when applying the result) or $\P^n \D_\eta = \P^{n-1} \Q = \Q \P^{n-1}$. In practice, adding $\D$ at the beginning or the end of a simulation has no effect since, anyway, the initial velocity is at equilibrium, and similarly the final velocity is not used. Moreover, applying $\D$ reduces both the relative entropy and the $\mathcal W_2$ distance with respect to $\mu$. The reason we work with $\Q$ instead of $\P$ here is that, as  discussed in \cite[Section 2.1]{MonmarcheSplitting}, having noise at the beginning and the end of a transition is necessary to get a one-step regularization (since both position and velocity have to be regularized).
\item As expected in this kinetic case, similarly to the Langevin case \cite{GuillinWang}, for a fixed $\eta$, $c_1(\eta,t)$ is of order $1/t^3$ for small $t$ and, similarly, $c_*(s)$ is of order $1/s^3$ for small $s$ \new{(treating $\gamma$ as a constant independent from $t$)}. In fact, we could distinguish the contribution from the position and the velocity in the Wasserstein distance and we would indeed get the scaling  $1/t^3$ (or $1/s^3$) for the position, but only $1/t$ (or $1/s$) for the velocity, see Proposition~\ref{prop:densitéW2} below.
\item The first statement is well adapted to the classical HMC case where $\eta$ is fixed (say $\eta=0$) and $t$ is of order $1/\sqrt{L}$. However, in the Langevin scaling where $t \ll 1$ is of the order of the time-step of a Verlet scheme and $\eta\simeq 1$, we have to use the second statement with $n$ of order $1/t$ to get a correct estimate (which are then consistent with the Langevin case \cite{GuillinWang}). Then, from this limit, if we accelerate time by a factor $\gamma$ and let $\gamma\rightarrow \infty$ we recover the result for the overdamped Langevin diffusion \cite{RocknerWang}.
\item \new{A similar $\mathcal W_1$/total variation regularization result has been established   in \cite{BouRabeeEberle} in the case $\eta=0$ and in \cite{MonmarcheSplitting} for a Langevin splitting scheme.  In fact, the main part of the  proof of Theorem~\ref{thm:Wasserstein/entropie}, which  is  Proposition~\ref{prop:densitéW2} below, is based on  \cite[Lemma 16]{BouRabeeEberle}.
\item \new{Thanks to Theorem~\ref{thm:Wasserstein/entropie}, any result of long-time convergence for the chain in the $\mathcal W_2$ Wasserstein distance (as in \cite{ChenVempala} for $\eta=0$ or \cite{MonmarcheHMCconvexe} for $\eta>0$, both in the strongly log-concave case, see Section~\ref{subsec:convex_result}) automatically gives a similar result for the relative entropy, with the same rate. }
}
\end{itemize}

\subsection{\new{Three} variations on Theorem~\ref{thm:dissipation_modified}}

\new{In this section, we give two generalizations of Theorem~\ref{thm:dissipation_modified}, and a slight adaptation in the strongly log-concave case. In fact, Theorem~\ref{thm:dissipation_modified} is a particular case of Theorem~\ref{thm:dissipation_modified_G} below, which is itself a particular case of Theorem~\ref{thm:dissipation_modified_G_random} below. The reason we present Theorem~\ref{thm:dissipation_modified}   instead of its generalizations as one our main results is that the former is our main motivation and its proof contains all the interesting ideas of this work. Avoiding any superfluous generality, the proof of Theorem~\ref{thm:dissipation_modified} is clearer, and from this first case the extensions are easy.
}

\subsubsection{$L^2$ norm and other entropies}\label{subsc:G-entropy}

Although we primarily focus on the relative entropy due to the scaling properties mentioned in the introduction and in view of Theorem~\ref{thm:Wasserstein/entropie}, the hypocoercive decay stated in Theorem~\ref{thm:dissipation_modified} also holds for more general  entropies.

Let $G$ be a $\mathcal C^4$ convex function on $\R_+$ such that $G(1)=0$ and $1/G''$ is positive concave. For $\nu \ll \mu$, denoting $h = \dd \nu/\dd \mu$, we call
\[\Ent_G\po \nu|\mu \pf = \int_{\R^d} G(h) \dd \mu\,,  \]
which is positive,  the $G$-entropy of $\nu$ with respect to $\mu$. An associated quantity is the $G$-Fisher Information defined by
\[\mathcal I_G\po \nu |\mu\pf = \int_{\R^d} G''(h) |\na h|^2 \dd \mu\,.\]
 We say that $\mu$ satisfies a $G$-Poincaré inequality with constant $C_G>0$ if
  \[\forall \nu \ll \mu,\qquad \Ent_G(\nu|\mu) \ \leqslant \ C_{G} \mathcal I_G(\nu|\mu)\,.\]
For $G(u) = u\ln u$, we recover the relative entropy, Fisher Information and log-Sobolev inequality. For $G(u)=(u-1)^2/2$, we get respectively $\| h-1\|_{L^2(\mu)}^2/2$ (i.e. half the square of the chi-square divergence of $\nu$ with respect to $\mu$), $\|\na h\|_{L^2(\mu)}^2$ and the classical Poincaré inequality. Other classical cases are $G(u) = u^p-1$ for $p\in(1,2]$, corresponding to Beckner inequalities.   We refer to \cite{BolleyGentil} for general considerations on such entropies  and some criteria to prove such inequalities.

Theorem~\ref{thm:dissipation_modified} in fact applies in this more general framework. In other words, considering a modified $G$-entropy of the form
\begin{equation}\label{eq:LG}
\mathcal L_G\po \nu \pf \ = \ \Ent_G(\nu|\mu) + a\int_{\R^{2d}} G''(h) |\sqrt{L} \na_x h + \na_v h|^2 \dd \mu 
\end{equation}
 for some parameter $a>0$, the following holds:

\begin{theorem}\label{thm:dissipation_modified_G}
Theorem~\ref{thm:dissipation_modified} is still true (with the same constants) if $\mathcal L$ is replaced by $\mathcal L_G$ and the LSI inequality is replaced by a $G$-Poincaré inequality.
\end{theorem}

We do not detail the full proof, and simply explain in Section~\ref{subsec:Gentropie} how the proof of Theorem~\ref{thm:dissipation_modified} should be adapted to get this generalization.

One interest of Theorem~\ref{thm:dissipation_modified_G} with respect to Theorem~\ref{thm:dissipation_modified} is that, for instance with $G(u) = (u-1)^2/2$, it only requires a Poincaré inequality, which is weaker than the LSI and can hold for target measure with exponential tails (instead of Gaussian for the LSI).

\subsubsection{Random step-size}

Finally, the adaptation of the proof of Theorem~\ref{thm:dissipation_modified} to the case where the step-size $t$ is random is straightforward (at least if $t\leqslant 1/(4\sqrt{L})$ almost surely, otherwise we have to take into account that for larger values the entropy does not decay and the Fisher information part of $\mathcal L$ may in fact increase. \new{This extension would have no interest since allowing times larger than $1/(4\sqrt{L})$ would only result with our method in a worse upper bound on the convergence rate.} We won't discuss this case). The interest of a random step-size in practice is that it reduces the sensibility of HMC to periodic resonances \cite{RHMC}.

Let $\theta$ be a probability measure on $(0,\infty)$, and $\eta \in[0,1)$ be a fixed constant. The discrete-time Randomized idealized HMC chain is the Markov chain with transition operator $\P_\theta$ given by
\begin{equation}\label{eq:Ptheta}
\P_\theta f(z) = \int_0^\infty \D_\eta  \H_t  f(z) \theta(\dd t)\,,
\end{equation}
so that the integration time of the Hamiltonian dynamics is a random variable with law $\theta$. In the next result, we write $\rho(t)$ the constant \eqref{eq:rho} associated to the parameter $t$.

\begin{theorem}\label{thm:dissipation_modified_G_random}
Assume that:
\begin{itemize}
\item the target distribution $\pi$ has a density proportional to $\exp(-U)$ where $U\in\mathcal C^2(\R^d)$ and satisfies a $G$-Poincaré inequality with constant $C_G$,
\item for all $x\in\R^d$, $|\na^2 U(x)|\leqslant 1$, 
\item the support of $\theta$ is included in $[t_0,t_1]$ for some $0<t_0 \leqslant t_1 \leqslant 1/4$.
\end{itemize}
Then, considering $\mathcal L_G$ given by \eqref{eq:LG} and $a$ sufficiently small so that $ \rho(t)>0$ for all $t\in[t_0,t_1]$,  then, 
 for all $\nu\in\mathcal P(\R^{2d})$,
\[\mathcal L_G(\nu \P_\theta) \leqslant  \int_0^{\new{1/4}}\po 1+ \frac{\rho(t)}{\max(C_{G},1)/a +2 }\pf^{-1} \theta(\dd t) \mathcal L_G(\nu)   \,.\]
\end{theorem}

This is proven in Section~\ref{subsec:random}. A few comments:

\begin{itemize}
\item In particular, in view of the rough bound stated at the end of Theorem~\ref{thm:dissipation_modified}, we can take
\[a = \min_{i\in\{0,1\}}  \frac{\gamma_i}{14 + 8(\gamma_i+3)^2}\quad\text{where}\quad \gamma_i = \frac{1-\eta}{t_i}\,, \ i\in\{0,1\}\,,\]
in which case $\rho(t) \geqslant 3t/8$ for all $t\in[t_0,t_1]$.
\item In fact we could similarly consider the case where $\eta$ is random at each step. In particular we could take $\eta = 1-\gamma T$ where $\gamma$ is fixed  and $T\sim \theta$. The proof would be exactly the same, with the same condition of $a$ that in Theorem~\ref{thm:dissipation_modified}. However, contrary to random step sizes, random damping parameters are not used in practice so we do not detail this.
\item Notice that, contrary to the study in \cite{RHMC} of the auto-correlation time in the Gaussian case, Theorem~\ref{thm:dissipation_modified_G_random} does not help in understanding whether using a random step-size is useful in practice. This is because we work under the assumption that $t\sqrt{L}\leqslant 1/4$, which means we are below the threshold above which HMC suffers from periodicity issue, even for the highest frequency of the system. \new{For this reason,  Theorem~\ref{thm:dissipation_modified_G_random} has to be understood as a small collateral gain from Theorem~\ref{thm:dissipation_modified_G} worth mentioning rather than a crucial progress in the study of Randomized HMC. Besides, concerning the efficiency in term of convergence rate for the time marginal, for unadjusted Randomized HMC, the recent work \cite{rHMCGaussien} (released as a preprint at the same time as the present work, and after \cite{MonmarcheHMCconvexe}) shows that, for Gaussian target distributions, the use of a suitably tuned random refreshment time leads to a convergence rate of order $1/\sqrt{\kappa}$ (in total variation), where $\kappa$ is the condition number of the covariance matrix of the target (which also correspond to the results of \cite{LuWang} for more general potentials but for the continuous-time Randomized HMC, and for the $L^2$ norm, which has already been discussed in Section~\ref{sec:MainResults}), which is an improvement with respect to the optimal HMC with deterministic integration time and $\eta=0$, for which the convergence rate scales as $1/\kappa$. However, it is established in \cite{MonmarcheHMCconvexe} that, similarly, for Gaussian targets, by allowing $\eta>0$, the convergence rate (in Wasserstein distances) of the optimal unadjusted HMC with deterministic integration time also scales as $1/\sqrt{\kappa}$, and so does the optimal unadjusted kinetic Langevin   splitting scheme (which is thus to \cite{CaoLuWang} what  \cite{rHMCGaussien} is to \cite{LuWang}, namely a result  for a realistic discretized algorithm, but restricted to Gaussian targets, instead of a general result in $L^2$ for a continuous-time process).   In other words, long randomized integration times or inertia with $\eta >0$ are two different approaches leading to the same gain. In the continuity of \cite{MonmarcheHMCconvexe}, the present work is mostly concerned with the latter approach, which is why we focus   much in the whole article on the Langevin regime where $1-\eta = \mathcal O(t) $ is small. Notice that both approaches are in fact linked by the same idea, which is that the typical mean free path of the position should be of order $\sqrt{\kappa}$. Indeed, assume for simplicity that $U=0$. Then, starting from a state $(x,v)$, $\mathbb E(V_n) = \eta^n v$ and thus $\mathbb E(X_n) \simeq  x + t v/(1-\eta) $ for large $n$. This mean free path $t/(1-\eta)$ is of order $\sqrt{\kappa}$ under the optimal scaling in the Gaussian case established in \cite{MonmarcheHMCconvexe}. Similarly, the time integration in \cite[Lemma 3.1]{rHMCGaussien} is uniformly distributed over  $[0,10\pi\sqrt{\kappa}]$ (normalizing $\|\na^2 U\|_\infty = 1$). This is also the consistent with the damping parameter  in \cite{CaoLuWang}  and the refreshment rate of \cite{LuWang} (both of order $1/\sqrt{\kappa}$) in continuous time. As a last comment on this topic, we can think that the two approaches can in fact easily be combined, i.e. using short random step size (which is sufficient to destroy any high-frequency periodicity of the order of the step size) and inertia (i.e. $1-\eta$ of the order of the step size) to ensure a long mean free path. 
 }     
\item \new{A variant of the randomized integration time procedure addressed here is the Andersen's thermostat \cite{Andersen1,Andersen2}. Initially introduced for models in physics where the state is $x=(x_1,\dots,x_N) \in (\R^p)^N$ with $N$  the number of particles and $x_i\in\R^p$ the position of the $i^{th}$ particle, the difference with $\P_\theta$ is that the velocity of each particle is refreshed at a time which is independent from the other particles (so, our framework is similar to the particular case of Andersen's thermostat where $N=1$ and $p=d$). Our method could be extended to a slight modification of $\P_\theta$ in this spirit where, in a time interval $[0,t_*]$ (with $t_*$ either deterministic or random, but with $t_* \leqslant 1/(4\sqrt{L})$ almost surely in any cases), each particle sees its velocity refreshed at a time $t_i \leqslant t_*$ (for instance $t_i = u_i t_*$ where $u_1,\dots,u_N$ are uniformly distributed over $[0,1]$ and independent from $t_*$). Indeed, it is crucial in our proof that, during one iteration of the chain, all velocities have been (partially) refreshed with probability 1. When this is not the case, $L^2$ hypocoercivity results have been established (e.g. in \cite{LuWang,Doucet,ADNR}) but entropic results are very restricted \cite{Evans2017,M21} (see also the discussion in \cite[Section 2.3]{MonmarcheContraction} on a related topic). Extending our results to the genuine Andersen's thermostat thus seems challenging.   }
\end{itemize}

\subsubsection{\new{The strongly convex case}}\label{subsec:convex_result}

\new{
In the  case where $U$ is strongly convex, a method to get a long-time convergence in Wasserstein distance is to consider the parallel coupling of two processes, namely to construct two trajectories from  different initial conditions with the same source of randomness (i.e. same Brownian motion for the overdamped or kinetic Langevin diffusions, same Gaussian variables in the velocity refreshment for HMC), and then prove that the distance between the two trajectories goes to $0$ with time. This has been extensively used over the last years, see \cite{MonmarcheHMCconvexe} and references within.

 Let us for instance quote two results, from \cite{ChenVempala} and \cite{MonmarcheHMCconvexe}, which concern the idealized HMC considered in the present work. In the next statement, $(z_k)_{k\in\N}$ and $(z_k')_{k\in\N}$ are called a parallel coupling of two idealized HMC if they are both Markov chains associated to $\P$ and such that the variables $G_k$ used in the step \eqref{eq:damping} are the same for both chains.

\begin{proposition}\label{prop:couplage_para}
Assume that there exists $ m>0$ such that $m  I_d\leqslant \na^2 U(x) \leqslant I_d$ for all $x\in\R^{2d}$. Let $(z_k)_{k\in\N}=(x_k,v_k)_{k\in\N}$ and $(z_k')_{n\in\N}=(x_k',v_k')_{k\in\N}$ be  a parallel coupling of idealized HMC.
\begin{enumerate}
\item (from \cite[Lemma 6]{ChenVempala}) Assume furthermore that  $\eta=0$ and $t\leqslant 1/2$. Then, for all $k\in\N$, almost surely,
\[|x_{n}- x_{n}'| \leqslant \po 1 - \frac{m}{4} t^2 \pf^n |x_0-x_0'|\,.\] 
\item (from \cite[Proposition 25]{MonmarcheHMCconvexe}) Let $\overline{\gamma}\geqslant 2 $. Assume furthermore that $t \leqslant m /[24\overline{\gamma}(2+\overline{\gamma}^2)]$ and $\eta = 1 - \overline{\gamma}  t$. Then, for all $n\in\N$, almost surely
\begin{equation}\label{eq:couplage_para_HMC}
|z_{n}- z_{n}'| \leqslant  3 \po 1 - \frac{m}{12\overline{\gamma}} t \pf^n |z_0-z_0'|\,.
\end{equation}
\end{enumerate}
\end{proposition}
More precisely, \cite{MonmarcheHMCconvexe} is concerned with the unadjusted chain where the Hamiltonian dynamics is replaced by a Verlet integrator and thus we get the result here by sending the step size to $0$ in the results of \cite{MonmarcheHMCconvexe} (which is possible thanks to the numerical error bounds established in \cite{MonmarcheHMCconvexe}).

Results such as those given in Proposition~\ref{prop:couplage_para} immediately yields a contraction of Wasserstein distances $\mathcal W_p$ for all $p\geqslant 1$ (hence a convergence in relative entropy thanks to Theorem~\ref{thm:Wasserstein/entropie}). As we state in Proposition~\ref{prop:convex} below, which can be seen as a variation of Theorem~\ref{thm:dissipation_modified} where no entropy part is required in the modified entropy \eqref{eq:modifiedEntropie},  in fact,  they also imply a convergence of the Fisher Information. First, let us highlight a link between  the proof of \eqref{eq:couplage_para_HMC} or of similar results for kinetic processes (e.g. in \cite{Dalalyan2018OnSF} for the kinetic Langevin diffusion or \cite{Doucet} for the continuous-time Randomized HMC) and our proof of Theorem~\ref{thm:dissipation_modified} (or Villani's modified entropy method in \cite{Villani2009}), which is the use of modified Euclidean norms. Indeed, \eqref{eq:couplage_para_HMC} (similarly to the results of \cite{Dalalyan2018OnSF,Doucet}) is established by proving that $\|z_{1}- z_{1}'\|_M \leqslant  ( 1 - m/(12\overline{\gamma}) t ) \|z_0-z_0'\|_M$ with $\|z\|_M^2 = z \cdot M z$ for a suitable matrix $M$, and then the result is obtained using the equivalence between the norms. This is related to the use of a mixed gradient in the Fisher information term of the modified entropy \eqref{eq:modifiedEntropie}. Indeed, we could use a Fisher information term involving $\na h \cdot \tilde M^{-1} \na h$ for some $\tilde M$ (see Remark~\ref{rem:convex}) and recover a contraction of this term at the same rate as in the parallel coupling. We will not detail this   (and rather use directly the results already established with parallel couplings) and refer the interested reader to \cite{MonmarcheContraction} for more details on this connection in the case of continuous-time diffusion processes (see also \cite{Doucet} where two proofs are given, one for a $\mathcal W_2$ convergence using a parallel coupling, and one in the Sobolev space $H^1$, in the spirit of Villani's method but without using the $L^2$ part of the norm, so that in fact the $H^1$ result could have been obtained from the $\mathcal W_2$ one as explained in the proof of Proposition~\ref{prop:convex} below).

In the next statement,  as in Section~\ref{subsc:G-entropy}, $G$ is a $\mathcal C^4$ convex function on $\R_+$ such that $1/G''$ is positive concave, and $\P_\theta$ is given by \eqref{eq:Ptheta} (here we do not assume that the support of $\theta$ is in $[0,1/4]$). When the integration time is random, distributed according to $\theta$, a parallel coupling of two chains is obtained by taking the same Gaussian variables in the step \eqref{eq:damping}  and the same random integration time $t\sim \theta$ in the Hamiltonian step \eqref{eq:HD2}.

\begin{proposition}\label{prop:convex} \

\begin{enumerate}
\item Assume that there exist $n\in\N,\kappa_n >0$ such that for all parallel coupling $(z_k)_{k\in\N}=(x_k,v_k)_{k\in\N}$ and $(z_k')_{n\in\N}=(x_k',v_k')_{k\in\N}$  associated with $\P_\theta$,  almost surely,
\begin{equation}\label{eq:couple1}
|z_n-z_n'|^2 \leqslant  \kappa_n |z_0-z_0'|^2\,.
\end{equation}
Then, for all  $\nu\in\mathcal P(\R^{2d})$,
\begin{equation}\label{eq:IGconvex}
\mathcal I_G(\nu \P_\theta^n|\mu) \leqslant \kappa_n \mathcal I_G(\nu|\mu)\,.
\end{equation}
\item Assume that there exist $n\in\N,\kappa_n >0$ such that for all parallel coupling $(z_k)_{k\in\N}=(x_k,v_k)_{k\in\N}$ and $(z_k')_{n\in\N}=(x_k',v_k')_{k\in\N}$  associated with $\P_\theta$, 
\begin{equation*}
\mathbb E\po |z_n-z_n'|^2\pf  \leqslant  \kappa_n \mathbb E\po |z_0-z_0'|^2 \pf \,.
\end{equation*}
Then, in the case where $G(u) = (u-1)^2$ for $u\in\R_+$,\eqref{eq:IGconvex} holds for all  $\nu\in \mathcal P(\R^{2d}) $.
\end{enumerate}
\end{proposition}

The proof is given in Section~\ref{subsec:convexcase}.

In the case of Gaussian target distribution, it holds $z_{n}- z_{n}' = A^n(z_0-z_0')$ for all $n\in \N$ for some matrix $A$, so that \eqref{eq:couple1} simply holds with $\kappa_n = |A^n|$, and then  the whole analysis of \cite{MonmarcheHMCconvexe} in this case for the  Wasserstein distance  can be transfered to the Fisher Information and relative entropy via Proposition~\ref{prop:convex}.

 The continuous-time Randomized HMC studied in \cite{Doucet} (where, among other things, a convergence rate of order $\sqrt{m}$ for Gaussian targets is established) does not enter exactly the framework of Proposition~\ref{prop:convex} but the proof of the second statement is easily adapted (which is how   \cite[Theorem 5]{Doucet}  could be obtained from \cite[Theorem 3]{Doucet}).

}

\section{Some estimates on the Hamiltonian dynamics}\label{sec:prelimHamilton}

We gather in this section a series of bounds on the Hamiltonian dynamics which will prove useful in the rest of the analysis.

Denoting by $R(x,v)=(x,-v)$ the reflection of the velocity (which is a linear involution) it holds $\Phi_t^{-1} = R \circ \Phi_t \circ R$. In particular \new{$\na ( \Phi_t^{-1}) = R[(\na \Phi_t)\circ R]R$ (writing $R$ both as a function on $\R^{2d}$ and a matrix)}, and thus $|\det (\na \Phi_t^{-1})|=1$. We denote by $\Phi_t = (\Phi_t^1,\Phi_t^2)$ the two $d$-dimensional components of the flow. \new{Recall that our notation is $\na \Phi = (\partial_{z_i} \Phi_j)_{i\in\cco 1,n\ccf,j\in\cco 1,m\ccf}$ ($i$ stands for the \new{row} and $j$ the column).}

\begin{lemma}\label{lem:JacobPhi}
Under Assumption~\ref{hyp:main} with $L=1$, for all  $z\in\R^{2d}$,
\begin{equation}\label{eq:lem1_Et}
\left|\na \Phi_t(z)  -  E_t  \right| \ \leqslant \ 
  \frac{t^3}6 e^{t} \,, 
\end{equation}
 where,  denoting $(x_s,v_s)=\Phi_s(z)$ for $s\geqslant 0$,
 \[E_t = \begin{pmatrix}
1 - \int_0^t(t-s)\na^2 U(x_s)\dd s & -\int_0^t \na^2 U(x_s)\dd s \\
t & 1 - \int_0^t s \na^2 U(x_s)\dd s
\end{pmatrix}\,.\]
In particular, for all  $z,z'\in\R^d$, with $z=(x,v)$, $z'=(x',v')$,
\begin{eqnarray}
|\na_x \Phi_t^1(z)|&   \leqslant &   1 + \frac{t^2}{2} + \frac{t^3}6 e^{t}  \label{eq:lem1_na_xPhi1}     \\
|\na_v \Phi_t^1(z)| &  \leqslant &    t + \frac{t^3}6 e^{t}  \label{eq:lem1_na_vPhi1}\\
|\na_x \Phi_t^2(z)| &   \leqslant &   t + \frac{t^3}6 e^{t}    \label{eq:lem1_na_xPhi2}\\
|\na_v \Phi_t^2(z)| &   \leqslant &    1+ \frac{t^2}2 + \frac{t^3}6 e^{t}  \label{eq:lem1_na_vPhi2}
\end{eqnarray} 
and   
\begin{eqnarray}
|\Phi_t^1(z) - \Phi_t^1(z')| & \leqslant &  \po 1 + \frac{t^2}{2} + \frac{t^3}6 e^{t}  \pf |x-x'| + \po t + \frac{t^3}6 e^{t}\pf |v-v'|   \label{eq:lem1_Phi1z-z'} \\
|\Phi_t^2(z)-\Phi_t^2(z')| 
& \leqslant  &  \po t + \frac{t^3}6 e^{t}\pf |x-x'| + \po 1 + \frac{t^2}{2} + \frac{t^3}6 e^{t}  \pf |v-v'|\,.   \label{eq:lem1_Phi2z-z'}
\end{eqnarray}
If, moreover, Assumption~\ref{hyp:Frobenius} holds then, for all $z,z'\in\R^{2d}$, considering $E_t'$ defined as $E_t$ but with $x_s$ replaced by $\Phi_s^1(z')$,
\begin{equation*}
\|\na \Phi_t(z) - E_t - \new{(\na \Phi_t(z') - E_t')}\|_F \leqslant \frac{7}{20} L_H t^3 \po 2|x-x'| + t |v-v'|\pf\,.
\end{equation*}
In particular,
\begin{eqnarray}
  \|\na_x \Phi_t^1(z) - \na_x\Phi_t^1(z')\|_F  &\leqslant&  \frac{7}{10} t^2 L_H |x-x'| + \frac{21}{40} t^3 L_H |v-v'| \label{eq:lem1_Frob_naxPhi1}\\
  \|\na_v \Phi_t^1(z) - \na_v\Phi_t^1(z')\|_F  &\leqslant& \frac{7}{10} t^3 L_H |x-x'| + \frac{7}{20} t^4 L_H |v-v'| \label{eq:lem1_Frob_navPhi1} \\
  \|\na_x \Phi_t^2(z) - \na_x\Phi_t^2(z')\|_F  &\leqslant&  \new{\frac{21}{20} t L_H |x-x'| + \frac{21}{40} t^2 L_H |v-v'|} 
   \label{eq:lem1_Frob_naxPhi2}  \\
  \|\na_v \Phi_t^2(z) - \na_v\Phi_t^2(z')\|_F  &\leqslant&  \frac{21}{40} t^2 L_H  |x-x'| + \frac{11}{30} t^3 L_H  |v-v'|  \label{eq:lem1_Frob_navPhi2} \,.
\end{eqnarray}
\end{lemma}
To simplify high order terms in $t$ in the inequalities \eqref{eq:lem1_na_xPhi1}-\eqref{eq:lem1_Phi2z-z'}, we will often use the following bounds for $t\leqslant 1/4$:
\begin{equation}\label{eq:simplifier_t}
 1 + \frac{t^2}{2} + \frac{t^3}6 e^{t}    \leqslant   \frac{16}{15}\,,\qquad \frac{t^2}{2} + \frac{t^3}6 e^{t}\leqslant \frac{3}{5}t^2\,,  \qquad t + \frac{t^3}6 e^{t}   \leqslant \frac{71}{70} t\,. 
\end{equation}

\begin{proof}
  Differentiating the ODE $\partial_t \Phi_t(z) = F\po \Phi_t(z)\pf$ with $F(x,v)=(v,-\na U(x))$, we get that, for a fixed $z\in\R^d$, $t\mapsto \na \Phi_t(z)$ solves the matrix-valued time-inhomogeneous linear ODE
  \begin{equation}\label{ODEPhi}
  \partial_t \na\Phi_t(z) =\na \Phi_t(z)
 \na F(z_t)
\,,\qquad \Phi_0(z) = I_{2d}\,,\quad\text{with}\quad \na F(z_t) = \begin{pmatrix}
0 &   \  - \na^2 U(x_t) \\ 1    & 0
\end{pmatrix}
  \end{equation}
and $z_t=(x_t,v_t)=\Phi_t(z)$. Besides, $E_0=I_{2d}$ and
\begin{eqnarray*}
\partial_t E_t &=& \begin{pmatrix}
 -\int_0^t \na^2 U(x_s)\dd s & \ -\na^2 U(x_t) \\
1 & \ \new{-}t \na^2 U(x_t)
\end{pmatrix} \\
& = & \begin{pmatrix}
1   & \ -\int_0^t \na^2 U(x_s)\dd s \\
t & 1 
\end{pmatrix} \na F(z_t) \ = \ E_t \na F(z_t) + R_t
\end{eqnarray*}
with
\begin{eqnarray*}
R_t &=& \begin{pmatrix}
 \int_0^t(t-s)\na^2 U(x_s)\dd s & 0 \\
0 &  \int_0^t s \na^2 U(x_s)\dd s
\end{pmatrix} \na F(z_t)\\
& =&
\begin{pmatrix}
0 &  -\na^2 U(x_t) \int_0^t(t-s)\na^2 U(x_s)\dd s  \\
  \int_0^t s \na^2 U(x_s)\dd s  & 0
\end{pmatrix} \,,
\end{eqnarray*}
which can be bounded as $|R_t|\leqslant  t^2/2$. Then, for all $t\geqslant 0$,
\begin{eqnarray*}
|\na \Phi_t(z) - E_t | & \leqslant &  \int_0^t\po  |\po  \na \Phi_s(z) - E_s\pf \na F(z_s)| + |R_s|\pf  \dd s \ \leqslant \  \frac{t^3}{6} + \int_0^t |\na \Phi_s(z) - E_s | \dd s
\end{eqnarray*}
and the Grönwall's Lemma concludes the proof of the first part. The inequalities \eqref{eq:lem1_na_xPhi1} \eqref{eq:lem1_na_vPhi1}, \eqref{eq:lem1_na_xPhi2}, \eqref{eq:lem1_na_vPhi2}, \eqref{eq:lem1_Phi1z-z'}, \eqref{eq:lem1_Phi2z-z'} are then straightforward corollaries.

The second part of the lemma is similar. Recall that $\|AB\|_F \leqslant  \min(\|A\|_F |B|,|A|\|B\|_F)$ for all matrices $A,B$. 
    For $z,z'\in\R^{2d}$, writing $z_t'=(x_t',v_t')=\Phi_t(z')$ and defining $E_t'$ and $R_t'$ as $E_t$ and $R_t$ except that $x_t$ is replaced by $x_t'$, let
\[A_t = \na \Phi_t(z) - E_t - \new{\po \na \Phi_t(z') - E_t'\pf}\,.\]
Then
\begin{eqnarray*}
\|A_t\|_F & = &  \left\|\int_{0}^t \co A_s \na F(x_s) + (\na\Phi_s(z') - E_s') \po \na F(x_s')-\na F(x_s)\pf + R_s - R_s'\cf  \dd s\right\|_F\\
& \leqslant & \int_{0}^t \co \| A_s \|_F | \na F(x_s)| + |\na\Phi_s(z') - E_s'| \| \na F(x_s')-\na F(x_s)\|_{F} + \| R_s - R_s'\|_{F}\cf  \dd s\,.
\end{eqnarray*}
Using \eqref{ass:na2UFrob} and \eqref{eq:lem1_Phi1z-z'}, we bound 
\begin{eqnarray*}
\|\na F(x_t') - \na F(x_t)\|_F & =& \new{\|\na^2 U(x_t') - \na^2 U(x_t)\|_F}\\
 & \leqslant & 
  \po 1 + \frac{t^2}{2} + \frac{t^3}6 e^{t}  \pf L_H  |x-x'| + \po t + \frac{t^3}6 e^{t}\pf L_H  |v-v'| \\
\left\| \int_0^t s \po \na^2 U(x_s) - \na^2 U(x_s')\pf \dd s \right\|_F & \leqslant &   \po \frac{t^2}2+ \frac{t^4}{8} + \frac{t^5}{30} e^t\pf L_H |x-x'| + \po \frac{t^3}3 + \frac{t^5}{30} e^{t}\pf L_H  |v-v'| 
\end{eqnarray*} 
and
\begin{eqnarray*}
\lefteqn{\left\| \int_0^t(t-s)\po \na^2 U(x_t)\na^2 U(x_s) - \na^2 U(x_t')\na^2 U(x_s')\pf\dd s\right\|_F}\\
& \leqslant &   L_H|x-x'| \int_0^t(t-s) \po 2+ \frac{t^2}2 + \frac{t^3}{6} e^t + \frac{s^2}2 + \frac{s^3}{6} e^s\pf \dd s \\
& & +\ L_H|v-v'| \int_0^t(t-s) \po t+ \frac{t^3}{6} e^t + s + \frac{s^3}{6} e^s\pf \dd s \\
& \leqslant &    L_H|x-x'|\po t^2  + \frac{7t^4}{24}  + \frac{t^5}6 e^t  \pf  + L_H|v-v'|\po  \frac{2t^3}3   + \frac{t^5}6 e^t  \pf \,.
\end{eqnarray*}
to end up, \new{using moreover \eqref{eq:lem1_Et} to bound $|\na\Phi_s(z') - E_s'| $ and that  $|\na F(x_s)|\leqslant 1$,} with   
\begin{multline*}
\|A_t\|_F \ \leqslant \ \int_{0}^t    \| A_s \|_F\dd s +L_H |v-v'| \int_0^t\co     \frac{s^3}6 e^{ s}  \po s + \frac{s^3}{6}e^s \pf +  s^3 + \frac{s^5}{5}e^s    \cf  \dd s \\
    +L_H |x-x'| \int_0^t\co     \frac{s^3}6 e^{ s}  \po 1+ \frac{s^2}{2} + \frac{s^3}{6}e^s \pf +  \frac{3 s^2}2+ \frac{5s^4}{12} + \frac{7 s^5}{30} e^s    \cf  \dd s\,.
\end{multline*}
The Grönwall's Lemma then yields
\[\|A_t\|_F \leqslant k_1(t) L_H|x-x'| + k_2(t) L_H|v-v'|\]
where, using that $t\leqslant 1/4$,
\[\begin{array}{rclcl}
k_1(t) & =&  e^t \po     e^{ t}  \po  \frac{t^4}{24} + \frac{t^6}{72} + \frac{t^7}{252}e^t \pf +  \frac{t^3}2+ \frac{t^5}{12} + \frac{7 t^6}{180} e^t    \pf & \leqslant & \frac{7}{10} t^3 \\
k_2(t) &=& e^t \po     e^{ t}  \po \frac{t^5}{30}  + \frac{t^7}{252}e^t \pf +  \frac{t^4}4 + \frac{t^6}{30}e^s    \pf  & \leqslant & \frac{7}{20} t^4 \,.
\end{array}\]
The other inequalities then follow from the previous bounds (simplified by the fact $t\leqslant 1/4$) and, using \eqref{eq:lem1_Phi1z-z'}, 
\begin{eqnarray}
\lefteqn{\left\| \int_0^t (t-s) \po \na^2 U(x_s) - \na^2 U(x_s')\pf \dd s \right\|_F} \nonumber\\
 & \leqslant &   \po \frac{t^2}2+ \frac{t^4}{24}   + \frac{t^5}{120} e^t\pf L_H |x-x'| + \po  \frac{t^3}{6} + \frac{t^5}{120} e^{t}\pf L_H  |v-v'| \nonumber \\
 & \leqslant & \frac{21}{40} t^2 L_H |x-x'| + \frac{7}{40} t^3 L_H |v-v'|\,, \label{eq:naU(x_s)}
\end{eqnarray} 
\new{and, similarly,
\[
 \left\| \int_0^t  \po \na^2 U(x_s) - \na^2 U(x_s')\pf \dd s \right\|_F 
  \leqslant    \po t + \frac{t^3}{6}   + \frac{t^4}{24} e^t\pf L_H |x-x'| + \po  \frac{t^2}{2} + \frac{t^4}{24} e^{t}\pf L_H  |v-v'|  \,.
\]
}
 
\end{proof}

\begin{lemma}\label{lem:Hxx'}
Under Assumptions~\ref{hyp:main} and \ref{hyp:Frobenius} with $L=1$, the following holds. There exists a function $K\in\mathcal C^1 (\R^{2d},\R^d)$ such that, for all $x,x'\in\R^{d}$, $v=K(x,x')$ is the unique solution of $\Phi_t^1(x,K(x,x'))=x'$.  For $u_0,u_1\in\R^d$, consider the function $K_{u_0,u_1}\in\mathcal C^1(\R^{2d},\R^d)$ given by $K_{u_0,u_1}(x,v) = K(x+u_0,\Phi_t^1(x,v)+u_1)$, i.e. such that $v'=K_{u_0,u_1}(x,v)$  is the unique solution of  $\Phi_t^1(x+u_0,v')=\Phi_t^1(x,v)+u_1$. Then, for all  $x,v,u_0,u_1\in\R^d$, 
\begin{eqnarray}
|u_1-u_0+tv -t K_{u_0,u_1}(x,v)| &\leqslant&   \frac{3}{5} t^2 |u_0| + \frac15 t^2 |u_1-u_0|  \label{lem:Hxx'_eq1}\\
 \| \na_v K_{u_0,u_1}(x,v) - I_d \|_F & \leqslant & \frac{1}{5} t^2 L_H |u_0| + \frac{1}{10}t^3 L_H|u_1-u_0|  \label{lem:Hxx'_eq2}  \\
  \new{| \na_v K_{u_0,u_1}(x,v) - I_d |} & \new{\leqslant} & \new{\frac{1}{2} t^2 } \label{lem:Hxx'_eq3}  \\
 \| \na_x K_{u_0,u_1}(x,v) \|_F & \leqslant &  \frac35 t L_H |u_0| + \frac15 t^2 L_H |u_1-u_0|\label{lem:Hxx'_eq4}\\ 
  \new{| \na_x K_{u_0,u_1}(x,v) |} &\new{ \leqslant} & \new{ \frac65 t}  \,. \label{lem:Hxx'_eq5}
\end{eqnarray}
\end{lemma}
\begin{proof}
The existence and smoothness of $K$ can be established as in \cite{BouRabeeEberle} (recall that $t\sqrt{L} \leqslant 1/4$ under Assumption~\ref{hyp:main}), or directly by letting the step-size vanish in the results of \cite{BouRabeeEberle}.

Fix $x,v,u_0,u_1\in\R^{d}$. First, writing $v'=K_{u_0,u_1}(x,v)$, $x_s=\Phi_s^1(x,v)$ and $x_s'=\Phi_s^1(x+u_0,v')$ for $s\in[0,t]$, using that
 \begin{equation}\label{eq:Kuu}
 \begin{array}{lccl}
 & x_t +  u_1 &=& x + u_1 + t v - \int_0^t (t-s) \na U \po x_s \pf \dd s\\
= & x_t' &=& x + u_0 + t v' - \int_0^t (t-s) \na U \po x_s'\pf \dd s\,,
 \end{array} 
 \end{equation}
and  a computation similar to \eqref{eq:naU(x_s)} but using that $|\na U(x)-\na U(x')|\leqslant |x-x'|$ instead of $\|\na^2 U(x)-\na^2 U(x')\|_F \leqslant L_H|x-x'|$,  we get
\begin{eqnarray*}
t|v-v'| &\leqslant& |u_1-u_0|+   \frac{21}{40} t^2  |u_0| + \frac{7}{40} t^3  |v-v'|\,,
\end{eqnarray*}
 from which, using that $t\leqslant 1/4$ \new{to absorb the last term of the right hand side in the left hand side,}
 \begin{equation}\label{eq:Hxx'v-v}
 t|v' - v| \leqslant \frac{640}{633} |u_1-u_0| +  \frac{112}{211} t^2 |u_0|\,.  
 \end{equation}
 Reinjecting this in \new{the difference between the two lines of} \eqref{eq:Kuu},   we get
 \begin{eqnarray*}
 |u_1-u_0+tv - tv'|
  & \leqslant & \frac{21}{40} t^2  |u_0| + \frac{7}{40} t^3  |v-v'| \ \leqslant\  \frac{3}{5} t^2 |u_0| + \frac15 t^2 |u_1-u_0|\,.
\end{eqnarray*}  

Second,  writing $g_s(x,v) =  \na U \po \Phi_s^1 (x,v)\pf$, differentiating \new{the difference between the two lines of} \eqref{eq:Kuu} with respect to $v$ reads
\begin{eqnarray}\label{eq:navv'}
t\po \na_v v' - I_d \pf & =  & \new{-} \int_0^t (t-s)  \co \na_v g_s(x,v) - \na_v v' \na_v g_s(x+u_0,v') \cf \dd s \,.
\end{eqnarray}
Since $\na_v g_s(x,v) = \na_v \new{\Phi_s^1 (x,v)} \na^2 U(x_s)$, we get from \eqref{eq:lem1_na_vPhi1}, \eqref{eq:lem1_Phi1z-z'}, \eqref{eq:lem1_Frob_navPhi1} and Assumption~\ref{hyp:Frobenius} that, for all $x,v\in\R^{d}$ and $s\leqslant 1/4$,  $|\na_v g_s(x,v)| \leqslant s + s^36 e^s/6$  and
\begin{eqnarray*}
\lefteqn{\|\na_v g_s(x,v) - \na_v g_s(x+u_0,v')\|_F}\\
 & \leqslant & \|\na_v \new{\Phi_s^1 (x,v)} -\na_v \new{\Phi_s^1 (x',v')}\|_{F} |\na^2 U(x_s)|   + \po s + \frac{s^3}6 e^s\pf \|\na^2 U(x_s)-\na^2 U(x_s')\|_F \\
 & \leqslant &  \frac{7}{10} s^3 L_H |x-x'| + \frac{7}{20} s^4 L_H |v-v'| \\
 & & +\  L_H  \po s + \frac{s^3}6 e^s\pf \po \po 1 + \frac{s^2}{2} + \frac{s^3}6 e^{s}  \pf |x-x'| + \po s + \frac{s^3}6 e^{s}\pf |v-v'| \pf\\
 & \leqslant & \frac{10}{9} s L_H |u_0| + \frac{16}{15}s^2 L_H|u_1-u_0|\,,
\end{eqnarray*}
where we used \eqref{eq:Hxx'v-v} in the last line and that $s\leqslant 1/4$. Hence, from \eqref{eq:navv'}, we bound
\begin{eqnarray*}
t\| \na_v v' - I_d \|_F 
 &\leqslant&      \int_0^t(t-s)   \co \|\na_v g_s(x,v) - \na_v g_s(x+u_0,v')\|_F + \|\na v'-I_d\|_F |\na_v g_s(x+u_0,v')|  \cf \dd s \,, 
 \\
&\leqslant &\frac{5}{27} t^3 L_H |u_0| + \frac{4}{45}t^4 L_H|u_1-u_0| +   \po \frac{t^3}{6} + \frac{t^5e^t}{120} \pf  \| \na_v v' - I_d \|_F \,.
\end{eqnarray*}
 Using that $t\leqslant 1/4$, this yields \eqref{lem:Hxx'_eq2}. 
\new{Alternatively, to bound $  \na_v v' - I_d $ starting from \eqref{eq:navv'}, we can also simply use that $\|\na_v g_s\|_\infty \leqslant s  + s^3 e^{ s}/6 \leqslant 71 s/70$ to get 
\begin{eqnarray*}
t|\na_v v' - I_d | & \leqslant   & \int_0^t (t-s)  \co |\na_v g_s(x,v)| + (1+ |I_d - \na_v v'|) | \na_v g_s(x+u_0,v')| \cf \dd s \\
& \leqslant & \frac{71 t^3}{320} (2+ |I_d - \na_v v'|) \,,
\end{eqnarray*}
which, using that $t\leqslant 1/4$, gives \eqref{lem:Hxx'_eq3}.
}

Finally,  differentiating \eqref{eq:Kuu} with respect to $x$ leads to
\begin{equation*}\label{eq:naxv'}
t \na_x v' = \new{-} \int_0^t (t-s)  \co \na_x g_s(x,v) - \na_x g_s(x+u_0,v') -  \na_x v'  \na_v g_s(x+u_0,v') \cf \dd s \,. 
\end{equation*}
From \eqref{eq:lem1_na_xPhi1} and \eqref{eq:lem1_Frob_naxPhi1},
\begin{eqnarray*}
\lefteqn{\|\na_x g_s(x,v) - \na_x g_s(x+u_0,v')\|_F}\\
 & \leqslant & \|\na_x x_s -\na_x x_s'\|_{F} |\na^2 U(x_s)|   + \po 1+ \frac{s^2}2  + \frac{s^3}6 e^{s} \pf  \|\na^2 U(x_s)-\na^2 U(x_s')\|_F \\
& \leqslant &  \frac{7}{10} s^2 L_H |x-x'| + \frac{21}{40} s^3 L_H |v-v'| \\
 & & +\  L_H  \po 1+ \frac{s^2}2  + \frac{s^3}6 e^{s} \pf \po \po 1 + \frac{s^2}{2} + \frac{s^3}6 e^{s}  \pf |x-x'| + \po s + \frac{s^3}6 e^{s}\pf |v-v'| \pf\\
 & \leqslant & \frac{9}{8}L_H |u_0| + \frac{13}{12} s |u_1-u_0|
\end{eqnarray*}
where we used \eqref{eq:Hxx'v-v} in the last line and that $s\leqslant 1/4$. This yields
\[ t \|  \na_x v' \|_F \leqslant   \frac9{16} t^2 L_H   |u_0|  + \frac{13}{72} t^3     L_H |u_1-u_0| + \|\na_x v'\|_F    \po \frac{t^3}{6} + \frac{t^5e^t}{120} \pf  \,, \]
and then \eqref{lem:Hxx'_eq4} using that $t\leqslant 1/4$. 
\new{
As in the case of the gradient in $v$, alternatively, from \eqref{eq:naxv'}, we can bound $\na_x v'$  
using only that $\|\na_v g_s\|_\infty  \leqslant 71 s/70$  and $\|\na_x g_s\|_\infty  \leqslant 1+s^2/2+s^3 e^s/6 \leqslant 16/15 $  (from \eqref{eq:lem1_na_xPhi1} with \eqref{eq:simplifier_t}) to get 
\begin{eqnarray*}
t|\na_x v' | 
& \leqslant & \int_0^t (t-s)  \co |\na_x g_s(x,v)| + |\na_x g_s(x+u_0,v')| +  |\na_x v'|  |\na_v g_s(x+u_0,v')| \cf \dd s  \\
& \leqslant &   \frac{16}{15} t^2 + \frac{11 t^3}{6} |\na_x v' |   \,,
\end{eqnarray*}
hence \eqref{lem:Hxx'_eq5}.
}
\end{proof}

\section{Proofs of the main results}\label{sec:proof}

\subsection{The modified entropy dissipation}\label{sec:proof-hypoco}

In this section, implicitly,  we only consider initial conditions $\nu \in \mathcal P(\R^{2d})$ with $\nu \ll \mu$ (the results being  trivial otherwise). Besides, by density, we can assume that $\dd \nu / \dd \mu$  is $\mathcal C^1$, Lipschitz, bounded and lower bounded by a positive constant, and then it is readily checked that this is propagated by the transitions.

\subsubsection{Preliminary considerations}\label{sec:preliminary}

If $\nu$ has a density $h$ with respect to $\mu$ then the density $h_1$ of $\nu \P$  is given by the fact for  $f\in L^2(\mu)$,
\[\int_{\R^{2d}} f h_{1} \dd \mu = \int_{\R^{2d}} \P f h \dd \mu \,,\]
in other words $h_{1} = \P^* h$ where $\Q^*$ denotes the dual of an operator $\Q$ in $L^2(\mu)$. The \new{randomization} part is self-adjoint, i.e. $\D^* = \D$. For the Hamiltonian part, using that $\mu$ is invariant by the Hamiltonian flow, a change of variable yields
\[\int_{\R^{2d}} \H f h \dd \mu =\int_{\R^{2d}}   f \circ \Phi_t  h \dd \mu =\int_{\R^{2d}}   f  h \circ \Phi_t^{-1} \dd \mu \,,  \]
i.e. $\H^* h = h \circ \Phi_t^{-1}$.

\new{
\begin{remark}\label{eq:non-reversible}
Since $\Phi_t^{-1} = R\circ \Phi_t \circ R$ with $R(x,v)=(x,-v)$, $\H^* = \mathcal V \H \mathcal V$ with $\mathcal V h(x,v)= h(x,-v)$. In the case $\eta=0$, the chain forgets its velocity at each step, and thus we can equivalently say that the chain of position $(x_n)_{n\in\N}$ is the first marginal of a chain with transition $\D\P\D$. Since $\D \mathcal V = \mathcal V = \D$ if $\eta=0$, we get in that case that $(\D\P\D)^* = \D\P\D$, i.e. the chain is reversible. On the contrary, when $\eta>0$, the chain is non-reversible, even if we consider the transition $\D^{1/2} \P\D^{1/2}$. In that case, using that $\mathcal V \D = \D\mathcal V$, we still get that $(\D^{1/2} \P\D^{1/2})^*=\mathcal V\D^{1/2} \P\D^{1/2}\mathcal V$, namely the chain is reversible up to velocity reflection.
\end{remark}
}
Let us now compute the evolution of the relative entropy and of Fisher-like terms along the \new{randomization} and Hamiltonian transitions. In the following, we will repeatedly use the fact that, for a Markov operator $\Q$, a positive \new{function} $h$ and a matrix $A$, Jensen's inequality implies
\begin{equation}\label{eq:JensenFisher0}
|\Q A \na  h|^2 \ \leqslant \ \Q  \po \frac{|A\na h|^2}h\pf   \Q h\qquad \text{i.e.} \qquad  \frac{|\Q A \na  h|^2 }{\Q h} \ \leqslant \ \Q  \po \frac{|A\na h|^2}h\pf \,,
\end{equation}
\new{so that,} if $\mu$ is invariant by $\Q$, integrating this equality reads
\begin{equation}\label{eq:JensenFisher}
\int_{\R^{2d}} \frac{|\Q A \na  h|^2 }{\Q h} \dd \mu  \ \leqslant \ \int_{\R^{2d}}    \frac{|A\na h|^2}h  \dd \mu \,. 
\end{equation}
In fact, for   $\Q=\H^*$, which is deterministic, this is an equality, i.e. a change of variable yields
\begin{equation}\label{eq:JensenFisher_H}
\int_{\R^{2d}} \frac{|\H^* A \na  h|^2 }{\H^* h} \dd \mu  \ = \ \int_{\R^{2d}}    \frac{|A\na h|^2}h  \dd \mu \,. 
\end{equation} 

First, \new{considering the Ornstein-Uhlenbeck generator $L_{OU} = -v\na_v +  \Delta_v$ on $\R^{2d}$, whose associated semi-group $e^{sL_{OU}}$ is given by
\begin{equation}\label{eq:semigroupOU}
e^{sL_{OU}} f(x,v) = \mathbb E \po f(x, e^{-s} v +\sqrt{1-e^{-2s}} G)\pf\,,\qquad G\sim \mathcal{N}(0,I_d)\,,
\end{equation}
we see that}, if $\eta>0$, then $\D = e^{t_*L_{OU}}$ with  $t_*=-\ln\eta$. \new{ As a consequence,}  by a classical computation \new{(see e.g. \cite[Lemma 7]{MonmarcheGamma})}
\begin{eqnarray*}
\int_{\R^{2d}} \D h \ln \D h  \dd \mu - \int_{\R^{2d}}   h \ln   h  \dd \mu & = & 
 \int_0^{t_*} \partial_s \int_{\R^{2d}}   e^{s L_{OU}} h \ln   \po e^{s L_{OU}} h\pf  \dd \mu \dd s \\
&=& -   \int_0^{t_*} \int_{\R^{2d}}  \frac{|\na_v e^{s L_{OU}} h|^2}{e^{s L_{OU}} h} \dd \mu  \dd s\,.
\end{eqnarray*}
Then, for $s\in [0,t_*]$, denoting $\overline{h}_s = e^{s L_{OU}} h$, using \eqref{eq:JensenFisher0} \new{ and that $\na_v e^{u L_{OU}} h = e^{-u} e^{u L_{OU}}\na_v h$ for all $u\geqslant 0$ (as can be seen by differentiating \eqref{eq:semigroupOU}), }
\[\frac{|\na_v \D h|^2}{\D h } \ = \ \frac{|\na_v e^{(t_*-s)L_{OU}}\overline{h}_s|^2}{e^{(t_*-s)L_{OU}}\overline{h}_s } \ \leqslant \ e^{-2(t_*-s)}e^{(t_*-s)L_{OU}}\po \frac{|\na_v \overline{h}_s|^2}{\overline{h}_s} \pf\,.  \]
Using that $\mu$ is invariant for $e^{(t_*-s)L_{OU}}$, we get
\begin{equation}\label{eq:evolutionEntropieD}
\int_{\R^{2d}} \D h \ln \D h  \dd \mu - \int_{\R^{2d}}   h \ln   h  \dd \mu  \ \leqslant \ - \frac{\eta^{-2}-1}{2} \int_{\R^{2d}}  \frac{|B_v \na \D h|^2}{\D h} \dd \mu\quad \text{with}\quad B_v = \begin{pmatrix}
0 & 0 \\ 0 & 1
\end{pmatrix}    \,.
\end{equation}
If $\eta=0$, then $\D h(x,v) = \mathbb E \po h(x,G)\pf $, in particular  $\na_v \D h=0$ and by Jensen's inequality we can still say that \eqref{eq:evolutionEntropieD} holds in the sense that
\begin{equation}\label{eq:evolutionEntropieD_eta0}
\int_{\R^{2d}} \D h \ln \D h  \dd \mu - \int_{\R^{2d}}   h \ln   h  \dd \mu  \ \leqslant \ 0 \new{\ = \  - \frac{\tilde \eta^{-2}-1}{2} \int_{\R^{2d}}  \frac{|B_v \na \D h|^2}{\D h}}  \,,
\end{equation}
\new{for any $\tilde \eta>0$.}

For the Hamiltonian step, by a change of variable and using that $\mu$ is invariant by $\Phi_t$,
\begin{equation}\label{eq:evolutionEntropieH}
\int_{\R^{2d}} \H^*  h \ln \H^* h \dd \mu = \int_{\R^{2d}}  h \ln  h \dd \mu\,. 
\end{equation}

We proceed with the analysis of the evolution of the Fisher-like terms  of the form 
\[\int_{\R^{2d}} \frac{|A\na h|^2}{h}\dd \mu\]
where $A$ is some $2d\times2d$ matrix. Starting again with the \new{randomization} part, and denoting simply by $\alpha$ a $d\times d$ block $\alpha I_d$ in a matrix when there is no ambiguity, we see that
\begin{equation}\label{eq:Beta}
\na \po \D h\pf \ = \ \D B_\eta  \na h\,,\qquad\text{with}\qquad B_\eta  = \begin{pmatrix}
1 & 0 \\ 0 & \eta
\end{pmatrix}\,,
\end{equation} 
and thus, using \eqref{eq:JensenFisher},
\begin{equation*}
\int_{\R^{2d}} \frac{|A\na  \D h|^2}{\D h}\dd \mu   \leqslant  \ \int_{\R^{2d}} \frac{| A B_\eta \na   h|^2}{ h}\dd \mu \,.
\end{equation*}

For the Hamiltonian part, 
\begin{equation}\label{eq:Ct}
\na \H^* h = \na \po h \circ \Phi_t^{-1}\pf = \H^* J_t \na h
\end{equation}
with $J_t = (\na \Phi_t^{-1}) \circ \Phi_t $, and thus, using   \eqref{eq:JensenFisher_H},
\begin{equation*}
\int_{\R^{2d}} \frac{|A\na  \H^* h|^2}{\H^* h}\dd \mu    = \int_{\R^{2d}} \frac{| A J_t\na   h|^2}{ h}\dd \mu \,.
\end{equation*} 

\subsubsection{Entropy dissipation}

\new{
Before proceeding with the proof of Theorem~\ref{thm:dissipation_modified}, let us recall the strategy of the proof of \cite[Theorem 28]{Villani2009} (or \cite[Theorem 9]{MonmarcheGamma}) in continuous-time, which will shed some light on the proof in our discrete-time case. Denoting by $(T_s)_{s\geqslant 0}$ the semi-group associated to the kinetic Langevin diffusion \eqref{eq:kinLangevin} and by $h_s$ the relative density of $\nu T_s$ with respect to $\mu$ for some initial condition $\nu$. By a classical computation, we get the following entropy dissipation:
\[\partial_s \mathrm{Ent}(\nu T_s|\mu) = -\gamma \int_{\R^d} \frac{|\na_v h_s|^2}{h_s} \dd \mu\,, \]
(notice that in our case we get a discrete-time analogous of this by applying successively  \eqref{eq:evolutionEntropieD} and  \eqref{eq:evolutionEntropieH}). In the overdamped case \eqref{eq:overdLangevin}, we would have the full gradient in the right hand side which, thanks to the LSI, immediately yields an exponential decay of the entropy at rate $C_{LS}$. In the kinetic case, the $\na_x h_s$ part is missing in the entropy dissipation, which is thus $0$ at time $s=0$ for instance if, initially, the velocity is at equilibrium, and thus we cannot conclude. This is solved in \cite{Villani2009} by seeing that, when $|\na^2 U|$ is bounded,
\[\partial_s \int_{\R^d} \frac{|\na_x h_s + \na_v h_s|^2}{h_s} \dd \mu \leqslant - c \int_{\R^d} \frac{|\na_x h_s|^2}{h_s} \dd \mu + C \int_{\R^d} \frac{|\na_v h_s|^2}{h_s} \dd \mu\,,\]
for some $c,C>0$. Hence, for $a>0$,
\[\partial_s \co  \mathrm{Ent}(\nu T_s|\mu) + a \int_{\R^d} \frac{|\na_x h_s + \na_v h_s|^2}{h_s} \dd \mu \cf \leqslant - \max\po ca, \gamma - Ca\pf  \int_{\R^d} \frac{|\na h_s|^2}{h_s} \dd \mu \,.\]
Now, taking $a< \gamma/C$ and assuming a LSI for $\mu$, the right-hand side controls the modified entropy appearing in the left-hand side since 
\[\mathrm{Ent}(\nu T_s|\mu) + a \int_{\R^d} \frac{|\na_x h_s + \na_v h_s|^2}{h_s} \dd \mu \leqslant (C_{LS} + 2 a)  \int_{\R^d} \frac{|\na h_s|^2}{h_s} \dd \mu \,, \]
and thus we get the exponential decay of the modified entropy. The next proof follows this structure.

}

\begin{proof}[Proof of Theorem~\ref{thm:dissipation_modified}]

We assume that $L=1$ and consider $\mathcal L$ of the form
 \[\mathcal L\po \nu \pf \ = \ \Ent(\nu|\mu) + a \int_{\R^{2d}} \frac{| A \na h|^2 }{h}\dd \mu\,,\qquad h=\frac{\dd \nu}{\dd\mu}\,, \qquad A = \frac{1}{\sqrt{2}} \begin{pmatrix}
 1 & \new{1} \\ \new{1} & 1
 \end{pmatrix}\,, \]
for some $a>0$. 
 We have already computed (see \eqref{eq:evolutionEntropieD} and  \eqref{eq:evolutionEntropieH}) that, for $\eta>0$,
\begin{eqnarray*}
\mathcal L(\nu \P ) & \leqslant &\Ent(\nu|\mu)  - \frac{\eta^{-2}-1}{2}  \int_{\R^{2d}}  \frac{|\na_v \D h|^2}{\D h} \dd \mu  + a \int_{\R^{2d}} \frac{|A\na \P^* h|^2}{\P^* h}\dd \mu\\
& \leqslant &\mathcal L(\nu) - a\int_{\R^{2d}}\frac{|A\na  h|^2}{ h}\dd \mu   - \frac{\eta^{-2}-1}{2}  \int_{\R^{2d}}  \frac{|\na_v \D h|^2}{\D h} \dd \mu  + a \int_{\R^{2d}} \frac{|A\na \P^* h|^2}{\P^* h}\dd \mu\,.
\end{eqnarray*}
\new{When $\eta=0$, thanks to \eqref{eq:evolutionEntropieD_eta0}, the same inequality holds with the term $\eta^{-2}$ replaced by $\tilde\eta^{-2}$ for any $\tilde \eta>0$. In the rest of the proof, $\tilde \eta >0$ is arbitrary in the case $\eta=0$ and is $\eta$ if $\eta>0$. } 
 Using \eqref{eq:JensenFisher}  with $\D$ and \eqref{eq:Beta}, 
\[\int_{\R^{2d}} \frac{| A \na h|^2 }{h}\dd \mu \geqslant  \int_{\R^{2d}} \frac{| A \D \na h|^2 }{\D h}\dd \mu  =  \int_{\R^{2d}} \frac{| A B_{\new{\tilde \eta}}^{-1} \na \D h|^2 }{\D h}\dd \mu\,, \]
\new{where, in the case $\eta=0$, we used that $\na_v \D h = 0$.}
Using    \eqref{eq:Ct} and \eqref{eq:JensenFisher_H},
 \[
 \int_{\R^{2d}} \frac{|A\na \P^* h|^2}{\P^* h}\dd \mu = \int_{\R^{2d}} \frac{|\H^* A J_t  \na \D  h|^2}{\P^* h}\dd \mu\ = \int_{\R^{2d}} \frac{| A J_t \na\D  h|^2}{\D h}\dd \mu\,.\] 
At this point,  we have established that
\begin{eqnarray*}
\mathcal L(\nu \P ) & \leqslant &\mathcal L(\nu)  -   a \int_{\R^{2d}} \frac{\na\D  h   \cdot S \na\D  h}{\D h}\dd \mu
\end{eqnarray*}
where,  decomposing  
\[J_t(z) = \begin{pmatrix}
p(z) & q(z) \\ r(z) & s(z)
\end{pmatrix}\,,\]
the matrix $S(z)$ is given by
\begin{eqnarray*}
S &=& \frac{\new{\tilde \eta}^{-2}-1}{2a} B_v +  B_\eta^{-1}A^T AB_\eta^{-1} -  (A J_t)^TA J_t   \\
&=& \begin{pmatrix}
 1 -(p\new{+}r)^2 & -\new{\tilde \eta}^{-1} +(p\new{+}r)(s\new{+}q)  \\
-\new{\tilde \eta}^{-1} +(p\new{+}r)(s\new{+}q)  & \  \new{\tilde \eta}^{-2} - (\new{s+q})^2  + \frac{\new{\tilde \eta}^{-2}-1}{2a}
 \end{pmatrix}  \,.
\end{eqnarray*}
\new{Using \eqref{eq:lem1_Et} with $J_t = R[(\na \Phi_t(R \Phi_t(z))]R $, we get that, for all  $z\in\R^{2d}$, 
\[|J_t(z)  -  \tilde E_t  | \leqslant t^3e^{t}/6\]
 with
  \[ \tilde E_t = \begin{pmatrix}
1 - \int_0^t(t-s)\na^2 U(\tilde x_s)\dd s & \int_0^t \na^2 U(\tilde x_s)\dd s \\
- t & 1 - \int_0^t s \na^2 U(\tilde x_s)\dd s
\end{pmatrix}\,,\]
where $\tilde x_s = \Phi_s^1 (\Phi_t^1(z),-\Phi_t^2(z))$ for $s\in[0,t]$.
} 
Hence, using that $t\leqslant1/4$ and $\|\na^2 U\|_\infty\leqslant 1$, for all $z\in\R^{2d}$,
\begin{equation}\label{eq:sistronglyconvex}
\max(|p(z) - 1|,|s(z)-1|) \leqslant \frac{t^2}2 + \frac{t^3}{4} \,, \quad |r(z)\new{+}t| \leqslant  \frac{t^3}{4}  \,,\quad |q(z)|  \leqslant t +  \frac{t^3}{4}  \,,
\end{equation}
in particular $0 \leqslant p\new{+}r\leqslant 1 - t + t^2/2 + t^3/2$ and $0\leqslant s\new{+}q \leqslant 1 + t + t^2/2 + t^3/2$ and thus
\begin{eqnarray*}
 1 -(p\new{+}r)^2 & \geqslant & 2t - \frac{21}{10} t^2 \\
  - (\new{s+q})^2 & \geqslant & -    1 - 2t -  \frac{31}{10} t^2\\
 |\new{\tilde \eta}^{-1}-(p\new{+}r)(s\new{+}q)| &\leqslant& |\new{\tilde \eta}^{-1}-1|+ |s\new{+}q-1| + (s\new{+}q)|p\new{+}r-1|\\
 & \leqslant & \new{\tilde \eta}^{-1} - 1 +   2 t + \frac{13}{5} t^2 \\
 & \new{=} & \new{\new{\tilde \eta}^{-1}\po 1 + \tilde \eta \po -1 +   2 t + \frac{13}{5} t^2\pf  \pf }\,.
\end{eqnarray*}
Using this bounds in the expression of $S$ and recalling the definition of $m_1,m_2,m_3,\rho$ in Theorem~\ref{thm:dissipation_modified}, we get that, 
  for any $\theta>0$ and all $z\in\R^{2d}$, \new{$u,v\in\R^d$
  \begin{eqnarray*}\label{eq:thetaS}
 \begin{pmatrix}
 u \\ v
 \end{pmatrix} \cdot 
  S(z)\begin{pmatrix}
 u \\ v
 \end{pmatrix}  &\geqslant& \begin{pmatrix}
 |u| \\ |v|
 \end{pmatrix} \cdot  
 \begin{pmatrix}
m_1  & - \frac{m_2}{\new{\tilde \eta}} \\
- \frac{m_2}{\new{\tilde \eta}} & \frac{m_3}{\new{\tilde \eta}^2} 
\end{pmatrix}
 \begin{pmatrix}
 |u| \\ |v|
 \end{pmatrix}\\
  & \geqslant & \begin{pmatrix}
 |u| \\ |v|
 \end{pmatrix} \cdot  
   \begin{pmatrix}
m_1 - \frac{\theta m_2}{\new{\tilde \eta}} & 0 \\
0 & \frac{m_3}{\new{\tilde \eta}^2} - \frac{m_2}{\theta \new{\tilde \eta}}
\end{pmatrix}\begin{pmatrix}
 |u| \\ |v|
 \end{pmatrix} \,.
  \end{eqnarray*}
  }
Choosing the positive $\theta$ which makes the two coefficients equal we get
\[\frac{m_3}{\new{\tilde \eta}^2}-\frac{m_2}{\theta \new{\tilde \eta}} = m_1-\frac{\theta m_2}{\new{\tilde \eta}} =\frac{\rho}{1-3t}=:\rho'\,,\]
which is positive by assumption on $a$. We have thus obtained that
\[
\mathcal L(\nu \P ) - \mathcal L(\nu)  \leqslant   -  a \rho'  \int_{\R^{2d}} \frac{|\na\D h|^2}{\D h}\dd \mu = -  a \rho'  \int_{\R^{2d}} \frac{|\H^* \na \new{\D} h|^2}{\P^* h}\dd \mu \] 
where we have used \eqref{eq:JensenFisher_H}. Then, since \new{$ \na \P^* h  =  \na \H^*\D h  = \na \Phi_t^{-1} \H^*\na \D  h $ (from \eqref{eq:Ct}) so that $\H^*\na \D  h = (\na \Phi_t^{-1})^{-1}\na \P^* h$}, it remains to use     to bound 
\begin{eqnarray*}
|\P^* \na h|^2 & \geqslant & |(\na \Phi_t^{-1})^{-1}\na \P^* h|^2 \\
& \geqslant &  (1-t) |\na \P^* h|^2 \new{ + \po 1 -  \frac{1}{t}\pf}  |(\na \Phi_t^{-1})^{-1}-I_{2d}|^2||\na \P^* h|^2\\
        & \geqslant  & \po 1 - 3 t\pf |\na \P^* h|^2\,, 
\end{eqnarray*}
where we used \new{\eqref{eq:lem1_Et}, that $t\leqslant 1/4$, $\na \Phi_t^{-1} = (\na \Phi_t) \circ R$ and that, for any matrix $M$ with $|M-I|=\varepsilon <1$, $|M^{-1}-1| \leqslant \varepsilon/(1-\varepsilon)$}. Finally, using the log-Sobolev inequality \eqref{eq:logSob_h_bar} and that $|A|^2=2$, we see that
\[\mathcal L(\nu \P) \leqslant \po \max(C_{LS},1) +2a \pf \mathcal I(\nu P) \leqslant  \frac{\max(C_{LS},1) +2a }{(1-3t)\rho' a}\po \mathcal L(\nu) - \mathcal L(\nu P)\pf \,.\]
 The proof is thus concluded, since this can be written as
\[\mathcal L(\nu \P) \leqslant  \po 1+ \frac{ \rho }{\max(C_{LS},1)/a +2 }\pf^{-1}  \mathcal L(\nu)   \,,\]
\new{and we can let $\new{\tilde \eta}$ vanish in the case $\eta=0$.}
 
To get the simpler rough bound stated at the end of Theorem~\ref{thm:dissipation_modified}, simply use that $t\leqslant 1$, $\eta\in[0,1]$ and $1-\eta = \gamma t$ to bound
 \[
m_1  \geqslant  t  \,,\qquad
m_2 \leqslant t(\gamma   +  3) \,,\qquad m_3 \geqslant  t\po -3  + \frac{\gamma }{2a}\pf \,.
\]
Then, in \eqref{eq:thetaS}, simply take $\theta  = \eta m_1/(2m_2)$, so that the first coefficient is larger than $m_1/2 \geqslant t/2$ and the second coefficient is larger than
\[t\po -3  + \frac{\gamma }{2a}\pf\eta^{-2} - \frac{2t^2(\gamma+3)^2}{t}\eta^{-2} \geqslant \frac t2\]
if $-3+\gamma/(2a)-2(\gamma+3)^2 \geqslant 1/2$. Conclusion follows from $1-3t \geqslant 1/4$. 
\end{proof}

\subsection{Wasserstein/entropy regularization}\label{subsec:proof-W2Ent}

This section is devoted to the proof of Theorem~\ref{thm:Wasserstein/entropie}.  Recall that $\Q=\D \H \D$. We start by proving the following.

\begin{proposition}\label{prop:densitéW2}
Under Assumptions~\ref{hyp:main} and \ref{hyp:Frobenius} with $L=1$, for all $z=(x,v),z'=(x',v')\in\R^{2d}$, $\delta_z \Q$ admits a smooth positive density and
\begin{equation}\label{eq:propdensite1}
\Ent(\delta_z \Q |\delta_{z'} \Q) \leqslant \hat c_1(t) \po  |x-x'|^2 + \eta^2 t^2 |v-v'|^2\pf \quad\text{with}\quad \hat c_1(t) = \frac{13}{2t^2(1-\eta^2)} +\new{5} L_H^2 t^4\,.
\end{equation}
Similarly, if $\eta>0$ and $t\leqslant 1/8$, for all $n\geqslant 2$,  
\begin{equation}\label{eq:propdensite2}
\Ent(\delta_z \Q^n |\delta_{z'} \Q^n) \leqslant \hat c_{n}(t)    \po |x-x'|^2 + (\eta n't)^2 |v-v'|^2\pf
\end{equation}
where, \new{writing $s  =  t \min(n , \lfloor 1/(4t)\rfloor )$, }
\[\hat c_{n}(t) = \hat c_*(s):=\frac{ s}{\gamma  } \po       \frac{12}{ \eta s^2}  + \frac{6\gamma}{\eta s } +              4\pf^2  + \new{132}   (L_H s)^2\,.\]
\end{proposition}

This result is interesting by itself, and calls for a few remarks.

\begin{itemize}
\item \new{As mentioned earlier,} a similar result is established along the proof of \cite[Lemma 16]{BouRabeeEberle} for the unadjusted HMC for $\eta=0$. In \cite{BouRabeeEberle}, only the case $n=1$ is considered (\emph{one-shot} coupling), but since in our case we are also interested in the Langevin regime $t\rightarrow 0$, $\eta = 1 - \gamma t$, we need to take $n$ of order $1/t$ so that the result do not degenerate at the limit. Besides, notice that, indeed, in our result, for $n=1$, if $\eta=0$ then the dependency on $|v-v'|$ vanishes, as it should, and we recover a result similar to the one of \cite{BouRabeeEberle}.
\item Contrary to Theorem~\ref{thm:Wasserstein/entropie} where, for simplicity, we take the distance $\mathcal W_{2}$ as a reference, in Proposition~\ref{prop:densitéW2} we have kept distinct the dependency in time of the contributions of $|x-x'|$ and $|v-v'|$. As expected in this kinetic settings, similarly to the hypoelliptic Langevin diffusion, for small $t$ (or $s$ in the second part) the scaling is $1/t^3$ (or $1/s^3$) for the position and $1/t$ (or $1/s$) for the velocity \new{(see e.g. \cite{GuillinWang,Herau2007} or \cite[Theorem 9]{MonmarcheGamma}).}
\item In view of the case $n=1$ and the proof, one can suspect that  the term $L_H^2 s^2$ in $c_*$ is not sharp and could maybe be replaced by $L_H^2 s^2 t^2$, which may be interesting in the Langevin scaling in cases where $L_H$ depends on the dimension, and would be consistent with  the results of \cite{GuillinWang} where $L_H$ is not involved for the Langevin diffusion. We do not address this question to alleviate the computations of the proof. 
\item In fact, it can be readily checked that the only information which is used on the flow $\Phi_t$ in the proof of Proposition~\ref{prop:densitéW2} are the estimates of Section~\ref{sec:prelimHamilton}. In particular, \new{the fact that $\mu$ is invariant by the Hamiltonian dynamics does not  intervene}. In particular, the conclusion of Proposition~\ref{prop:densitéW2}  holds if $\Phi_t$ is replaced by a function $\hat\Phi_t$ which satisfies estimates similar to those of Section~\ref{sec:prelimHamilton}. It is clear that this is the case of the Verlet Scheme, see in particular \cite{BouRabeeEberle} or \cite[Lemma 18]{MonmarcheHMCconvexe}. Thanks to Pinsker's inequality, we can thus use this result to get a total variation convergence from a Wasserstein convergence of the scheme. In particular, we can improve  \cite[Proposition 3]{MonmarcheSplitting} which is based on a one-shot coupling (i.e. $n=1$) for a splitting scheme of the Langevin diffusion and thus gives a poor bound as the step-size $\delta$ vanishes, by applying Proposition~\ref{prop:densitéW2} (with a Verlet scheme) with $n = 1/\delta$. This improves the dependency in $d$ of the efficiency bound in total variation given in \cite{MonmarcheSplitting}. Besides, combined with a $\mathcal W_1$ contraction, this also gives a local coupling condition for the splitting scheme in the sense that for all $R>0$ there exist $s$ such that for all $z,z'\in\R^{2d}$ with $|z|,|z'|\leqslant R$,
\[\|\new{\delta_{z}}\Q^{\lceil s/\delta\rceil} - \delta_{z'}\Q^{\lceil s/\delta\rceil}\|_{TV} \leqslant \frac12 \,,\]
where $s$ is explicit in terms of $R$, of the constants appearing in Proposition~\ref{prop:densitéW2} and of the $\mathcal W_1$ decay, as it is done in  \cite{BouRabeeEberle} with $\eta=0$ and $n=1$. Hence, for a particular splitting of the Langevin diffusion, this yields a result similar to \cite{DEMS} (with possibly slightly more explicit constants). This local coupling bound can then be transfered to the Metropolis-adjusted scheme, simply by using that there is a non-zero probability that all the moves of the Verlet scheme have been accepted during $\lceil s/\delta\rceil$ iterations (which may however be a rough way to estimate the convergence properties of the adjusted HMC).
\end{itemize}

\begin{proof}

\emph{\new{First part (proof of \eqref{eq:propdensite1})}.} For $z=(x,v) \in\R^d$, $\bG=(G,G')\in\R^{2d}$, let
\begin{equation}\label{eq:PsizG}
\Psi_{z}(\bG) = \begin{pmatrix}
1 &\  0 \\ 0 & \eta
\end{pmatrix} \Phi_t\po x , \eta v + \sqrt{1-\eta^2} G\pf + \sqrt{1-\eta^2} \begin{pmatrix}
0\\G'
\end{pmatrix}
\end{equation}
which is the state of the chain associated to the operator $\Q$, starting from $z$, after one transition, if $G$ and $G'$ are the variables used in the \new{randomization} steps. In particular, if $ \bG \sim \mathcal N(0,I_{2d})$ then the law of $\Psi_z(\bG)$ is $\delta_z \Q$.

Let us check that $\Psi_z$ is a diffeomorphism.  The equation $(X,Y) = \Psi_{z}(G,G')$ leads, considering first the velocities in this equality, to
\[G' = \frac{1}{\sqrt{1-\eta^2}} \po Y -  \eta \Phi_t^2\po x , \eta v + \sqrt{1-\eta^2} G\pf\pf \]
and, considering then the positions, to
\[X =   \Phi_t^1(x,\eta v + \sqrt{1-\eta^2} G)\,,\]
  in other words,
\[G =  \frac{1}{\sqrt{1-\eta^2}} \po K(x,X)- \eta v\pf\]
where $K$ is the function defined by the fact $\Phi^1_t(x,K(x,X)) =  X$, see Lemma~\ref{lem:Hxx'}.
This concludes since we have obtained that
\[\begin{pmatrix}
G \\ G' 
\end{pmatrix} = \Psi_z^{-1}\begin{pmatrix}
X\\ Y
\end{pmatrix} = \frac{1}{\sqrt{1-\eta^2}}  \begin{pmatrix}
  K(x,X)- \eta v \\
  Y -   \eta \Phi_t^2\po x , K(x,X) \pf 
\end{pmatrix}\,.\]
As a consequence, writing $j_z(y) = |\det(\na \Psi_{z}^{-1}(y)\new{)}|$, the law $\delta_z \Q$ admits a density $f_z$ given by
\[f_z(y) =  (2\pi)^{-d/2} j_z(y) e^{-|\Psi_z^{-1}(y)|^2\new{/2}}\,. \]
Denoting by $\rho$ the density of the standard Gaussian law $\mathcal N(0,I_{2d})$, 
\[\mathrm{Ent}(\delta_{z'} \Q |\delta_{z} \Q) = \int_{\R^{2d}} \ln \po \frac{f_{z'}}{f_{z}}\pf f_{z'} 
 =   \int_{\R^{2d}} \ln \po \frac{\rho}{g_{z,z'}}\pf \rho\]
with
\[g_{z,z'}(y) = \frac{f_{z}\po \Psi_{z'}(y)\pf \rho(y) }{f_{z'}\po \Psi_{z'}(y)\pf  } = |\det(\na  \Psi_{z'}(s))| f_{z}\po  \Psi_{z'}(s)\pf \]
which is the density of the image of $f_{z}$ by $\Psi_{z'}^{-1}$, namely is the image of $\rho$ by  $\Psi_{z'}^{-1}\circ \Psi_{z}=:\Psi_{z,z'}$. As established in the proof of \cite[Lemma 15]{BouRabeeEberle}, \new{ provided
\begin{equation}\label{eq:conditionPsi}
|\na \Psi_{z,z'}(\bG) - I_{2d}|\leqslant \frac12 \,,
\end{equation}
it holds
}
\begin{equation}\label{eq:entropie_kernel}
\int_{\R^{2d}} \ln \po \frac{\rho}{g_{z,z'}}\pf \rho \leqslant \mathbb E\po \frac12 |\Psi_{z,z'}(\bG)-\bG|^2 + \|\na \Psi_{z,z'}(\bG) - I_{2d}\|_F^2 \pf \,.
\end{equation}
 It remains to bound this expectation \new{and to establish \eqref{eq:conditionPsi}}.

 Since $\bW = (W,W'):= \Psi_{z,z'}(G,G')$ solves $\Psi_{z'}(W,W')=\Psi_{z}(G,G')$, we get that
\[W = \frac{1}{\sqrt{1-\eta^2}} \po H_{x,x'}\po \eta v+ \sqrt{1-\eta^2} G\pf - \eta v'\pf \]
where $H_{x,x'}$ is the function defined by the fact $\Phi_t^1(x',\new{H}_{x,x'}(v))=\Phi_t^1(x,v)$, i.e. $H_{x,x'}(v) = K_{x'-x,0}(x,v)$ with the notation of Lemma~\ref{lem:Hxx'}, 
and then
\begin{eqnarray*}
W' &=& G' + \frac{\eta}{\sqrt{1-\eta^2} } \co \Phi_t^2 \po x,\eta  v + \sqrt{1-\eta  ^2} G\pf - \Phi_t^2 \po x',\eta v' + \sqrt{1-\eta ^2} W\pf \cf \,.
\end{eqnarray*}
Using  \eqref{lem:Hxx'_eq1} and  that $t\leqslant 1/4$,
\begin{eqnarray}
|W-G| &=&  \frac{1}{\sqrt{1-\eta^2}} \left| H_{x,x'}\po \eta v+ \sqrt{1-\eta^2} G\pf - \eta v' - \sqrt{1-\eta^2} G\right|\nonumber \\
&=& \new{ \frac{1}{ \sqrt{1-\eta^2}} \left| K_{x-x',0}\po x, \eta v+ \sqrt{1-\eta^2} G\pf - \po \eta v + \sqrt{1-\eta^2} G\pf - \eta(v'-v)\right| } \nonumber \\
& \leqslant  &  \frac{1}{\sqrt{1-\eta^2}} \po  \po \frac 1t + \frac45 t\pf |x-x'| + \eta|v-v'|\pf \nonumber\\
& \leqslant  &  \frac{1}{\sqrt{1-\eta^2}} \po  \frac {21}{20t} |x-x'| + \eta|v-v'|\pf\,, \label{eq:G-W}
\end{eqnarray} 
and, thanks to \eqref{eq:lem1_Phi2z-z'} (with $t\leqslant 1/4$)
\begin{eqnarray}
|W' -G'| &=& 
 \frac{\eta }{\sqrt{1-\eta^2} }   \po \frac{15}{14}t   |x-x'| + \frac76  \left| \eta  v + \sqrt{1-\eta  ^2} G -\eta v' - \sqrt{1-\eta ^2} W  \right|\pf \nonumber\\
& \leqslant & \frac{\eta }{\sqrt{1-\eta^2} }   \po \frac{12}{7t}   |x-x'| + \frac73  \eta |v-v'|  \pf\,.\label{eq:G'-W'} 
\end{eqnarray}
Notice that 
\begin{equation}\label{eq:naG'W}
\na_{G'}W=0\,,\qquad \na_{G'} W'=I_d \,.
\end{equation}
Then, from \eqref{lem:Hxx'_eq2}, 
\begin{equation}\label{eq:naG-I}
\|\na_{G} W - I_{d}\|_F =\|\na  \new{K_{x-x',0}\po x, \eta  v+ \sqrt{1-\eta^2} G\pf} - I_d \|_F \leqslant \frac37 L_H t^2 |x-x'|\,,
\end{equation}
\new{while, using rather \eqref{lem:Hxx'_eq3},
\begin{equation}\label{eq:naG-I_pasF}
|\na_{G} W - I_{d}|   \leqslant   \frac12 t^2 \,.
\end{equation}
}
Then, we bound
\begin{eqnarray*}
\|\na_{G} W'\|_F & = & \eta   \left\| \na_v \Phi_t^2 \po x,\eta  v + \sqrt{1-\eta  ^2} G\pf -\na_{G}W \na_{v}\Phi_t^2 \po x',\eta v' + \sqrt{1-\eta ^2} W\pf   \right\|_F\\
& \leqslant &   \eta   \left\| \na_v \Phi_t^2 \po x,\eta  v + \sqrt{1-\eta  ^2} G\pf - \na_{v}\Phi_t^2 \po x',\eta v' + \sqrt{1-\eta ^2} W\pf   \right\|_F \\
& & \ + \|\na_G W-I_d\|_F |\na_{v}\Phi_t^2 \po x',\eta v' + \sqrt{1-\eta ^2} W\pf |\\
& \leqslant & \new{  \eta   \po \frac{21}{40} t^2 L_H  |x-x'| + \frac{11}{30} t^3 L_H  \po   \eta |v-v'| +  \sqrt{1-\eta^2}|G-W| \pf\pf }\\
&  & \new{  + \frac{16}{15} \|\na_G W-I_d\|_F\,,}   
\end{eqnarray*}
where we used \eqref{eq:lem1_na_vPhi2} (with \eqref{eq:simplifier_t}) and \eqref{eq:lem1_Frob_navPhi2}. 
\new{Hence, using the previous bounds \eqref{eq:G-W} on $|G-W|$ and  \eqref{eq:naG-I} on $\|\na_G W-I_d\|_F$,
\begin{eqnarray}
\|\na_{G} W'\|_F
 & \leqslant &   \eta   \po \frac{21}{40} t^2 L_H  |x-x'| + \frac{11}{30} t^3 L_H  \po  2 \eta |v-v'| +  \frac {21}{20t} |x-x'|  \pf\pf \nonumber\\
&  &    + \frac{16}{15} \times \frac37 L_H t^2 |x-x'|\nonumber\\
& \leqslant & \frac{7}{5} L_H t^2 |x-x'| + \frac{22}{30} t^3 \eta^2  L_H|v-v'|\,.\label{eq:naG'}
\end{eqnarray}

Alternatively, using \eqref{eq:lem1_Et} to see that 
\begin{equation}\label{eq:navPhi2z-z'}
| \na_v \Phi_t^2(z) -  \na_v \Phi_t^2(z')| \leqslant \frac{t^2}{2} + \frac{t^3}{6}e^t \leqslant \frac{3}{5}t^2
\end{equation}
for all $z,z'\in\R^{2d}$,
together with \eqref{eq:lem1_na_vPhi2} (with \eqref{eq:simplifier_t}) and \eqref{eq:naG-I_pasF}, we obtain
\begin{eqnarray}
|\na_{G} W'| 
& \leqslant &   \eta    \left| \na_v \Phi_t^2 \po x,\eta  v + \sqrt{1-\eta  ^2} G\pf- \na_{v}\Phi_t^2 \po x',\eta v' + \sqrt{1-\eta ^2} W\pf   \right|   \nonumber \\
& & \ + |\na_G W-I_d| |\na_{v}\Phi_t^2 \po x',\eta v' + \sqrt{1-\eta ^2} W\pf | \nonumber \\
& \leqslant & \frac{3}{5} \eta t^2  + \frac{8}{15} t^2 \ \leqslant \ \frac{17}{15}t^2\,. \label{eq:naGW'-pasF}  
\end{eqnarray}
}

\new{Gathering all these bounds concludes the proof of \eqref{eq:propdensite1}. Indeed, on the one hand, from \eqref{eq:naG'W}, \eqref{eq:naG-I_pasF}  and \eqref{eq:naGW'-pasF}, we get 
\[
|\na \Psi_{z,z'} - I_{2d}| \leqslant |\na_{G'} W| + |\na_{G} W - I_d| + |\na_{G'} W' - I_d| + |\na_G W'| \leqslant 2  t^2 \leqslant \frac{1}{8}\,,
\]
so that \eqref{eq:conditionPsi} holds. On the other hand, from \eqref{eq:G-W}, \eqref{eq:G'-W'}, \eqref{eq:naG'W}, \eqref{eq:naG-I} and \eqref{eq:naG'},
\begin{eqnarray*}
\lefteqn{\mathbb E\po \frac12 |\Psi_{z,z'}(\bG)-\bG|^2 + \|\na \Psi_{z,z'}(\bG) - I_{2d}\|_F^2 \pf}\\
 & \leqslant & \frac{1}{2(1-\eta^2)} \po  \frac {21}{20t} |x-x'| + \eta|v-v'|\pf^2
  + \frac{\eta^2 }{2(1-\eta^2) }   \po \frac{12}{7t}   |x-x'| + \frac73  \eta |v-v'|  \pf^2 \\
& & + \po \frac37 L_H t^2 |x-x'|\pf^2 + \po \frac{7}{5} L_H t^2 |x-x'| + \frac{22}{30} t^3 \eta^2  L_H|v-v'|\pf^2  \\
 & \leqslant & \po \frac{13}{2(1-\eta^2)t^2} + 5 L_H t^4 \pf     \po   |x-x'|^2 + \eta^2 t^2 |v-v'|^2\pf \,.
 \end{eqnarray*}
}

\bigskip

\emph{Second part \new{(proof of \eqref{eq:propdensite2})}. Step 1.}  We now turn to the $n$ steps case, assuming that $\eta>0$ and $t\leqslant 1/8$. Fix $z,z'\in\R^{2d}$. For $\bG=(\bG_1,\dots,\bG_n)$ i.i.d. standard Gaussian variables on $\R^{2d}$ (where we decompose $\bG_k = (G_k,G_k')\in\R^d\times\R^d$), denote by $\Psi_z^n(\bG)$ the state of a chain starting from $z$ after $n$ transitions, using the variables $(G_k,G_k')$ in the two \new{randomization} steps of the $k^{th}$ transition for $k\in\cco 1,n\ccf$. In other words, defining by induction $z_0=z$ and then $z_{k+1}=\Psi_{z_k}(\bG_{k+1})$, we have $\Psi_z^n(\bG)= z_{n}$.   Our goal is to define a function $\Psi_{z,z'}^n:\R^{2dn} \rightarrow \R^{2dn}$ in such  a way that $\bW = \Psi_{z,z'}^n(\bG)$ satisfies $\Psi_{z'}^n(\bW)=\Psi_z^n(\bG)$. There could be many ways to enforce this, for instance we could merge the two chains in one step, namely take $\bW_1=(W_1,W_1')$ as in the first part of the proof and then $\bW_k = \bG_k$ for all $k\geqslant 2$. However, for fixed $z,z'$, this would be a highly unlikely trajectory starting from $z'$ for small values of $t$, i.e. the law of $\bW_1$ would be far from a standard Gaussian law on $\R^{2d}$.

Let $y_0,\dots,y_n\in\R^{2d}$ be a fixed deterministic sequence which will be determined later on, with $y_0  = z'-z$ and $y_n=0$. We define the function $\Psi_{z,z'}^n$ by the fact $\bW = \Psi_{z,z'}^n(\bG)$ satisfies $\Psi_{z'}^k(\bW_1,\dots,\bW_k)=\Psi_z^k(\bG_1,\dots,\bG_k) + y_k$ for all $k\in\cco 1,n\ccf$.  In other words, $\bW=(\bW_1,\dots,\bW_n)$ is such that if two chains start respectively at $z$ and $z'$ and use respectively the variables $\bG_k=(G_k,G_k')$ and $\bW_k=(W_k,W_k')$ in the \new{randomization} steps of the $k^{th}$ transitions then after $k$ transitions the difference between the states of the two chains is $y_k$, for all $k\in\cco 0,n\ccf$.


Since $y_n=0$, this construction implies that $\Psi_{z'}^n(\bW) =\Psi_z^n(\bG)$ which, following the argument of the first part of the proof, implies that
\begin{equation}\label{eq:BorneGauss_n}
\Ent\po \delta_{z'} \Q^n |\delta_{z} \Q^n\pf  \leqslant  \mathbb E\po \frac12 |\Psi_{z,z'}^n(\bG)-\bG|^2 + \|\na \Psi_{z,z'}^n(\bG) - I_{2dn}\|_F^2 \pf\,,
\end{equation}
\new{provided
\begin{equation}\label{eq:condition_n}
|\na \Psi_{z,z'}^n(\bG) - I_{2dn}| \leqslant \frac12\,.
\end{equation}
It remains to bound the right hand side of \eqref{eq:BorneGauss_n} and to establish \eqref{eq:condition_n}.

In fact, in the case where $n\geqslant n_0 := \lfloor 1/(4t)\rfloor \geqslant 2$ (since $t\leqslant 1/8$), we choose $y_{n_0}=0$ (i.e. we merge the two chains in $n_0$ steps) and afterwards we take $\bW_k=\bG_k$ for all $k>n_0$ (so that $y_k=0$ for all $k>n_0)$. As a consequence, the indexes $k>n_0$ do not intervene in \eqref{eq:BorneGauss_n} and \eqref{eq:condition_n}. In other words, we have replaced the $n$-steps coupling by an $n_0$-steps coupling. Hence, without loss of generality, from now on we suppose that $n \in \cco 2, n_0\ccf $, and thus in particular $nt\leqslant 1/4$.
}

\medskip

\emph{Step 2.} In the rest of the proof, let $\bW=\Psi_{z,z'}^n(\bG)$ and, for conciseness, write $z_k'=\Psi_{z'}^k(\bW_1,\dots,\bW_k)$ the state of the a chain starting from $z'$ after $k$ iterations where the variables $\bW_j$ are used in the $j^{th}$ transition, $j\in\cco 1,k\ccf$. By design, $z_k' = z_k +y_k$. Write $z_k=(x_k,v_k)$, $z_k'=(x_k',v_k')$ and $y_k=(u_k,w_k)$.

Recall the notations of Lemma~\ref{lem:Hxx'}. Solving $\Psi_{z_k'}(W_{k+1},W_{k+1}') = \Psi_{z_k}(G_{k+1},G_{k+1}') + y_{k+1}$ yields, considering first the equality of the positions,
\begin{equation}\label{eq:expressionWk+1}
W_{k+1} = \frac{1}{\sqrt{1-\eta^2}} \po K_{u_k,u_{k+1}}\po x_k ,\eta v_k+ \sqrt{1-\eta^2} G_{k+1}\pf - \eta v_k'\pf
\end{equation}
and then, considering the equality of the velocities,
\begin{multline}\label{eq:expressionWk+1'}
W_{k+1}' = G_{k+1}'  \\
 +  \frac{1}{\sqrt{1-\eta^2} } \co  w_{k+1} + \eta \Phi_t^2 \po x_k,\eta  v_k + \sqrt{1-\eta  ^2} G_{k+1}\pf - \eta  \Phi_t^2 \po x_k',\eta v_k' + \sqrt{1-\eta ^2} W_{k+1}\pf \cf .
\end{multline}
 
First, thanks to \eqref{lem:Hxx'_eq1}, using that $1-\eta^2 \geqslant \gamma t$, $x_k'-x_k=u_k$ and $v_k'-v_k=w_k$,  
\begin{eqnarray}
\lefteqn{|W_{k+1}-G_{k+1}|}\nonumber\\
& \leqslant & \new{\frac{1}{\sqrt{\gamma t}} |K_{u_k,u_{k+1}}\po x_k ,\eta v_k+ \sqrt{1-\eta^2} G_{k+1}\pf -\eta v_k- \sqrt{1-\eta^2} G_{k+1} - \eta (v_k'-v_k)| } \nonumber\\
& \leqslant &  \frac{\sqrt{t}  }{5\sqrt{\gamma }} \po 3   |u_k| + |u_{k+1}-u_k|\pf   + \frac{1}{\sqrt{\gamma t^3 }}\left|u_{k+1}-u_k-\eta t w_k\right|\label{eq:Wk+1Gk+1}   \,.
\end{eqnarray}

 Second, using \new{that $\na U$ is $1$-Lipschitz, \eqref{eq:lem1_Phi1z-z'} and $t\leqslant1/4$, },
 \[|v_t - v - \new{(v_t' - v')}| \leqslant \int_0^t |\na U(x_s) - \na U(x_s')| \dd s \leqslant   \frac{19}{18} t |x-x'| +  \frac{15}{28} t^2 |v-v'| \]
 for any $(x_t,v_t)=\Phi_t(x,v)$ and $(\new{x_t'},v_t')=\Phi_t(x',v')$,  from which 
 \begin{eqnarray}
\lefteqn{|W_{k+1}' - G_{k+1}'|} \nonumber\\
 &=&   \frac{1}{\sqrt{1-\eta^2} } \left|  w_{k+1} + \eta \Phi_t^2 \po x_k,\eta  v_k + \sqrt{1-\eta  ^2} G_{k+1}\pf - \eta  \Phi_t^2 \po x_k',\eta v_k' + \sqrt{1-\eta ^2} W_{k+1}\pf \right| \nonumber \\
& \leqslant  &   \frac{1}{\sqrt{\gamma t} } \po |w_{k+1} - \eta^2 (v_k'-v_k)| + \eta   \frac{19}{18} t |x_k-x_k'| +  \frac{15}{28} t^2\eta^2 |v_k-v_k'|   \pf   +  \eta \po 1+ \frac{15}{28}t^2\pf   |G_{k+1}-W_{k+1}|\nonumber \\
& \leqslant  &   \frac{1}{\sqrt{\gamma t} } \po |w_{k+1} - w_k | +     \frac{19}{18} t |u_k| +  \po 2\gamma t +\frac{15}{28} t^2\eta^2 \pf  |w_k|  \pf    +    \eta \frac{127}{112}  |G_{k+1}-W_{k+1}|\,. \label{eq:Wk+1Gk+1bis}
\end{eqnarray}

\emph{Step 3.} Observing these first bounds  on $|W_{k+1}-G_{k+1}|$ and $|W_{k+1}' - G_{k+1}'|$, we can now fix $y_k=(u_k,w_k)$ for $k\in\cco 1,n-1\ccf$ to get a suitable scaling of the final estimate in the regime $t\rightarrow0$. We set 
\begin{eqnarray*}
w_k  &=&  \po 1 - \frac{k}{n}\pf \po v'-v\pf  - \frac{3k(n-k)}{(n^3 -n)\eta t } \po 2(x'-x)+ \eta t(n+1)(v'-v)\pf \\
&= & \po 1 - \frac{k}{n} - \frac{3k(n-k)}{n^2 -n }\pf \po v'-v\pf  - \frac{6k(n-k)}{(n^3 -n)\eta t } (x'-x)\,,
\end{eqnarray*}
which is designed so that
\[w_0 = v'-v\,,\qquad w_n= 0\,,\qquad \eta t \sum_{k=0}^{n} w_k =  x-x'\,.\]
Hence, setting  
\[u_k = x'-x + \eta t \sum_{j=0}^{k-1} w_j\,,\]
we get
\begin{equation}\label{eq:uk+1-uk}
u_0 = x'-x\,,\qquad u_n=0 \qquad\text{and}\qquad u_{k+1} = u_k + \eta t w_k\quad \forall k\in\cco 0,n-1\ccf\,.
\end{equation}
That way, the term of order $t^{-3/2}$ in \eqref{eq:Wk+1Gk+1} vanishes \new{(intuitively, this term is linked to unlikely variations of the position; our choice $u_{k+1} = u_k + \eta t w_k$ ensures that positions are simply driven by velocities, the latter being  directly controlled with some probability to follow the desired trajectory)}. Moreover, for all $k\in\cco 0,n\ccf$, using among other bounds that $k(n-k)/(n^2-n) \leqslant 1/2$ for all $n\geqslant 2$, $k\in\cco 0,n\ccf$,
\begin{equation}\label{eq:wkuk}
|w_k|\leqslant |v-v'| + \frac{3}{\eta n t } |x-x'|\,,\qquad |u_k|\leqslant |x-x'| + \eta n t |v-v'|
\end{equation}
and
\begin{eqnarray}
|w_{k+1}-w_k| & =  &  \left|\po \frac{1}{n} + \frac{3(n-2k-2)}{n^2 -n} \pf (v'-v) +  \frac{6(n-2k-2)}{(n^3 -n)\eta t }  (x'-x)  \right| \nonumber \\
&\leqslant & \frac{7}{n}|v-v'| + \frac{12}{n^2 \eta t}|x-x'|\,. \label{eq:w_k+1-wk}
\end{eqnarray}
With these choices, for $n\geqslant 2$, \new{\eqref{eq:Wk+1Gk+1} yields}
\begin{eqnarray}
|W_{k+1}-G_{k+1}|  
& \leqslant &  \frac{\sqrt{t}  }{5\sqrt{\gamma }} \co \po 3 + \frac3n \pf  |x-x'| +  \po  3 n +1 \pf \eta t |v-v'|\cf \nonumber \\
& \leqslant &  \frac{\sqrt{t}  }{10\sqrt{\gamma }} \co 9 |x-x'| +  7 \eta  nt |v-v'|\cf\,, \label{eq:Wk+1Gk+1Round2}
\end{eqnarray}
 and \new{\eqref{eq:Wk+1Gk+1bis} yields}
\begin{eqnarray*}
\lefteqn{|W_{k+1}'-G_{k+1}'|}\\
& \leqslant & \frac{\sqrt{t}}{\sqrt{\gamma } }\co     \frac{12}{n^2 \eta t^2} +  \frac{19}{18}  + \frac{3}{\eta n t }\po 2\gamma  +\frac{15}{28} t\eta^2   \pf  +    \frac{1143}{1120} \eta   \cf |x-x'| \\
& & +\  \frac{\sqrt{t}}{\sqrt{\gamma } }\co \frac7{nt} + \frac{19}{18}  \eta n t +  \po 2\gamma  +\frac{15}{28} t\eta^2 \pf  +    \frac{889}{1120} nt\eta   \cf  |v-v'| \\
& \leqslant & \frac{\sqrt{t}}{\sqrt{\gamma } }\co     \frac{12}{ \eta n^2t^2}  + \frac{6\gamma}{\eta n t } +              3   \cf |x-x'|  +   \frac{\sqrt{t}}{\sqrt{\gamma } }\co \frac7{\eta n^2t^2}   +  \frac{2\gamma}{\eta n t}+ \frac{11}5  \cf \eta n t  |v-v'| \\
\end{eqnarray*}
and then
\begin{eqnarray}
|\bG - \bW|^2  & = &  \sum_{k=1}^n \po  |W_k - G_k|^2 + |W_k'-G_k'|^2\pf \nonumber\\
&\leqslant &  \frac{ 2tn}{\gamma  } \po \frac{81}{100} + \co     \frac{12}{ \eta n^2t^2}  + \frac{6\gamma}{\eta n t } +              3   \cf^2\pf  |x-x'|^2\nonumber \\
& &+\  \frac{ 2tn}{\gamma  } \po \frac{49}{100} +\co \frac7{\eta n^2t^2}   +  \frac{2\gamma}{\eta n t}+ \frac{11}5  \cf ^2\pf \eta^2 n^2 t^2|v-v'|^2  \nonumber \\
& \leqslant & \frac{ 2tn}{\gamma  } \po       \frac{12}{ \eta n^2t^2}  + \frac{6\gamma}{\eta n t } +              4\pf^2  \po  |x-x'|^2 +\eta^2 n^2 t^2|v-v'|^2\pf   \label{eq:bGn}\,.
\end{eqnarray}

\new{
\emph{Step 4.} We now turn to the analysis of $\na_{\bG} \bW$. In this step, we focus on the operator norm in order to establish \eqref{eq:condition_n}. The study of the Frobenius norm to bound \eqref{eq:BorneGauss_n}, which follows similar computations,  will be addressed in Step 5 of the proof.

Recall the expressions \eqref{eq:expressionWk+1} and \eqref{eq:expressionWk+1'} for $W_k$ and $W_k'$.  First, for $k\geqslant 1$, 
\begin{equation}\label{eq:naGk'Wk'ok}
\na_{G_{k}'} W_{k}'=I_d\,,\qquad \na_{G_{k}'} W_{k}=0\,,
\end{equation}
and, using \eqref{lem:Hxx'_eq3}, 
\begin{equation}\label{eq:GkWk_bis}
|\na_{G_{k}} W_{k} - I_d|  \ = \   | \na_v K_{u_{k-1},u_{k}}\po x_{k-1},\eta v_{k-1}+ \sqrt{1-\eta^2} G_{k}\pf -   I_d| \ \leqslant \ \frac{t^2}{2} \, 
\end{equation}
 which, together with \eqref{eq:navPhi2z-z'} and \eqref{eq:lem1_na_vPhi2} (with \eqref{eq:simplifier_t}), gives 
\begin{eqnarray}
\lefteqn{|\na_{G_{k}} W_{k}'|}\nonumber\\
 &=&    \eta  \left |   \na_v \Phi_T^2 \po x_{k-1},\eta  v_{k-1} + \sqrt{1-\eta  ^2} G_{k}\pf -  \na_{G_{k}} W_{k} \na_v \Phi_T^2 \po x_{k-1}',\eta v_{k-1}' + \sqrt{1-\eta ^2} W_{k}\pf \right|\nonumber \\
& \leqslant & \frac{3}{5} t^2 +  \frac{16}{15} |\na_{G_{k}} W_{k} - I_d|  \ \leqslant \ \frac{17}{15} t^2\,.  \label{eq:GkWk'_bis} 
\end{eqnarray}

Since $\na_{\bG_j} \bW_k = 0$ if $j>k$, it remains to compute
\begin{equation}\label{eq:nabGjhop}
\na_{\bG_j} \bW_k = (\na_{\bG_j}z_{k-1}) \na_{z_{k-1}}\bW_k  
\end{equation}
 for $j<k$. On the one hand, using \eqref{lem:Hxx'_eq3}, \eqref{lem:Hxx'_eq5} 
 \begin{alignat}{4}
|\na_{v_{k-1}} W_k| &= \frac{\eta}{\sqrt{1-\eta^2}} \left| \na_v K_{u_{k-1},u_{k}}\po x_{k-1}, \eta v_{k-1}+ \sqrt{1-\eta^2} G_{k}\pf - I_d\right|  
& & \leqslant  \frac{t^{3/2}}{2\sqrt{\gamma}} \label{eq:navWk}\\
|\na_{x_{k-1}} W_k| &= \frac{1}{\sqrt{1-\eta^2}} \left|\na_x K_{u_{k-1},u_k}\po x_{k-1}, \eta v_{k-1}+ \sqrt{1-\eta^2} G_{k}\pf\right| 
& & \leqslant   \frac{6\sqrt{t}}{5\sqrt{\gamma}}\,.\label{eq:naxWk}
\end{alignat}
On the other hand, since $z_{j} = \Psi_{z_{j-1}}(\bG_{j})$, writing $A_j=\na_{z_{j-1}}\Psi_{z_{j-1}}(\bG_{j})$, for $j\leqslant k$,
\[\na_{\bG_j} z_{k} = \new{(} \na_{\bG_j} z_{j}\new{)} A_{j+1}\dots  A_{k} \,.\]
Differentiating \eqref{eq:PsizG} (recalling the notation $B_\eta$ from \eqref{eq:Beta}),
\begin{eqnarray*}
\na_G\Psi_{z}(G,G') &  =  & \sqrt{1-\eta^2} \na_v \Phi_t\po x , \eta v + \sqrt{1-\eta^2} G\pf B_\eta \\
\na_{G'}\Psi_{z}(G,G') &  =  &  \sqrt{1-\eta^2}  \begin{pmatrix}
 0 \\  I_d
\end{pmatrix}\\
\na_z\Psi_{z}(G,G') &  =  &  B_\eta \na \Phi_t\po x , \eta v + \sqrt{1-\eta^2} G\pf B_\eta  \,.
\end{eqnarray*}
Using \eqref{eq:lem1_Et} with $|\na^2 U|\leqslant 1$  and $t\leqslant 1/4$ yields   $|\na \Phi_t(z)|\leqslant 1+ t + t^2/2 + t^3 e^t/6 \leqslant 1+6 t /5$  
 and $|\na_v \Phi_t(z)|\leqslant 1+t^2/2 + t^3 e^t/6 \leqslant 16/15$ for all $z\in\R^{2d}$, so that
\begin{equation}\label{eq:bGj_bis}
|\na_{\bG_j} z_{k}| \leqslant \po1 + \frac65 t \pf^{k-j+1} \frac{16}{15}  \sqrt{1-\eta^2} \leqslant e^{6tn/5} \frac{16}{15} \sqrt{2\gamma t}\,,
\end{equation}
for $ 1\leqslant j\leqslant k\leqslant n$. Plugging this in \eqref{eq:nabGjhop} and then using \eqref{eq:navWk}, \eqref{eq:naxWk} yields, for $j<k$,
\[
|\na_{\bG_j} \bW_k| \ \leqslant \  e^{6tn/5} \frac{16}{15} \sqrt{2\gamma t}|\na_{z_{k-1}} \bW_k | \ \leqslant  \   e^{6tn/5} \frac{16}{15}   \sqrt{2 }\po \frac{t^2}{2} + \frac{6t}{5}\pf \leqslant  2 e^{6tn/5} t  \,.
\]

As a conclusion, this last inequality together with \eqref{eq:naGk'Wk'ok}, \eqref{eq:GkWk_bis}, \eqref{eq:GkWk'_bis} and the fact $\na_{\bG_j}\bW_k=0$ for $j>k$ yields   
\begin{eqnarray*}
| \na_{\bG} \bW - I_{2dn} |^2 & \leqslant  & \max_{k} \co |\na_{G_{k}} W_{k} - I_d|^2 + |\na_{G_{k}} W_{k}'|^2 +  \sum_{j=1}^{k-1} |\na_{\bG_j}\bW_k|^2\cf  \\
&\leqslant & \frac{t^4}{4} + \frac{17^2}{15^2} t^4 +  4 e^{12n t/5} n t ^2\,,
\end{eqnarray*}
which is less than $1/4$ if $nt \leqslant 1/4 $ (and thus $t\leqslant 1/8$), so that  \eqref{eq:condition_n} holds.
}

\medskip

\emph{Step 5.} \new{As announced above, the goal of this last step is to bound $\|\na_{\bG} \bW-I_{2dn}\|_F$.} Using now \eqref {lem:Hxx'_eq2}, and then \eqref{eq:uk+1-uk} and \eqref{eq:wkuk},
\begin{eqnarray}
\|\na_{G_{k}} W_{k} - I_d\|_F  &=&  \| \na_v K_{u_{k-1},u_{k}}\po x_{k-1},\eta v_{k-1}+ \sqrt{1-\eta^2} G_{k}\pf -   I_d\|_F \nonumber \\
&\leqslant & \frac{L_H t^2}{7} \po \frac15 |u_{k-1}|+\frac1{10} |u_k - u_{k-1}|\pf \nonumber  \\
& = & \frac{L_H t^2}{7} \po \po \frac15 + \frac3{10n}\pf |x-x'|+\eta t \po  \frac15 n+1\pf  |v-v'|\pf \nonumber\\
& = & L_H t^2 \po \frac12 |x-x'|+ \frac{7}{10}\eta nt |v-v'|\pf  \label{eq:GkWk-Id}
\end{eqnarray}
which, together with \eqref{eq:lem1_Frob_navPhi2} and \eqref{eq:lem1_na_vPhi2}   (with \eqref{eq:simplifier_t}), gives 
\begin{eqnarray}
\lefteqn{\|\na_{G_{k}} W_{k}'\|_F}\nonumber\\
 &=&    \eta  \left\|   \na_v \Phi_T^2 \po x_{k-1},\eta  v_{k-1} + \sqrt{1-\eta  ^2} G_{k}\pf -  \na_{G_{k}} W_{k} \na_v \Phi_T^2 \po x_{k-1}',\eta v_{k-1}' + \sqrt{1-\eta ^2} W_{k}\pf \right\|_F\nonumber \\
& \leqslant & \new{\frac{16}{15}} \eta  \|\na_{G_{k}} W_{k} - I_d\|_F +  \new{\frac{21}{40}}  L_H t^2 \eta |x_{k-1}-x_{k-1}'| \nonumber\\
& & \quad +\  \new{\frac{11}{30}}   L_H   t^3 |\eta  v_{k-1} + \sqrt{1-\eta  ^2} G_{k} - \eta v_{k-1}' - \sqrt{1-\eta ^2} W_{k}|\nonumber \\
& \leqslant &  \new{\frac{16}{15}} \eta  \|\na_{G_{k}} W_{k} - I_d\|_F +  \new{\frac{21}{40}} L_H t^2 \eta |u_{k-1}| +  \new{\frac{11}{30}} L_H   t^3\po  \eta|w_{k-1}| + \sqrt{2\gamma t} | G_{k}-W_{k}|\pf \nonumber\\
& \leqslant & \co  \new{\frac{16}{15}} \eta   \frac{L_H t^2}{2}  + \new{\frac{21}{40}} L_H t^2 \eta +  \new{\frac{11}{30}} L_H t^3 \po \frac{3}{ nt} + \frac{9\sqrt{2}  }{10} t\pf \cf  |x-x'|  \nonumber\\
& & + \co \new{\frac{16}{15}} \eta   \frac{7L_H t^2}{10}  +  \new{\frac{21}{40}} L_H t^2 \eta  +  \new{\frac{11}{30}}  L_H t^2 \po \frac{1}{n}     +   \frac{7 \sqrt{2}  }{10}  t \pf \cf \eta nt  |v-v'|  \nonumber\\
& \leqslant & \frac52 L_H t^2 |x-x'| + \frac32 L_H \eta nt^3  |v-v'|  \label{eq:GkWk'} \,,
\end{eqnarray}
\new{where we used \eqref{eq:wkuk} and \eqref{eq:Wk+1Gk+1Round2} to get the penultimate inequality (ant then $t\leqslant 1/8$, $n\geqslant 2$ in the last one).}

As in Step 4, since $\na_{\bG_j} \bW_k = 0$ if $j>k$, it now remains to bound 
\begin{equation}\label{eq:newnaG}
\new{\|\na_{\bG_j} \bW_k \|_F = \|\na_{\bG_j}(z_{k-1}) \na_{z_{k-1}}\bW_k\|_F \leqslant |\na_{\bG_j}(z_{k-1})| \| \na_{z_{k-1}}\bW_k\|_F } 
\end{equation}
 for $j<k$. We have already bounded $|\na_{\bG_j}(z_{k-1})|$ in \eqref{eq:bGj_bis}. From \eqref{lem:Hxx'_eq2} and \eqref{lem:Hxx'_eq4},
\begin{eqnarray*}
\|\na_{v_{k-1}} W_k\|_F &=& \frac{\eta}{\sqrt{1-\eta^2}} \left\| \na_v K_{u_{k-1},u_{k}}\po x_{k-1}, \eta v_{k-1}+ \sqrt{1-\eta^2} G_{k}\pf - I_d\right\|_F \\
& \leqslant & \frac{ L_H t^{3/2} }{\sqrt{ \gamma }} \po \new{\frac15} |u_{k-1}| + \new {\frac{t}{10}} |u_k-u_{k-1}| \pf\\
\|\na_{x_{k-1}} W_k\|_F &=& \frac{1}{\sqrt{1-\eta^2}} \|\na_x K_{u_{k-1},u_k}\po x_{k-1}, \eta v_{k-1}+ \sqrt{1-\eta^2} G_{k}\pf\|_F\\
& \leqslant & \frac{L_H \sqrt{t}}{\sqrt\gamma} \po \new{\frac{3}{5}} |u_{k-1}| + \new{\frac{t}{5}} |u_{k}-u_{k-1}|\pf\,,
\end{eqnarray*}
\new{from which, using \eqref{eq:uk+1-uk} and \eqref{eq:wkuk} (and $t\leqslant 1/8$, $n\geqslant 2$),
\begin{eqnarray}
\|\na_{z_{k-1}} W_k\|_F^2 &=& \|\na_{v_{k-1}} W_k\|_F^2 + \|\na_{v_{k-1}} W_k\|_F^2 \nonumber\\
& \leqslant & \frac{ L_H^2 t }{ \gamma }\co \po \new{\frac1{40}} |u_{k-1}| + \new {\frac{t}{80}} |u_k-u_{k-1}| \pf^2 + \po \new{\frac{3}{5}} |u_{k-1}| + \new{\frac{t}{5}} |u_{k}-u_{k-1}|\pf^2 \cf \nonumber \\
& \leqslant & \frac{ L_H^2 t }{2 \gamma }  \po  |x-x'| +  \eta nt |v-v'| \pf^2 \,.\label{eq:nazWk}
\end{eqnarray}
}
\new{Differentiating the expression \eqref{eq:expressionWk+1'} of $W_k'$ with respect to $z_{k-1}$ and using \eqref{eq:lem1_na_vPhi2} (with \eqref{eq:simplifier_t}),  \eqref{eq:lem1_Frob_naxPhi2} and  \eqref{eq:lem1_Frob_navPhi2} ,}
\begin{eqnarray*}
\lefteqn{\|\na_{z_{k-1}} W_k'\|_F}\\  
&=& \frac{\eta}{\sqrt{1-\eta^2} } \Big\|   B_\eta \na\Phi_t^2 \po x_{k-1},\eta  v_{k-1} + \sqrt{1-\eta  ^2} G_{k}\pf - B_\eta \na  \Phi_t^2 \po x_{k-1}',\eta v_{k-1}' + \sqrt{1-\eta ^2} W_{k}\pf \\
 &  & +\ \sqrt{1-\eta  ^2} \na_{z_{k-1}} W_{k} \na_v  \Phi_{\new{t}}^2 \po x_{k-1}',\eta v_{k-1}' + \sqrt{1-\eta ^2} W_{k}\pf \Big\|_F  \\
 & \leqslant & \new{ \frac{t L_H}{\sqrt{1-\eta^2}}\po  \frac{23}{20}  |u_{k-1}| + \frac{6}{10} t (\eta |w_{k-1}| + \sqrt{1-\eta^2}|G_k-W_k|)   \pf  + \frac{16}{15}\eta \|\na_{z_{k-1}} W_{k}\|_F  } \\
  & \leqslant & \new{ \frac{\sqrt{t} L_H }{\sqrt{\gamma }}\po  \frac{23}{20}  |u_{k-1}| + \frac{6}{10}  \eta t |w_{k-1}|   \pf + \frac{6}{10} t L_H |G_k-W_k|  + \frac{16}{15}\eta \|\na_{z_{k-1}} W_{k}\|_F  \,.}  
\end{eqnarray*}
From this, we use the bound \eqref{eq:Wk+1Gk+1Round2} on $|G_k-W_k|$,  \eqref{eq:nazWk} on  $\|\na_{z_{k-1}} W_{k}\|_F$  and \eqref{eq:wkuk} on $|u_{k-1}|,|w_{k-1}|$ (and that $t\leqslant 1/8$, $n\geqslant 2$ and $\eta\leqslant 1$ to simplify) to obtain 
\begin{equation*}
\|\na_{z_{k-1}} W_k'\|_F
  \ \leqslant\   3 \frac{\sqrt{t} L_H }{\sqrt{\gamma }}   \po |x-x'| + \eta n t  |v-v'|\pf     \,. 
\end{equation*}
Plugging this together with \eqref{eq:bGj_bis} and  \eqref{eq:nazWk}  in \eqref{eq:newnaG} yields
\begin{eqnarray*}
\|\na_{\bG_j} \bW_k \|_F & \leqslant &  |\na_{\bG_j}(z_{k-1})| \po \| \na_{z_{k-1}}W_k\|_F + \| \na_{z_{k-1}}W_k'\|_F\pf \\
& \leqslant &  6 e^{6tn/5}  L_H t    \po |x-x'| + \eta n t  |v-v'|\pf\,,
\end{eqnarray*}
for all $j<k$. We can now conclude, from this last inequality together with \eqref{eq:naGk'Wk'ok}, \eqref{eq:GkWk-Id}, \eqref{eq:GkWk'} and the fact $\na_{\bG_j}\bW_k=0$ for $j>k$, that 
\begin{eqnarray}
\| \na_{\bG} \bW- I_{2dn}\|_F^2 & = & \|\na_{G_k} W_k - I_d\|_F^2 + \|\na_{G_k} W_k'\|_F^2 + \sum_{k=1}^n \sum_{j=1}^{k-1} \|\na_{\bG_j}\bW_k\|^2_F \nonumber \\
& \leqslant & L_H^2   \co  \po \frac{49}{100} + \frac{25}{4}\pf t^4  + 36 e^{12 tn/5}   (nt)^2   \cf  \po  |x-x'|+ \eta nt |v-v'|\pf^2 \nonumber \\
&\leqslant & 132 (L_H nt)^2    \po |x-x'|^2 + (\eta nt)^2 |v-v'|^2\pf \,,\label{eq:conclusionbG}
\end{eqnarray}
where we used that $nt \leqslant 1/4$.  

\new{
\emph{Conclusion.} Since \eqref{eq:condition_n} holds, the proof is concluded by using  \eqref{eq:bGn} and \eqref{eq:conclusionbG} in \eqref{eq:BorneGauss_n}.
  }
\end{proof}

\begin{proof}[Proof of Theorem~\ref{thm:Wasserstein/entropie}]
 For a positive smooth bounded $h$ and  $z,z'\in\R^{2d}$, denoting by $f_{z,n}$ the density of $\delta_z \Q^n$, by Young's inequality, we obtain the following Harnack inequality:
\begin{eqnarray*}
\Q^n \ln h(z') 
& = & \int_{\R^{2d}} \ln h(y) \frac{f_{z',n}(y)}{f_{z,n}(y) }f_{z,n}(y)   \dd y \\
& \leqslant  &  \ln \int_{\R^{2d}}  h(y) f_{z,n}(y)  \dd y  + \int_{\R^{2d}} \ln \po \frac{f_{z',n}(y) }{f_{z,n}(y) } \pf f_{z',n}(y) \new{\dd y} \\
& = &  \ln \Q^n h(z) + \Ent(\delta_{z'} \Q^n |\delta_{z} \Q^n)\,.
\end{eqnarray*}
The last term is bounded by $c_n(\eta,t)|z-z'|^2$ thanks to Proposition~\ref{prop:densitéW2}. The rest of the proof follows \cite{RocknerWang}. Applying the previous inequality with $h$ replaced by $\dd (\nu\Q^n)/\dd \mu$, i.e. $h= (\Q^n)^* h_0$ with $h_0 = \dd \nu/\dd \mu$, 
\[\Q^n \ln (\Q^n)^* h_0 (z') \leqslant \ln \Q^n (\Q^n)^*  h_0(z) + c_n(t)|z-z'|^2\,.\]
Integrating this inequality with respect to $(z',z)\sim \pi$ where $\pi$ is a coupling of $\nu$ and $\mu$, we get
\[\Ent(\nu \Q^n|\mu)  \leqslant \int_{\R^{2d}} \ln \Q^n (\Q^n)^*  h_0  \dd \mu + c_n(t)\int_{\R^{2d}}|z-z'|^2\pi(\dd z',\dd z) \leqslant c_n(t)\int_{\R^{2d}}|z-z'|^2\pi(\dd z',\dd z)\,,\]
where we used Jensen's inequality. Taking the infimum over all couplings concludes the proof.

\end{proof}

\subsection{Adaptation of the proof for the two extensions}

\subsubsection{General entropies}\label{subsec:Gentropie}

 For general $G$-entropies, the function $\Psi(u,v)=G''(u) |v|^2$ being convex on $\R_+\times\R^d$, for any Markov operator $\Q$,
\[\Q\Psi(h,A\na h) \geqslant \Psi(\Q h,\Q A\na h)\,.\]
Integrating by $\mu$ if $\mu$ is $\Q$-invariant, we end up with 
\begin{equation}\label{eq:22avecG}
\int_{\R^{2d}} G''(\Q h) |\Q A\na h| ^2 \dd \mu \leqslant \int_{\R^{2d}} G''(h) |A\na h| ^2 \dd \mu\,,
\end{equation}
which generalizes \eqref{eq:JensenFisher}. The proof of Theorem~\ref{thm:dissipation_modified_G} is then exactly the proof of Theorem~\ref{thm:dissipation_modified} except that all quantities of the form $\int |B\na h|^2 /h \dd \mu$ are replaced by $\int G''(h) |B\na h|^2 \dd \mu$ and \eqref{eq:22avecG} is used whenever \eqref{eq:JensenFisher} was used. For the derivation of the $G$-entropy dissipation along the Ornstein-Uhlenbeck semi-group, see \cite{BolleyGentil} or \cite[Lemma 7]{MonmarcheGamma}.

\subsubsection{Random step-size}\label{subsec:random}

Reasoning as in Section~\ref{sec:preliminary}, denoting by $h$ and $h_1$ the respective density of $\nu$ and $\nu \P_\theta$ with respect to $\mu$, we see that 
\[h_1  = \P^*_\theta h = \mathbb E_{\theta} \po \H_T^*\D h\pf\,,\qquad T\sim \theta\,.\]
By the Jensen inequality,
\[\int_{\R^{2d}} G \po \mathbb E_\theta \po \H_T^*\D h\pf \pf  \dd \mu \leqslant \mathbb E_\theta \co  \int_{\R^{2d}}  G\po \H_T^*\D h\pf\dd \mu  \cf\,,\]
and, as in the previous section,
\[\int_{\R^{2d}} G''\po \mathbb E_\theta \po \H_T^*\D h\pf \pf |\mathbb E_\theta  A\na  \po \H_T^*\D h\pf|^2 \dd \mu \leqslant \mathbb E_\theta \co  \int_{\R^{2d}}  G''(\H_T^*\D h)|A\na \H_T^*\D h|^2\new{\dd \mu}\cf\,.\]
As a consequence, 
\[\mathcal L_G(\nu \P_\theta) \leqslant \mathbb E_\theta \co \mathcal L_G\po \nu \D\H_T\pf\cf\,.\]
It is thus sufficient to work conditionally to a fixed $T=t$, and applying Theorem~\ref{thm:dissipation_modified} yields Theorem~\ref{thm:dissipation_modified_G_random}.

\subsection{\new{The strongly convex case}}\label{subsec:convexcase}

\new{
Proposition~\ref{prop:convex} is a corollary of the following general result. In the next statement, $G$ and $\mathcal I_G$ are as in Section~\ref{subsc:G-entropy}.

\begin{proposition}\label{prop:general_convex}
Let $\mathcal R$ be a Markov transition operator on $\R^d$ and $\nu_*$ be an invariant measure of $\mathcal R$.  Denote by $\mathcal R^*$ the adjoint of $\mathcal R$ in $L^2(\nu_*)$ (which is a Markov transition operator).
\begin{enumerate}
\item  Assume that there exists $\kappa\geqslant0$ such that, for all $y,y'\in\R^d$, there exist a random variable $(Y,Y')$ with $Y \sim \delta_y \mathcal R^*$, $Y'\sim \delta_{y'} \mathcal R^*$ and, almost surely,
\begin{equation}\label{eq:YY'couple}
|Y-Y'|^2 \leqslant  \kappa |y-y'|^2\,.
\end{equation}
Then, for all  $\nu\in\mathcal P(\R^d)$,
\begin{equation}\label{eq:IGconvex_general}
\mathcal I_G(\nu \mathcal R|\nu_*) \leqslant \kappa \mathcal I_G(\nu|\nu_*)\,.
\end{equation}
\item  Assume that there exists $\kappa,\geqslant0$ such that, for all $y,y'\in\R^d$, there exist a random variable $(Y,Y')$ with $Y \sim \delta_y \mathcal R^*$, $Y'\sim \delta_{y'} \mathcal R^*$, and 
\begin{equation}\label{eq:YY'couple2}
\mathbb E\po |Y-Y'|^2 \pf  \leqslant  \kappa |y-y'|^2\,.
\end{equation}
Then, for all  $\nu\in\mathcal P(\R^d)$, \eqref{eq:IGconvex_general} holds in the particular case $G(u)=(u-1)^2$.
\end{enumerate}

\end{proposition}

\begin{proof}
First, from   \cite[Proposition 3.1]{Kuwada1}, \eqref{eq:YY'couple} implies that
\begin{equation}\label{eq:naRf}
 |\na \mathcal R^* f| \leqslant \sqrt{\kappa} \mathcal R^*|\na f|
\end{equation}
for all bounded Lipschitz functions $f$ on $\R^d$. For the reader's convenience, we recall the short proof of this. Fix  such $f$ and, for $r>0$ and $z\in\R^d$, let
\begin{equation}\label{eq:defGr}
G_r(z) \ = \ \sup_{y\in\mathcal B(z,r)\setminus\{z\}} \frac{|f(y)-f(z)|}{|y-z|}\,,
\end{equation}
so that for all $y,z\in\R^d$,
\[|y-z|\leqslant r \qquad \Rightarrow \qquad |f(y)-f(z)|\leqslant r G_r(z)\,.\]
Fix   $y,y'\in \R^d$ and consider $Y \sim \delta_y \mathcal R^*$, $Y'\sim \delta_{y'} \mathcal R^*$ such that \eqref{eq:YY'couple} holds almost surely. Then, we bound 
\begin{eqnarray*}
\left| \mathcal R^* f(y) - \mathcal R^* f(y')\right|  & \leqslant & \mathbb E \po |f(Y)-f(Y')|\pf \\ & \leqslant &   \sqrt{\kappa} |y-y'|\mathbb E \po G_{\sqrt{\kappa}|y-y'|}(Y)\pf \ = \ \sqrt{\kappa}|y-y'| \mathcal R^*\po  G_{\sqrt{\kappa}|y-y'|}\pf (y)\,.
\end{eqnarray*}
Since $f$ is Lipschitz, $G_r$ is uniformly bounded by  $\|\na f\|_\infty$, and by the dominated convergence theorem we get the convergence of  $\mathcal R^*\po  G_{\sqrt{\kappa}|y-y'|}\pf (y)$ toward $\mathcal R^*(|\na f|)(y)$ as $y'\rightarrow y$, which concludes the proof of \eqref{eq:naRf}.

Similarly, the weaker condition \eqref{eq:YY'couple2} gives 
\begin{equation}\label{eq:naRf2}
 |\na \mathcal R^* f|^2 \leqslant \kappa \mathcal R^*\po |\na f|^2\pf
\end{equation}
for all bounded Lipschitz functions $f$ on $\R^d$. Indeed, now, we use the Cauchy-Schwarz inequality to bound
\[
\left| \mathcal R^* f(y) - \mathcal R^* f(y')\right|^2   \leqslant  \mathbb E \po |Y-Y'|^2\pf \mathbb E \po G_{|Y-Y'|}^2(Y)\pf  
 \leqslant    \kappa |y-y'|^2\mathbb E \po G_{|Y-Y'|}^2(Y)\pf\,.\]
The condition \eqref{eq:YY'couple2} implies that $Y'\rightarrow Y$ almost surely as $y'\rightarrow y$, and thus we conclude as before to get \eqref{eq:naRf2}.

The conclusion then follows from Jensen's inequality. Indeed, assuming that $\nu\ll \nu_*$ (the result being trivial otherwise), writing $h=\dd \nu/\dd \nu_*$ so that $\mathcal R^* h$ is the density of $\nu \mathcal R$ with respect to $\nu_*$, using \eqref{eq:naRf}
\begin{eqnarray*}
\mathcal I_G(\nu \mathcal R|\nu_*) &=& \int_{\R^d} G''(\mathcal R^* h) |\na \mathcal R^* h|^2\dd \nu_* \\
&\leqslant & \kappa \int_{\R^d} G''(\mathcal R^* h) \po \mathcal R^*|\na   h|\pf^2\dd \nu_* \\
& \leqslant & \kappa \int_{\R^d} G''( h) |\na  h|^2\dd \nu_* \,.
\end{eqnarray*}
where we used Jensen's inequality and that $\nu_*$ is invariant by $\mathcal R^*$, as in \eqref{eq:22avecG}. Alternatively, if we only assume \eqref{eq:YY'couple2} and $G(u)=(u-1)^2$,
\begin{eqnarray*}
\mathcal I_G(\nu \mathcal R|\nu_*) &=& 2 \int_{\R^d}  |\na \mathcal R^* h|^2\dd \nu_* \\
& \leqslant & 2\kappa \int_{\R^d}   |\na   h|^2\dd \nu_* \\
& = & \kappa \int_{\R^d} G''( h) |\na  h|^2\dd \nu_* \,,
\end{eqnarray*}
where wed used \eqref{eq:naRf2} and that $\nu_*$ is invariant by $\mathcal R^*$.
\end{proof}


\begin{proof}[Proof of Proposition~\ref{prop:convex}]
The proposition is proven by applying Proposition~\ref{prop:general_convex} to $\mathcal R = \P_\theta^n$. Indeed, as seen in Section~\ref{sec:preliminary}, using the reversibility up to velocity reversal of the Hamiltonian dynamics, writing $\mathcal V$ the operator given by $\mathcal Vf(x,v)= f(x,-v)$, then $(\P_\theta^n)^* = (\H^* \D^*)^n = (\mathcal V\mathcal H\mathcal V \D)^n = \mathcal V (\mathcal H\D)^n \mathcal V$, where we used that $\mathcal V^2$ is the identity and that $\mathcal V \D\mathcal V = \D$. Let $z=(x,v),z'=(x',v')\in\R^{2d}$. In order to apply Proposition~\ref{prop:general_convex}, we have to construct a coupling of $\delta_z (\P_\theta^n)^*$ and $\delta_{z'}(\P_\theta^n)^*$ which satisfy either \eqref{eq:YY'couple} or \eqref{eq:YY'couple2} (the two proofs are similar, we only write the case of an almost sure contraction). To do so, we consider $Z=(X,V),Z'=(X',V')$ a coupling of $\delta_{(x,-v)} \P_\theta^n$ and $\delta_{(x',-v')}\P_\theta^n$ such that, almost surely,
\[|Z-Z'|^2 \leqslant \kappa_n |(x,-v)-(x',-v')|^2\,,\]
which is possible under the assumption of Proposition~\ref{prop:convex}(1).
 Then $(X,-V)$ and $(X',-V')$ form a coupling of  $\delta_z (\P_\theta^n)^*$ and $\delta_{z'}(\P_\theta^n)^*$. The fact that  $|(X,-V)-(X',-V')|=|Z-Z'|$ and $|(x,-v)-(x',-v')| = |z-z'|$ concludes the proof.
\end{proof}

\begin{remark}\label{rem:convex}
If, instead of working with the standard Euclidean norm, we were working in  the proof of Proposition~\ref{prop:convex} with $\|z\|_M = \sqrt{z\cdot M z}$ (as e.g. in the proof of \eqref{eq:couplage_para_HMC} in \cite{MonmarcheHMCconvexe}), then a contraction of $\|\cdot\|_M$ along the chain with transition $\P_\theta$ yields a contraction of $\|\cdot\|_{RMR}$ along the chain with transition $\P_\theta^*$, with $R$ the matrix corresponding to $(x,v)\mapsto (x,-v)$. This leads to gradient estimates of the form $\|\na \P_\theta^* h \|_{(RMR)^{-1}} \leqslant \P_{\theta}^*\| \na h\|_{(RMR)^{-1}} $ in the proof of Proposition~\ref{prop:general_convex} (and similarly with a square). Similarly, $\|\na \P_\theta h \|_{M^{-1}} \leqslant \P_{\theta}\| \na h\|_{M^{-1}} $. This is consistent with Villani's modified entropy based on a gradient term involving, on the one hand, $|\na_x h + \na_v h|^2$ when $h$ is the relative density of the law of the process (as in our case \eqref{eq:modifiedEntropie} or in \cite[Theorem 35]{Villani2009}) or $|\na_x f - \na_v f|^2$ when $f$ is a test function (as in \cite{MonmarcheGamma}).
\end{remark}
}

\section{Examples and applications}\label{sec:examples}

\subsection{Log-concave target measures}\label{sec:logconcave}

Many results are known for the idealized or unadjusted HMC chain or the Langevin diffusion when $U$ is strongly convex, see e.g. \cite{ChenVempala,CaoLuWang,MonmarcheSplitting,MonmarcheContraction,Dalalyan2018OnSF,Dwivedi1,Ma2,MangoubiSmith,Seiler2014PositiveCA} . However, our result easily applies to the case where $U$ is convex without being strongly convex. Indeed, it is known that all log-concave probability measures satisfy a Poincaré inequality (i.e. corresponding to $G(u)=(u-1)^2/2$ with the notations of Section~\ref{subsc:G-entropy}). Moreover, the KLS conjecture (for Kannan-Lov{\'a}sz-Simonovits) states that the Poincaré constant of all isotropic (i.e. centered with covariance matrix the identity) log-concave probability measure on $\R^d$ should be bounded by a universal constant (uniformly in $d$). This conjecture hasn't been established yet in its full generality (see \cite{LeeVempala} for a recent review) but it is known for spherically symmetric measures \cite{Bobkov2003,BonnefontJoulinMa}
 and very recent progresses \cite{ChenKLS,Lehec,Klartag} have established that if $\mu$ is an isotropic log-concave probability measure on $\R^d$ then it satisfies a Poincaré inequality with a constant $C_P \leqslant M \new{\ln d}$, where $M$ is a universal constant.
 
Hence, for a log-concave target, we can apply Theorem~\ref{thm:dissipation_modified_G} to get a convergence rate in the chi-square divergence with a very mild dependency in the dimension for the contraction rate.

To get a result in relative entropy, log-concavity alone is not sufficient since a log-Sobolev inequality requires a Gaussian tail. However, if $U$ is convex in $\R^d$ and strongly convex outside a ball, the  Poincaré inequality can be combined with a Lyapunov condition to get a log-Sobolev inequality. Indeed, \cite[Theorem 1.2]{CattiauxGuillinWu} (or more precisely here \cite[Theorem 3.15]{MenzSchlichting} with $\Omega=\R^d$  since the constants are explicit) states the following:

\begin{theorem}[\cite{CattiauxGuillinWu,MenzSchlichting}]\label{thmCGW-MS}
Assume that $\pi\propto e^{-U}$ satisfies a Poincaré inequality with constant $C_P$, and that there exist $K,\lambda,b\geqslant 0$  and a $\mathcal C^2$ function $W:\R^d \rightarrow [1,\infty)$ such that for all $x\in\R^d$, $\na^2 U(x) \geqslant - K$ and
\[-\na U(x)\cdot \na W(x) + \Delta W(x) \leqslant \po- \lambda |x|^2 + b \pf W(x)\,.\]
Then $\pi$ satisfies a LSI($C_{LS}$) with, writing $m_2=\int_{\R^d} |x|^2 \pi(\dd x)$,
\[C_{LS} \leqslant 2\sqrt{\frac1\lambda \po \frac12 + C_P(b+\lambda m_2)  \pf } + \frac{K\po 1+ 2C_P(b+\lambda m_2)\pf +4\lambda C_P}{2\lambda } \,.\]  
\end{theorem}

Without loss of generality we can always assume that $\pi$ is isotropic (up to multiplying by the norm of the covariance matrix at the end in the Poincaré or log Sobolev inequalities at the end), in which case $m_2=d$.

For instance, assuming that $\na U(x)\cdot x \geqslant  \rho |x|^2 -R $ for all $x\in\R^d$ for some $\rho,R>0$, taking $W(x) = e^{-\alpha|x|^2/2}$ yields
\[\frac{ -\na U(x)\cdot \na W(x) + \Delta W(x)}{W(x)} \leqslant - \rho\alpha |x|^2 + \alpha R + \alpha^2 |x|^2 + \alpha d\,.  \]
Taking $\alpha =\rho/2$ we can apply Theorem~\ref{thmCGW-MS} with $\lambda =3\rho^2/4$ and $b=\rho(R+d)/2$. If, for instance, $\rho$ is uniform in $d$ and $R$ is of order $d$ in high dimension, we get a LSI with a constant of order $\sqrt{d} (\ln d)^{5}$ if $\mu$ is isotropic and log-concave (since in that case $K=0$).

As a conclusion, for log-concave target measures which are strongly log-concave outside a ball with $\|\na^2 U\|_\infty <\infty$, Theorem~\ref{thm:dissipation_modified} yields a long-time convergence for the relative entropy with an explicit polynomial rate in the dimension (provided a polynomial dependency of $\|\na^2 U\|_\infty$, of  the covariance of the measure and of $\rho,R$ in the Lyapunov condition).

\subsection{Mean-field systems}\label{Sec:mean-field}

In this section, we consider the case where $d = N p$ for some $N,p\in \N$ and, decomposing $\mathbf{x}=(x_1,\dots,x_N) \in (\R^p)^N$, the target measure is $\pi_N \propto \exp (-\beta U_N)$ with $\beta>0$ and
\[U_N(\mathbf{x}) = \sum_{i=1}^N U(x_i) + \frac{1}{2N} \sum_{i\neq j} W(x_i-x_j)\,,\]
for some $U,W\in\mathcal C^2(\R^p)$, respectively called the confinement and interaction potential. We will not give detailed formal proofs in this section, since many arguments are classical or very similar to the case of the Langevin diffusion studied in \cite{MonmarcheMeanField,GuillinMonmarcheMF}. Given some fixed parameters $t,\eta$ independent from $N$, let $\P_N$ be the transition operator of the idealized HMC with potential $U_N$. We work under the following conditions.

\begin{assumption}\label{hyp:meanfield}\
The potentials $U$ and $W$ are $\mathcal C^\infty$  with all their derivatives of order larger than 2 bounded. There exist $c_U>0$, $c_U',c_W',R\geqslant 0$ and $c_W\in\R$  such that for all $x,y,z\in\R^d$,
\begin{eqnarray*}
\po \na U(x) - \na U(y) \pf \cdot   (x-y) & \geqslant & c_U |x-y|^2 - c_U' |x-y| \1_{\{|x-y|\leqslant R\}}\\
\po \na_x W(x,z) - \na_y W(y,z) \pf \cdot   (x-y) & \geqslant & c_W |x-y|^2 - c_W' |x-y| \1_{\{|x-y|\leqslant R\}}\,.
\end{eqnarray*}
 Moreover, $U$ is the sum of a strictly convex function and of a bounded function, $W$ is lower bounded, $c_U+c_W > \|\na^2_{x,x'} W\|_\infty$ and  $\beta <\beta_0$ where
 \[\beta_0 \  := \  \frac{4}{(c_U'+c_W')R} \ln\po \frac{c_U+c_W}{\|\na^2_{x,x'} W\|_\infty}\pf \qquad  (:=\ +\infty\text{ if }(c_U'+c_W')R=0).\]
Finally, $t\sqrt{\|\na^2 U\|_\infty + 2 \|\na^2 W\|_\infty} \leqslant 1/4$.
\end{assumption}

\new{These conditions are similar to those of \cite[Section 2.2]{BouRabeeSchuh}. } In particular, this holds if $U$ and $W$ are both strongly convex (with bounded Hessian), or if $U$ is quadratic at infinity and either $W$ is small or the temperature $\beta^{-1}$ is large enough. 

Since Assumption~\ref{hyp:meanfield} implies that $\|\na^2 U_N\|_\infty \leqslant \|\na^2 U\|_\infty + 2 \|\na^2 W\|_\infty$, it implies Assumption~\ref{hyp:main} for the HMC chain on $\R^{pN}$. Moreover, according to \cite[Theorem 8]{GuillinMonmarcheMF} (based on \cite[Theorem 8]{GuillinWuZhang}), under Assumption~\ref{hyp:meanfield}, there exists $\lambda>0$ such that $\mu_N = \pi_N \otimes \mathcal N(0,I_d)$ satisfies a LSI($\lambda$) for all $N\in\N$.  As a consequence, under this condition, Theorem~\ref{thm:dissipation_modified} holds and provide a contraction rate of the modified entropy independent from $N$. This is thus a case of dimension-free convergence rate. \new{Hence, for idealized HMC, we get a result similar to \cite[Theorem 3]{BouRabeeSchuh} but in relative entropy instead of $\mathcal W_1$ distance. }

The regularization result of Theorem~\ref{thm:Wasserstein/entropie} also scales well here. Indeed, under Assumption~\ref{hyp:meanfield}, $U_N$ satisfies Assumption~\ref{hyp:Frobenius} with $L_H \leqslant \sqrt{L_U^2 + 2L_W^2}$ which is uniform in $N$. Besides, Proposition~\ref{prop:densitéW2} is also interesting in this settings, and we recover a result similar to \cite{BouRabeeEberle}. 

\medskip

To go further, consider the case of independent initial conditions:
\begin{assumption}\label{hyp:MeanField_nu0}
The initial distribution of the chain is of the form $\nu_0 = \bar \nu_0^{\otimes \new{N}}$ where $\bar \nu_0 \in\mathcal P(\R^{2p})$ has a finite second moment and a density (still denoted $\bar \nu_0$) such that
\[\int_{\R^{2d}} \left| \na \ln \bar\nu_0 \right|^2 \bar\nu_0 < \infty\,.\]
\end{assumption}
From the uniform in $N$ log-Sobolev constant and bound on $\|\na^2 U_N\|_\infty$, is is straightforward to check that Assumptions~\ref{hyp:meanfield} and \ref{hyp:MeanField_nu0} implies that 
\begin{equation}\label{Eq:MFnu0}
\Ent(\nu_0|\mu_N) + \mathcal I(\nu_0|\nu_N) \leqslant CN
\end{equation}
for some $C>0$ independent from $N$.

Let us  write $\mathbf{Z}_n = (Z_{1,n},\dots,Z_{N,n})$ the state of the chain after $n$ transitions, with $Z_{i,n}=(X_{i,n},V_{i,n})$  the position of velocity of the $i^{th}$ particle. Assuming that the initial condition is $\nu = \bar \nu_0^{\otimes N}$ for some $\bar \nu_0 \in\mathcal P(\R^p)$, it is known that a propagation of chaos phenomenon occurs \cite{MonmarcheMeanField,GuillinMonmarcheMF}:  for a fixed $k\in\N$, the law of $(Z_{1,n},\dots,Z_{k,n})$ converges as $N\rightarrow \infty$ to $\bar \nu_n^{\otimes k}$ where $(\bar \nu_n)_{n\in\N}$ is given by
 \[\bar \nu_{n+1} = \H^{nl}_t \po \bar \nu_{n+1} \D_\eta\pf\,,\]
 where $\D_\eta$ is as in Section~\ref{sec:MainResults} but $\H^{nl}_t$ is the non-linear operator corresponding to the Vlasov equation, namely, for $\rho\in\mathcal P(\R^{2p})$, $\rho_t:= \H_t^{nl}(\rho)$ is the weak solution to
 \[\partial_t \rho_t(x,v) + v\cdot \na_x \rho_t(x,v) =  \na_v \cdot \co  (\na V + \na W\ast \rho_t)  \rho_t \cf (x,v)\,,\qquad \rho_0 = \rho \,.\]
 Here, we mean $\na W \ast \rho (x) = \int_{\R^{2d}} \na W(x-x') \rho(\dd x',\dd v')$.
 
 A probabilistic interpretation of this limit non-linear idealized HMC chain is given by the time-inhomogeneous Markov chain $(\bar Z_n)_{n\in\N}$ on $\R^{2p}$ whose transitions are given by an alternance of the velocity \new{randomization} step given by $\D_\eta$ and the integration for a time $t$ of the inhomogeneous Hamiltonian dynamics
 \[\dot X_t = V_t \qquad \dot V_t =  - \na U(X_t) - \na W\ast \rho_t (X_t) \,,\qquad \text{where}\qquad \rho_t =  \mathcal Law(X_t,V_t)\,.  \ \]

Considering a parallel coupling of the mean-field HMC $\mathbf{Z}=(Z_1,\dots,Z_N)$ and $N$ independent copies $\mathbf{\bar Z}=(\bar Z_1,\dots,\bar Z_N)$ of $\bar Z$ (i.e. using the same Gaussian variables for the two systems at each \new{randomization} step), it  is standard to show that under Assumption~\ref{hyp:meanfield} there exists $C$ (independent from $\eta,t,N)$ such that for all $N,n\in\N$, 
\[\mathbb E \po |\mathbf{Z}-\mathbf{\bar Z}|^2\pf \leqslant C e^{Cn t}\,, \] 
see e.g. the proof of \cite[Proposition 12]{MonmarcheMeanField}. Denoting by $\nu_n^{N,k}$ the law of $(Z_{1,n},\dots,Z_{k,n})$, this bound together with the interchangeability of the particles immediatly give the estimate
\[\mathcal W_2^2 \po \nu_n^{N,k},\bar \nu_n^{\otimes k}\pf \leqslant \frac{kC e^{Cnt}}N\]
for any $k\in\cco 1,N\ccf$.  From this, by following the proof of \cite[Theorem 10]{GuillinWuZhang}, we get that 
\begin{equation}\label{eq:MFN}
 \liminf_{N\rightarrow \infty} \frac1N \Ent(\nu_n | \mu_N) \geqslant  \mathcal H_W(\nu_n)\,, 
\end{equation}
where $\mathcal H_W$ is the so-called free energy, given by
\[\mathcal H_W(\nu) = E(\nu) - \inf_{\nu'\in\mathcal P(\R^{2p}} E(\nu')\]
where, considering the probability measure $\alpha \propto e^{-U(x)-|v|^2/2}$,
\[E(\nu) = \Ent\po \nu|\alpha\pf + \frac12 \int_{\R^{2p}} W(x-x') \nu(\dd x) \nu(\dd x')\,.\]
From \cite[Lemma 21]{GuillinWuZhang}, $E$ admits a unique minimizer over $\mathcal P(\R^{2p})$. Combining \eqref{Eq:MFnu0}, \eqref{eq:MFN} and Theorem~\ref{thm:dissipation_modified} finally yields the following:

\begin{proposition}
Under Assumptions~\ref{hyp:meanfield} and \ref{hyp:MeanField_nu0}, there exist $C,\kappa>0$ such that for all $n\in\N$,
\[\mathcal H_W \po \bar \nu_n\pf \leqslant e^{- \kappa n} C\,.\]
\end{proposition}

This convergence of the free energy to zero in turns implies the long-time convergence in $\mathcal W_2$ and in total variation of $\bar \nu_n$ toward the minimizer of the free energy, see e.g. \cite{GuillinWuZhang,GuillinMonmarcheMF}.

\subsection{Low temperature and simulated annealing}\label{sec:annealing}

In this section, instead of sampling the target measure $\pi \propto e^{-U}$, we consider the problem of finding a global minimizer of $U$. For this, we consider a simulated annealing algorithm based on the idealized HMC\new{, which illustrates both the robustness of our approach as it applies to time-inhomogeneous target distributions (whose interest goes beyond optimization, see \cite{doucet2022scorebased} and references within) and its sharpness in the low temperature regime. Notice that practical motivations to use HMC for simulated annealing instead of, say, the overdamped Langevin diffusion, are the same as for sampling at constant temperature, namely the possibility to use higher order numerical schemes, and possibly their improved sampling efficiency locally within each potential well (which is not captured in Proposition~\ref{prop:recuit} below since, in the theoretical limit $\beta\rightarrow \infty$, the overwhelming issue is the transition across energy barriers, which are not reduced when using local non-reversible dynamics).  }

\new{The (idealized) HMC-based simulated annealing algorithm is} the time-inhomogeneous Markov chain $(X_n,V_n)_{n\in\N}$ defined as follows. Let $(\beta_n)_{n\in\N}$, $(t_n)_{n\in\N}$  (resp. $(\eta_n)_{n\in\N}$) be two sequences on $\R_+^*$ (resp. on $[0,1)$).   Then, the transition from $(X_n,V_n)$ to $(X_{n+1},V_{n+1})$ is a step of the idealized HMC chain with integration time $t_n$ and damping parameter $\eta_n $ as defined in Section~\ref{sec:MainResults}, except that the potential $U$ is replaced by  $\beta_n U$. We call $(\beta_n)_{n\in\N}$ the cooling schedule of the algorithm, $\beta_n$ being the inverse temperature at the $n^{th}$ step.

Similarly to the sampling problem associated to Theorem~\ref{thm:dissipation_modified}, we assume that $U$ is quadratic at infinity, and more precisely:

\begin{assumption}\label{hyp:recuit}
The potential $U\in\mathcal C^2(\R^d)$ is such that $\min U = 0$ and there exist $L_0,C_0>0$ such that for all $x\in\R^d$,
\begin{equation}\label{eq:cond_Lyap_LSI}
|\na^2 U(x)|\leqslant L_0,\qquad \frac{1}{C_0}|x| - C_0 \leqslant |\na U(x)|\leqslant C_0|x|+C_0,\qquad U(x) \geqslant \frac{1}{C_0}|x|^2 - C_0\,. 
\end{equation}
\end{assumption}

In particular, under Assumption~\ref{hyp:recuit},  $e^{-\beta U}$ is integrable for all $\beta>0$. Denote by $\pi_\beta$ the corresponding probability measure and $\mu_\beta = \pi_\beta\otimes\mathcal N(0,I_d)$. Moreover,  \eqref{eq:cond_Lyap_LSI} implies that $U$ goes to infinity at infinity and that $\na U \neq 0$ outside  a compact \new{set}, and thus that the critical depth of the potential, defined by
\begin{equation}\label{eq:c*}
c_* := \inf_{\Gamma} \left\{\sup_{s\in[0,1]} U\po \Gamma(s)\pf -  U\po \Gamma(0)\pf\right\} 
\end{equation}
where the infimum runs over all continuous path\new{s} $\Gamma:[0,1]\rightarrow \R^d$ with $\Gamma(0)$ a  local minimum of $U$ and $\Gamma(1)$ a global minimum, is necessarily finite.

We write $\nu_n$ the law of $(X_n,V_n)$. For simplicity, we only consider parameters of the form
\[\forall n\in\N\,,\qquad \beta_n = \beta_0  + \frac{\ln(1+ n)}{\hat c}\,,\qquad t_n = q/  \sqrt{\beta_nL_0}\,,\qquad \eta_n = 1 - \gamma t_n\,, \]
for some fixed $\beta_0,\hat c,\gamma>0$, $q\in(0,1/4)$.  \new{Let us briefly discuss these choices. The scaling of $t_n$ in terms of $\beta_n$ is dictated by Assumption~\ref{hyp:main}, since the Lipschitz constant of $\beta_n U$ is $\beta_n L_0$.  For the damping parameter, we choose here a Langevin scaling in terms of $t_n$, as $t_n\rightarrow 0$ with $n$ going to infinity, but we could similarly take $\eta_n = \eta$ fixed. Finally, the logarithmic scaling for the cooling schedule is classical in the study of the simulated annealing, as it is known for other Markov processes (e.g. \cite{Holley1} for Markov chains on finite sets, \cite{Holley} for the overdamped Langevin diffusion, \cite{M17,M32} for the underdamped Langevin diffusion) that convergence in probability to  global minima always (resp. never) occurs for cooling schedule slower (resp. faster) than logarithmic while, for logarithmic schedules, a phase transition occurs in terms of $\hat c$ at the value $\hat c= c_*$. The intuition behind this is the following. At low temperature, the probability that, during a time interval of length $1$, the process escapes from the basin of attraction of a non-global local equilibrium (which is a rare event) scales like $e^{-\beta w}$ for some $w>0$. For the process to converge to a global minimum, such an event has to occur with probability 1 which, in the spirit of the Borel-Cantelli theorem (here the events are not independent but the process is Markovian and metastable and thus the situation is similar), the question is whether $\sum_{n\in\N} e^{-\beta_n w} $ is finite or not, which is why the transition happens when $\beta_n$ scales like $\ln n$.
}

An application of  Theorem~\ref{thm:dissipation_modified} in this time-inhomogeneous case yields the following.

\begin{proposition}\label{prop:recuit}
Under Assumption~\ref{hyp:recuit}, assume moreover that $\hat c > c_*$ and that $\mathcal I(\nu_0|\mu_{\beta_0})$ is finite. Then, for all $\delta>0$, there exists $C>0$ such that for all $n\in\N$,
\[\Ent(\nu_n|\mu_{\beta_n})  \leqslant \frac{C}{n^{1-\new{(c_*/\hat c)}-\delta}}\,.\]
\end{proposition}

The proof is postponed to the end of this section. Thanks to Pinsker's inequality and the fact that for all $u,\delta>0$ there exists $C>0$ such that
\[\mathbb P_{\pi_\beta} (U(X) \geqslant u) \leqslant Ce^{-\beta (u-\delta)}\]
(see e.g. \cite[Lemma 3]{M17}), we get the following final result for the convergence of the simulated annealing algorithm based on an idealized HMC chain. 

\begin{corollary}
In the settings of Proposition~\ref{prop:recuit}, for all $\delta>0$, there exists $C>0$ such that for all $n\in\N$ and all $u>0$,
\[\mathbb P \po U(X_n) \geqslant u \pf \leqslant C \po \frac{1}{n^{(1-c_*/\hat c-\delta)/2}} + \frac{1}{n^{u/\hat c -\delta}}\pf\,. \]
\end{corollary}

The rest of this section is devoted to the proof of Proposition~\ref{prop:recuit}. The main point is to quantify the dependency in $\beta$ of the log-Sobolev constant of $\pi_\beta$, as $\beta\rightarrow \infty$ (i.e. in the low temperature regime). This behaviour is known, see e.g. \cite{MenzSchlichting}. The work \cite{MenzSchlichting} is very accurate since it gives the correct subexponential prefactor of the log-Sobolev constant, but the cost of this accuracy is some assumptions of non-degeneracy on $U$ (which has to be a Morse function and has some constraints on its  local minimizers and saddle points). Since the subexponential prefactor is negligible in the analysis of the simulated annealing algorithm, for completeness and for the convenience of the reader we recall   some known arguments to get a slightly rougher estimate without any other condition than Assumption~\ref{hyp:recuit}.

\begin{proposition}\label{prop:LSIrecuit}
Under Assumption~\ref{hyp:recuit}, for all $\beta_0>0$, there exist $C>0$ such that for all $\beta\geqslant \beta_0$, $\pi_\beta$ satisfies a LSI($C_{LS}(\beta)$) with a constant
\[C_{LS}(\beta) \leqslant C \beta^{5d-1} e^{\beta c_*}\,.\]
\end{proposition}

\begin{proof}
 Consider the generator 
\[L_\beta f(x) = -\beta \na U(x) \cdot \na f(x) + \Delta f(x)\]
of the overdamped Langevin process reversible with respect to $\pi_\beta(x) \propto e^{-\beta U}$. For $x\in\R^d$, set $W(x) = e^{\beta_0U(x)/2}$. Then, using \eqref{eq:cond_Lyap_LSI},
\begin{equation}\label{eq:WLSI}
\frac{L_\beta W(x)}{W(x)} = - \po \beta-\frac{\beta_0}2\pf \beta_0  |\na U(x)|^2  +  \frac{\beta_0^2}{4}\Delta U(x) \leqslant - \frac{1}{C_0'}|x|^2 +\1_{|x|\leqslant C_0'}C_0'\,,
\end{equation}
for some $C_0'$ uniformly in $\beta\geqslant \beta_0$. From \cite{CattiauxGuillinWu}, this ensures that $\pi_\beta$ satisfies a log-Sobolev and a Poincaré inequality. Denote by $C_{LS}(\beta)$ and $C_P(\beta)$ the corresponding optimal constant. Since $W,L_0$ and $C_0'$ are independent from $\beta$, \cite[Lemma 3.19]{MenzSchlichting} shows that the second moment of $\pi_\beta$ is bounded uniformly in $\beta\geqslant \beta_0$ and then, from Theorem~\ref{thmCGW-MS},
\[C_{LS}(\beta)  \leqslant C \beta \po 1 + C_{P}(\beta)\pf \]
for some $C>0$ independent from $\beta$. Using again that $W$ and $C_0'$  are independent from $\beta$ in \eqref{eq:WLSI},  \cite[Theorem 1.4]{BBCG} yields, for any $R\geqslant C_0'$,
\[C_P(\beta) \leqslant C'(1+C_P(\beta,R)) \]
where $C'>0$ is independent from $\beta$ and $C_{P}(\beta,R)$ is the Poincaré constant of the restriction of $\pi_\beta$ to the ball centered at the origin with radius $R$. Take $R$ large enough so that $\{x\in \R^d \text{ s.t. } U(x) \leqslant
  \sup\{U(y),\ y\in\R^d,\na U(y) = 0\} + c_* +1\}$ is included in this ball. In particular, a continuous path $\gamma:[0,1]\rightarrow \R^d$ starting at a minimum of $U$ and with $\sup \gamma - \gamma(0) < c_*+1$ never leaves this ball, which means the critical depth of $U$ and of its restriction on this ball are the same. From  \cite[Theorem 1.14 and Remark 1.16]{Holley}, there exist $C''>0$ independent from $\beta$ so that 
\[C_P(\beta,R') \leqslant C'' \beta^{5d-2} e^{\beta c_*}\,,\]
which concludes.
\end{proof}

\begin{remark}
The polynomial pre-factor of Proposition~\ref{prop:LSIrecuit} is not sharp but the term $e^{\beta c_*}$ is optimal, as \cite[Theorem 1.14]{Holley} also gives a similar lower bound on the local Poincaré inequality.
\end{remark}

Since we want to apply Theorem~\ref{thm:dissipation_modified} with a potential $\beta U$, hence with $L= \beta L_0$, we consider 
\[\mathcal L_\beta\po \nu \pf \ = \ \Ent(\nu|\mu_\beta) + a(\beta)\int_{\R^{2d}} \frac{|\sqrt{\beta L_0} \na_x h + \na_v h|^2 }{h}\dd \mu_\beta \quad\text{with}\quad a(\beta)= \frac{\gamma \sqrt{\beta L_0}}{27 \beta L_0+12\gamma^2} \,.\]
The dependency in $\beta$ of this modified entropy is addressed as follows.

\begin{proposition}\label{prop:recuitLbetabeta'}
Under Assumption~\ref{hyp:recuit}, for all $\beta_0>0$, there exist $C>0$ such that for all $\beta'>\beta\geqslant \beta_0$ and all $\nu\in\mathcal P(\R^d)$,
\[\mathcal L_{\beta'}(\nu) \leqslant \po 1+ \varepsilon \pf \mathcal L_\beta(\nu) + \varepsilon \qquad \text{with}\qquad \varepsilon = C\beta (\beta'-\beta) \po 1+ (\beta'-\beta)^2\pf \,.\]
\end{proposition}

\begin{proof}
The functions $\beta\rightarrow a(\beta)\beta^{k/2}$ for $k=0,1,2$ are globally Lipschitz on $[\beta_0,\infty)$ and thus there exists $C>0$ such that
\[a(\beta')\int_{\R^{2d}} \frac{|\sqrt{\beta' L_0} \na_x h + \na_v h|^2 }{h}\dd \mu_{\beta'} \leqslant  \po 1 + C(\beta'-\beta)\pf a(\beta)\int_{\R^{2d}} \frac{|\sqrt{\beta L_0} \na_x h + \na_v h|^2 }{h}\dd \mu_{\beta'}\] 
uniformly in $\beta'> \beta\geqslant \beta_0$.

Denoting by $Z_\beta = \int_{\R^d} e^{-\beta U}$ the normalization constant of $\pi_\beta$,
\[\Ent\po \nu | \mu_{\beta'} \pf 
=
\Ent\po \nu | \mu_{\beta} \pf + \po \beta'-\beta\pf\int_{\R^d} U(x)\nu(\dd x,\dd v) + \ln \frac{Z_{\beta'}}{Z_\beta}\,.\] 
Since $U\geqslant 0$, $Z_{\beta'}\leqslant Z_\beta$, and using \eqref{eq:cond_Lyap_LSI} we bound
\[\Ent\po \nu | \mu_{\beta'} \pf\leqslant  \Ent\po \nu | \mu_{\beta} \pf + \po \beta'-\beta\pf\int_{\R^d} \po 2 L_0|x|^2 + C_0'\pf \nu(\dd x,\dd v)  \]
for some $C_0'$ independent from $\beta\geqslant\beta_0$. For a matrix $A$,  using \eqref{eq:cond_Lyap_LSI}  again,
\begin{multline*}
\int_{\R^d} \left|A \na \ln \po \frac{\nu}{\mu_{\beta'}}\pf \right|^2 \nu \ = \  \int_{\R^d} \left|A \na \ln \po \frac{\nu}{\mu_{\beta}}\pf + (\beta'-\beta) A\begin{pmatrix}
\na U \\ 0 \end{pmatrix}  \right|^2 \nu \\
 \leqslant  \po 1 + \beta'-\beta\pf  \int_{\R^d} \left|A \na \ln \po \frac{\nu}{\mu_{\beta}}\pf \right|^2 \nu  + (\beta'-\beta) \po 1 + \beta'-\beta \pf |A|^2 \int_{\R^d} \po 2C_0^2 |x|^2 + C_0^2\pf \nu(\dd x,\dd v)\,.
\end{multline*}
 Applying the inequality $uv \leqslant u\ln u -u + e^{v}$ for all $u\geqslant 0,v\in\R$ to $u=\nu/\pi_\beta$ and $v= \kappa |x|^2/2$ for any $\kappa>0$ yields
 \begin{equation}\label{eq:momententropie}
\kappa \int_{\R^d} |x|^2\nu(\dd x,\dd v) \leqslant  \Ent\po \nu | \mu_{\beta} \pf + \int_{\R^d} \po e^{\kappa |x|^2} - 1\pf \pi_\beta(\dd x)\,. 
 \end{equation}
Recall the Lyapunov function $W(x) = e^{\beta_0 U(x)/2}$ of the previous proof. Using that $\pi_\beta$ is invariant for $L_\beta$ and using \eqref{eq:WLSI},
\[0 = \int_{\R^d} L_\beta W  \pi_\beta \leqslant \int_{\R^d} \po - W(x) + C_0''\pf \pi_\beta(\dd x)\,, \]
for some $C_0''$ uniform in $\beta\geqslant \beta_0$, from which we get that $\int_{\R^d} W \pi_\beta$ is bounded uniformly in $\beta\geqslant \beta_0$. Thanks to the last part of \eqref{eq:cond_Lyap_LSI}, taking $\kappa = \beta_0/(2C_0)$,  we get that $\int_{\R^d} e^{\kappa|x|^2} \pi_\beta(\dd x)$ is bounded uniformly in $\beta\geqslant \beta_0$.

Gathering these different bounds concludes.

\end{proof}

\begin{remark}
In \cite{M17},  where a similar computation is conducted in the case of the Langevin process, the moments $\int_{\R^d} |x|^2 \nu(\dd x,\dd v)$, for $\nu$ the law of the process at time $t$, are controlled by Lyapunov arguments. Using \eqref{eq:momententropie} requires less effort, in particular for HMC.
\end{remark}

\begin{proof}[Proof of Proposition~\ref{prop:recuit}]
In this proof, we denote by $C$ various positive constants, independent from $n$, large enough, which varies along the computations. The condition on $\nu_0$ implies that $\mathcal L_{\beta_0}(\nu_0) <\infty$. We can apply Theorem~\ref{thm:dissipation_modified} and  Propositions~\ref{prop:LSIrecuit} and \ref{prop:recuitLbetabeta'} to get that for all $n\in\N$
\[\mathcal L_{\beta_{n+1}} (\nu_{n+1}) \leqslant \po 1+ \varepsilon_n \pf \mathcal L_{\beta_n}(\nu_{n+1}) + \varepsilon_n\qquad \text{and}\qquad  \mathcal L_{\beta_n}(\nu_{n+1}) \leqslant \frac{1}{1+\theta_n} \mathcal L_{\beta_n}(\nu_{n})\,,\]
with
\[\varepsilon_n = C \beta_n (\beta_{n+1}-\beta_n) \leqslant \frac{C \ln(2+n)}{1+n}\quad\text{and}\quad \theta_n = \frac{t_n}{C \beta_n^{5d-1} e^{\beta_n c_*} } \geqslant \frac{1}{C(1+n)^{c_*/\hat c-\delta}} \,, \]
where we can take $\delta>0$ arbitrarily small (up to changing $C$), and in particular we enforce that $\alpha := 1 -c_*/\hat c -\delta >0$. Since $\varepsilon_n \ll \theta_n \rightarrow 0$ as $n\rightarrow \infty$, then $\ln((1+\varepsilon_n)/(1+\theta_n)) \leqslant  -\theta_n/2$ for $n$ large enough and thus we get, for all $n\geqslant k\geqslant 0$,
\[\prod_{j=k}^n \frac{1+\varepsilon_j}{1+\theta_j } \leqslant C e^{-\sum_{j=k}^n \theta_j/2}\leqslant C e^{ ( k^{\alpha} - n^{\alpha})/C }\,.\]
We have obtained
\[
\Ent(\nu_n|\mu_{\beta_n}) \ \leqslant \ \mathcal L_{\beta_n}(\nu_n) \ \leqslant \ C e^{-n^{\alpha}/C} \mathcal L_{\beta_0}(\nu_0) + C \sum_{k=0}^{n-1} \frac{\ln(2+k)}{k+1} e^{( k^{\alpha} - n^{\alpha})/C}\,. \]
 For the last term, distinguishing whether $k \lessgtr n/2$ and taking  $\delta'<\alpha$ arbitrarily small, we bound
 \[
 \sum_{k=0}^{n-1} \frac{\ln(2+k)}{k+1} e^{( k^{\alpha} - n^{\alpha})/C}  \leqslant  Cn e^{-(1-(1/2)^\alpha )n^{\alpha}/C } + \frac{C}{(n/2)^{\alpha-\delta'}}\sum_{k\geqslant n/2 }^{n-1} (k+1)^{\alpha-1} e^{( k^{\alpha} - n^{\alpha})/C} 
  \leqslant  \frac{C}{n^{\alpha-\delta'}}\,,
 \]
which concludes.
\end{proof}

\subsection{Unadjusted HMC}\label{sec:unadjusted}

In this section, we briefly consider the unadjusted HMC sampler, which is the Markov chain with transition $\hat\P = \D\hat\H$ where
\[\hat H f(z) = f\po \hat \varphi_{K,\delta}(z)\pf\]
with $\hat\varphi_{K,\delta}$ a Verlet scheme of the Hamiltonian dynamics with $K\in\N$ iterations and step-size $\delta>0$, i.e. $\hat \varphi_{K,\delta} = \hat \Phi_\delta^{\circ K}$ where
\[\hat \Phi_\delta(x,v) = \po x + \delta v - \frac{\delta^2}2 \na U(x)\,,\, v-\frac{\delta}2\po \na U(x) + \na U\po x + \delta v - \frac{\delta^2}2 \na U(x) \pf \pf \pf \,,\]
corresponding to the sequence
\begin{align*}
v &\leftarrow v -\frac{\delta}2 \na U(x)\\
x &\leftarrow x + \delta v\\
v &\leftarrow v -\frac{\delta}2 \na U(x).
\end{align*}
Let us recall a basic result of numerical error analysis concerning this chain. For simplicity, we take as granted the following:
\begin{assumption}\label{hyp:momentschema}
For either $p=1$ or $p=2$, there exist $C_0>0$ such that 
 $\int_{\R^{2d}} |z|^{2p} \nu_0 \P^n(\dd z) \leqslant C_0 d^p$ for all $n\in\N$.
\end{assumption}
For $p=1$, thanks to Theorem~\ref{thm:dissipation_modified}, this holds with $C_0$ independent from $d$ provided $\mathcal L(\nu_0)$ and  $\int_{\R^d}|x|^2 \pi(\dd x) $   are of order $d$ and $C_{LS}$ is independent from $d$. Indeed,
\[\int |z|^{2} \nu_0 \P^n \leqslant 2 \int |z|^{2} \mu + 2 \mathcal W_2^2(\nu_0\P^n,\mu) \ \leqslant \  2 \int |z|^{2} \mu + 2 C_{LS} \mathcal L(\nu_0) \]
 where we used Talagrand's inequality and the fact $\mathcal L(\nu_0 \P^n) \leqslant \mathcal L(\nu_0)$. For $p=2$, under suitable assumptions of growth of $U$, it can usually be established via Lyapunov arguments, see e.g. the proofs of \cite[Lemma 30, Proposition 31]{MonmarcheSplitting} or \cite{Toappear}.
 
 \begin{proposition}\label{prop:numerique}
 Under Assumptions~\ref{hyp:main}, \ref{hyp:Frobenius} and \ref{hyp:momentschema}, if $K\delta=t$, there exist $C'>0$ (depending only on $L,L_H$ and $C_0$) such that
 \begin{equation}
 \label{eq:numerique}
  \mathcal W_p (\nu_0\P^n,\nu_0\hat \P^n ) \leqslant C'\delta^{\new{p}} nt e^{C'nt} d^{\new{p/2}}\,,
 \end{equation}
 where $p$ is given by Assumption \ref{hyp:momentschema}.
 \end{proposition}
 
 The proof is similar to \cite[Section 4.3]{MonmarcheHMCconvexe}, hence omitted. Combining this result with Theorem~\ref{thm:dissipation_modified} and using Talagrand's inequality yields
 \begin{equation}
\label{eq:crude} 
\mathcal W_p (\mu,\nu_0\hat \P^n ) \leqslant C'\delta^{\new{p}} nt e^{C'nt} d^{\new{p/2}} + \po 1+ \frac{3t }{8\max(C_{LS},1)/a +16 }\pf^{-n/2}  \sqrt{C_{LS}\mathcal L(\nu_0)}\,,
 \end{equation}
  with $a=\gamma/[14+8(\gamma+3)^2]$.   For a given tolerance $\varepsilon>0$, we measure the numerical  complexity of the algorithm in terms of number of computation of gradients of $U$ (which is the main part of the numerical cost of the algorithm) by
  \begin{equation}\label{eq:Mepsi}
M_\varepsilon = K\inf\left\{n\in\N,\ \mathcal W_p (\mu,\nu_0\hat \P^n )  \leqslant \varepsilon \new{\mathcal W_2(\pi,\delta_0)}\right\}\,,  
  \end{equation}
where we used the \new{scaling-invariant criterion} advocated by \cite{dalalyan2} \new{(see Equation (5) in Section 3 of \cite{dalalyan2})}  when using Wasserstein distances, and we multiplied by $K$ which is the number of computations of gradients per transition of the chain. \new{In terms of the dimension $d$, we consider that $\mathcal W_2(\pi,\delta_0)$ is of order $\sqrt{d}$ (which is consistent with Assumption~\ref{hyp:momentschema} and, for instance, holds for the i.i.d. case). } Considering the behaviour of $M_\varepsilon$ in terms of small $\varepsilon$ and large $d$ when $C_{LS}$, $\gamma$, $\mathcal L(\nu_0)/d$, $L$, $L_H$ and $C_0$ are fixed (independent from $d$), \new{assuming that $p=2$ in Assumption~\ref{hyp:momentschema}}, we see by taking $nt$ of order $\ln(1/\varepsilon)$ and then $\delta^{-2}$ of order  $d /\varepsilon^{q}$ for some $q>0$ that
\[M_\varepsilon = \mathcal O \po \frac{\sqrt{d}}{\varepsilon^r}\pf   \]
for some $r>0$. The dependency in $\varepsilon$ is really bad, which is related to the fact that the bound \eqref{eq:crude}  is very crude. However, in terms of the dimension, in this context (where the target measure might not be log-concave for instance), we are already able to get with Theorem~\ref{thm:dissipation_modified} coupled with a basic numerical analysis a complexity which scales as $\sqrt{d}$, as e.g. in \cite[Theorem 1.5]{ChenVempala}, \cite[Table 1]{MangoubiSmith}, \cite[Theorem 1.6]{BouRabeeSchuh} or \cite[Table 1]{MonmarcheSplitting} \new{(of course if we remove the scaling in $\mathcal W_2(\pi,\delta_0) \simeq \sqrt{d}$  in \eqref{eq:Mepsi} then the bound obtained here becomes very bad also in $d$. We emphasize that the scaling in \eqref{eq:Mepsi}  is not an artificial way to get a nice result here but, as explained in \cite{dalalyan2}, is in fact the natural relevant criterion to assess the accuracy in terms of Wasserstein distance, which are used to control the estimation of moments, which scale as $\mathcal W_2(\pi,\delta_0)$).}


\new{The point of this section was to see that Theorem~\ref{thm:dissipation_modified}, which concerns the idealized HMC, already yields with  no additional work (given the known error bounds on the Verlet integrator) a result on the practical unadjusted HMC. Of course, more work is now required to get a better dependency in $\varepsilon$ and a bound similar to \eqref{eq:crude}  but uniform in time. This is postpone to a future work \cite{Toappear}. 

\begin{remark}\label{rem:comparaison_avec_couplage}
Besides, as a last comment, notice that, by some aspects, in the non-convex case, the bound \eqref{eq:crude} is possibly already better than the bounds obtained by reflection coupling  methods (in e.g. \cite{BouRabeeSchuh,BouRabeeEberleZimmer,BouRabeeEberle} for HMC with $\eta=0$). Indeed, to fix ideas, consider the low-temperature regime where $\pi \propto e^{-\beta U_0}$, as in Section~\ref{sec:annealing}, so that we focus on the dependency in $\beta$. From Proposition~\ref{prop:LSIrecuit}, up to polynomial terms in $\beta$, we get a complexity of order $e^{\beta c_*}$ where $c_*$ is the critical height of the potential, defined by \eqref{eq:c*}. On the contrary, for instance, under the assumption that $\na U_0$ is $L$-Lipschitz and that $(x-y)\cdot (\na U_0(x)-\na U_0(y)) \geqslant m|x-y|^2$ if $|x-y|\geqslant R$ for some $m,L,R>0$ (i.e. $U_0$ is strongly convex outside a ball of radius of order $R$),  the result of \cite{BouRabeeEberleZimmer} (this is the same for similar works) gives a complexity for the $\mathcal W_1$ distance of order $e^{-3R/(2t)} = e^{-24L\beta R^2}$ (here we take an integration time $t=1/(16 \beta L R)$ to minimize the exponential term in the bound on the contraction rate in \cite{BouRabeeEberleZimmer} for small $\beta$). Since the bound obtained with reflection coupling applies to all potentials $U_0$ with Lipschitz gradient and strongly convex outside a ball, obviously the result cannot be better than the worse convergence rate among all potentials satisfying the conditions for fixed $L,m,R$, and thus this bound has to scale at least like $e^{\beta \tilde c}$ where $\tilde c$ is the worse critical height among all those. Now, for a specific $U_0$, the upper bound $24 LR^2$ might be possibly off (for instance, take a potential $U_0$ in dimension $2$ with the shape of a Mexican hat: then $c_*=0$, i.e. \eqref{eq:crude} gives a polynomial bound in $\beta$).
\end{remark}
 }

\section*{Acknowledgements}

This work has been partially funded by the  French ANR grants EFI (ANR-17-CE40-0030) and SWIDIMS (ANR-20-CE40-0022) and by the European Research Council (ERC) under the Euro-
pean Union’s Horizon 2020 research and innovation program (grant agreement No 810367),
project EMC2. P. Monmarch\'e thanks Alain Durmus for fruitful discussions.

\bibliographystyle{plain}
\bibliography{bibliography/bibliography}

\end{document}